\documentclass[final,leqno,onefignum,onetabnum]{siamltex1213}

\usepackage{amsmath,amsfonts,amssymb}
\usepackage{url}
\usepackage[ruled,linesnumbered]{algorithm2e}

\newtheorem{thm}{Theorem}[section]
\newtheorem{lem}[thm]{Lemma}
\newtheorem{prop}[thm]{Proposition}

\newcommand{\CL}[1]{\mathcal{#1}}
\newcommand{\RM}[1]{\mathrm{#1}}
\newcommand{\BF}[1]{\mathbf{#1}}
\newcommand{\BB}[1]{\mathbb{#1}}

\newcommand{\ten}[1]{{\mathcal{#1}}} 

\title{Regularized Computation of Approximate Pseudoinverse
	of Large Matrices Using Low-Rank Tensor Train Decompositions}

\author{Namgil Lee\footnotemark[2]
	\and Andrzej Cichocki\footnotemark[3]}

\begin{document}

\maketitle



\renewcommand{\thefootnote}{\fnsymbol{footnote}}
\footnotetext[2]{Laboratory for Advanced Brain Signal Processing,
	RIKEN Brain Science Institute, Wako-shi, Saitama 3510198, Japan.
	(\email{namgil.lee@riken.jp}).}
\footnotetext[3]{Skolkovo Institute of Science and Technology (Skoltech), Moscow 143025, Russia, and 
	Systems Research Institute, Polish Academy of Sciences,
	Warsaw, Poland, and
	Laboratory for Advanced Brain Signal Processing,
	RIKEN Brain Science Institute, Wako-shi, Saitama 3510198, Japan.
	(\email{cia@brain.riken.jp}).}
\renewcommand{\thefootnote}{\arabic{footnote}}


\begin{abstract} 
We propose a new method for low-rank approximation of Moore-Penrose pseudoinverses (MPPs) of large-scale matrices using tensor networks. The computed pseudoinverses can be useful for solving or preconditioning of large-scale overdetermined or underdetermined systems of linear equations. The computation is performed efficiently and stably based on the modified alternating least squares (MALS) scheme using low-rank tensor train (TT) decompositions and tensor network contractions. The formulated large-scale optimization problem is reduced to sequential smaller-scale problems for which any standard and stable algorithms can be applied. Regularization technique is incorporated in order to alleviate ill-posedness and obtain robust low-rank approximations. Numerical simulation results illustrate that the regularized pseudoinverses of a wide class of non-square or nonsymmetric matrices admit good approximate low-rank TT representations. Moreover, we demonstrated that the computational cost of the proposed method is only logarithmic in the matrix size given that the TT-ranks of a data matrix and its approximate pseudoinverse are bounded. It is illustrated that a strongly nonsymmetric convection-diffusion problem can be efficiently solved by using the preconditioners computed by the proposed method.


\end{abstract}


\pagestyle{myheadings}
\thispagestyle{plain}
\markboth{N. LEE AND A. CICHOCKI}{PSEUDOINVERSE COMPUTATION
USING TENSOR TRAIN}



\section{Introduction}









In this paper, we consider the approximate numerical solution of
very large-scale systems of linear equations
	\begin{equation} \label{eq:linear_system}
	\BF{Ax} = \BF{b},
	\quad \BF{A}\in\BB{R}^{I\times J}, 
	\quad \BF{b}\in\BB{R}^{I}.
	\end{equation}
In the case that both the number of equations $I$ and the
number of unknowns $J$ have an exponential rate of increase,
e.g., $I=P^N$ and $J=Q^N$ with $N\geq 20$ and $P,Q\geq 2$,
standard numerical solution methods cannot be applied directly 
without exploiting structure of the matrix such as the sparsity 
due to high computational and storage costs. 
Such an exponential increase with size of the matrix is referred
to as the curse of dimensionality. 
For example, standard numerical methods for solving
high-dimensional partial differential equations often become
intractable as the dimensionality of the involved operators and
functions increases. 
We consider structured matrices with billions of
rows and columns and beyond, without any assumption that the data matrix is sufficiently sparse.



In order to break the curse of dimensionality, low-rank tensor
decomposition techniques are gaining growing attention in numerical
computing and signal processing
\cite{Bey2002,Cic2014a,Cic2014TN-b,Gra2013,Kho2012,Verv2014}.
A tensor of order $N$ is an $N$-dimensional array.
Higher-order tensors arise from various sources such as
multi-dimensional data analysis \cite{Cic2009,KolBa2009} and
high-dimensional problems in scientific computing
\cite{Gra2013,Kho2012}.
Tensor decomposition techniques transform such higher-order
tensors into low-parametric data representation formats.
Comprehensive surveys about traditional tensor decomposition methods
such as CANDECOMP/PARAFAC (CP) and Tucker formats
are provided in \cite{Cic2009,KolBa2009}. However, 
traditional tensor decomposition methods have limitations
in high-order tensor approximation and numerical computing
\cite{Gra2013,Kho2012}.
On the other hand, modern tensor decomposition methods such
as the hierarchical Tucker (HT) \cite{Gra2010,Hac2009} and
tensor train (TT) formats \cite{Ose2011,OseTyr2009}
are promising tools for breaking down the curse of dimensionality.
For instance, once large-scale matrices and vectors are reshaped into
higher-order tensors and represented approximately in TT format (e.g.,
see \cite{Kaz2012,Kho2011}), basic algebraic operations such as addition
and matrix-by-vector multiplication can be performed
with logarithmic storage and computational complexity \cite{Ose2011}. 


We mostly focus on the TT format, which is equivalent to the matrix product states (MPS) with open boundary conditions (OBC) in quantum physics \cite{USch2011}. Algorithms for solving optimization problems using TT formats, due to its simple linear structure of separation of variables (or dimensions), are often presented in simple recursive forms, please see, e.g., TT-rank truncation algorithm \cite{Ose2011}.


However, previous studies on numerical algorithms
for solving systems of linear equations based on TT formats
focus mostly on the cases of square data matrices, i.e.,
$\BF{A}\in\BB{R}^{I\times J}$ with $I=J$;
see, e.g., TT-GMRES \cite{Dol2013gmres},
alternating least squares (ALS) \cite{Holtz2012},
modified alternating least squares (MALS) \cite{Holtz2012},
density matrix renormalization group (DMRG) \cite{OseDol2012}, and
alternating minimal energy (AMEn) \cite{DolSav2014}.
In order to apply the existing algorithms to the case of
non-square or strongly nonsymmetric coefficient matrices,
the normal equation
$\BF{A}^\RM{T}\BF{Ax} = \BF{A}^\RM{T}\BF{b}$ can be considered,
in which the solution is equivalent in the sense of minimum of the linear least
squares problem. However, the matrix product $\BF{A}^\RM{T}\BF{A}$
is often extremely ill-conditioned since the singular values of $\BF{A}$ are
squared. The ill-conditioning can slow down the convergence rate and reduce accuracy of
developed algorithms. 

Our main objective is to develop a new method to compute an approximation to the Moore-Penrose pseudoinverse (MPP) of $\BF{A}$, i.e., $\BF{P}^\RM{T}\approx \BF{A}^\dagger\in\BB{R}^{J\times I},$
and solve the preconditioned system 
	\begin{equation} \label{eqn:precond:system}
	\BF{P}^\RM{T}\BF{Ax} = \BF{P}^\RM{T}\BF{b},
	\end{equation}
when $I\geq J$, or $\BF{AP}^\RM{T}\BF{y} = \BF{b}$ with the substitution $\BF{x}=\BF{P}^\RM{T}\BF{y}$ when $I\leq J$. The preconditioned matrix $\BF{P}^\RM{T}\BF{A}$ (resp. $\BF{AP}^\RM{T}$) should be square and near symmetric, and have more uniformly distributed eigenvalues possibly far away from zero to improve the convergence property of an iterative method. 

We can compute an approximate MPP $\BF{P}^\RM{T}\approx\BF{A}^\dagger\in\BB{R}^{J\times I}$ by minimizing the cost function
	\begin{equation} \label{eq:min_norm_original}
	F_\lambda (\BF{P}) = \left\| \BF{I}_J - \BF{P}^\RM{T}\BF{A}
	\right\|_\RM{F}^2
	+ \lambda \|\BF{P}\|_\RM{F}^2,
	\quad \lambda \geq 0,
	\end{equation}
assuming without loss of generality that $I\geq J$. The objective
function \eqref{eq:min_norm_original} with $\lambda=0$ has been
considered widely for the computation of preconditioners by, e.g.,
sparse approximate inverse (SPAI) preconditioners
\cite{Benzi1999,Grote1997}  and generalized approximate inverse
(GAINV) preconditioners \cite{Cui2009}. Most approximate inverse
techniques are claimed to be largely immune to the risk of
breakdown during the preconditioner construction \cite{Benzi1999}.
The general case that $\lambda \geq 0$ was considered by, e.g.,
the optimal low-rank regularized inverse matrix approximation
\cite{Chung2014inv,Chung2015laa}. The regularization term is helpful
for alleviating ill-posedness of the minimization problem 
\eqref{eq:min_norm_original} and improving the convergence property of
the proposed algorithm. The simulation results in 
Section~\ref{sec_num_simulation} illustrate that the
regularization approach is also helpful for avoiding overestimation of
TT-ranks of the pseudoinverses.

We propose a new method for computing the approximate generalized inverse $\BF{P}^\RM{T}$ in the form of a low-rank TT decomposition. Many of the sparse approximate inverse algorithms such as the SPAI preconditioners \cite{Grote1997} assume that approximate inverses are sparse, however, inverses of
sparse matrices are often not sparse. On the other hand, matrices represented in TT format do not have to be sparse, instead, they are required to have relatively small TT-ranks. It has been shown that inverses and preconditioners of several classes of important matrices such as Laplace-like operators and banded Toeplitz matrices admit approximate low-rank TT representations \cite{Braess2005,Kaz2012,Kho2009,Kho2011,OseTyrZam2011}.

Holtz, Rohwedder, and Schneider \cite{Holtz2012} and Oseledets and Dolgov \cite{OseDol2012} developed DMRG (also called as MALS) methods for solving systems of linear equations with square matrices in TT formats. The application to matrix inversion is presented in \cite{OseDol2012}. However, the matrix inversion method proposed in \cite{OseDol2012} is applicable only to square matrices. And when we want to apply this approach, the linear matrix equation $\BF{P}^\RM{T}\BF{A} = \BF{I}_J$ is converted to a larger linear least squares (LS) problem. On the other hand, the proposed method can be applied to general non-square matrices, whose pseudoinverses can be efficiently computed without a need to solve a larger linear LS problem. Finally, once an approximate pseudoinverse is computed, the preconditioned linear system can be solved by applying existing TT-based optimization algorithms such as the ones developed in \cite{DolSav2014,Ose2014,OseDol2012}. Numerical simulation results illustrate that the approximate
generalized inverses obtained by the proposed method provide usually low TT-ranks with a sufficiently small approximation error. The resulting preconditioned linear systems can be solved by existing
TT-based algorithms much faster than linear systems without preconditioning. A wide class of structured matrices in TT format are considered and demonstrated to have low TT-rank preconditioners.

The paper is organized as follows. In Section 2, we describe briefly mathematical representations for TT decomposition. In Section 3, we design a new tensor network and develop a MALS algorithm for computing approximate pseudoinverses. In Section 4, experimental results are presented to demonstrate the validity and effectiveness of the proposed method. In Section 5, discussion and concluding remarks are given.

\section{The TT Decompositions}

\subsection{Notation}
\label{sec_notation}

We will briefly describe notation for tensors and tensor operations
used in this paper. We refer to \cite{Cic2009,DeLath2000,KolBa2009}
for further details. Scalars, vectors, and matrices are denoted by 
lower-case letters ($a$, $b$, \ldots), lower-case bold letters 
($\BF{a}$, $\BF{b}$, \ldots), and upper-case bold letters
($\BF{A}$, $\BF{B}$, \ldots), respectively. $N$th-order tensors, i.e.,
$N$-way arrays (for $N\geq 3$), are denoted by calligraphic letters
($\ten{A}$, $\ten{B}$, \ldots). For a tensor 
$\ten{X}\in\BB{R}^{I_1\times I_2\times\cdots\times I_N}$,
where $I_n$ is the size of the $n$th \textit{mode}, 
the $(i_1,i_2,\ldots,i_N)$th entry of $\ten{X}$ is denoted by
$x_{i_1,i_2,\ldots,i_N}$. Mode-$n$ fibers of a tensor
$\ten{X}\in\BB{R}^{I_1\times I_2\times \cdots \times I_N}$ are column 
vectors $\BF{x}_{i_1,\ldots,i_{n-1},:,i_{n+1},\ldots,i_N}\in\BB{R}^{I_n}$
determined by fixing all the indices except for the $n$th index. 
The mode-$1$ contracted product of tensors
$\ten{A}\in\BB{R}^{I_1\times I_2\times\cdots\times I_N}$ and $\ten{B}
\in \BB{R}^{I_N\times J_2\times J_3\times\cdots\times J_M}$
is a binary operation defined by the tensor\footnote{
	We call for simplicity such operation mode-1 contraction, because 
	the mode one of a tensor $\ten{B}$ is contracted with the mode $N$ of 
	a tensor $\ten{A}$. 
}
	\begin{equation}
	\ten{A}\bullet\ten{B}\in\BB{R}^{I_1\times I_2\times\cdots\times I_{N-1}
	\times J_2\times J_3\times \cdots\times J_M}
	\end{equation}
with entries
	$$
	(\ten{A}\bullet\ten{B})_{i_1,i_2,\ldots,i_{N-1},j_2,j_3,\ldots,j_M}
	= \sum_{i_{N}=1}^{I_N}
	a_{i_1,i_2,\ldots,i_N} b_{i_N,j_2,j_3,\ldots,j_M}.
	$$
The mode-1 contracted product is a natural generalization of
the matrix-by-matrix product to higher-order tensors. 
Note that it has associativity: 
$(\ten{A}\bullet\ten{B})\bullet\ten{C} =
\ten{A}\bullet(\ten{B}\bullet\ten{C})$ for any tensors
$\ten{A}$, $\ten{B}$, and $\ten{C}$ of proper sizes.

Basic symbols for tensor are shown in Figure~\ref{Diagram_tensor}. In particular, symbols for vectors ($1$st-order tensors), matrices (2nd-order tensors), and 3rd-order tensors are illustrated in Figure~\ref{Diagram_tensor}(a), 
while the mode-1 contraction of two 3rd-order tensors is illustrated 
in Figure~\ref{Diagram_tensor}(b).

\begin{figure}
\centering
\begin{tabular}{cc}
\includegraphics[height=1.cm]{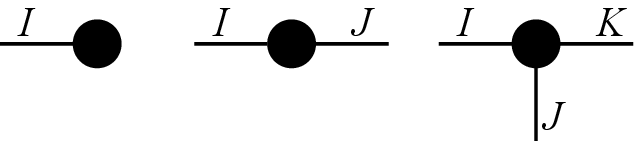} &
\includegraphics[height=1.1cm]{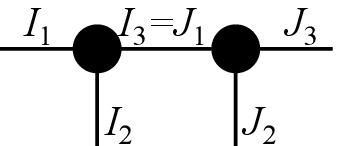} \\
(a) & (b) \\
\includegraphics[height=1.1cm]{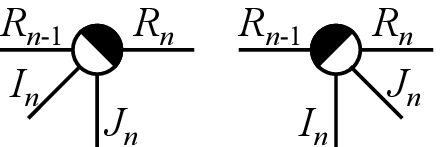} &
\includegraphics[height=1.4cm]{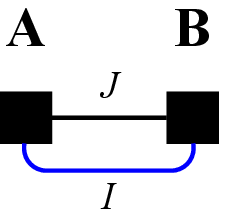}\\
(c) & (d)
\end{tabular}
\caption{\label{Diagram_tensor} Tensor network diagrams for (a)
vector, matrix, and third-order tensor, (b) mode-1 contracted
product of two third-order tensors, (c) left-orthogonalized
and right-orthogonalized fourth-order tensors, and (d)
trace$(\BF{AB})$ for matrices $\BF{A}\in\BB{R}^{I\times J}$ and
$\BF{B}\in\BB{R}^{J\times I}$.
}
\end{figure}

Indices $i_1,i_2,\ldots,i_N$, which run through $i_n=1,2,\ldots,I_n$,  
can be grouped in a single multi-index 
$\overline{i_1i_2\cdots i_N}$.\footnote{
	The multi-index can be defined by either the \textit{big-endian} 
	$\overline{i_1i_2\cdots i_N} = i_N + (i_{N-1}-1) I_N + \cdots 	+ (i_1-1)I_2I_3\cdots I_N$ 
	or the \textit{little-endian} 
	$\overline{i_1i_2\cdots i_N} = i_1 + (i_2-1) I_1 + \cdots + (i_N-1)I_1I_2\cdots I_{N-1}$. 
	In this paper, we use the little-endian convention unless otherwise mentioned. 
	}
The following two notations are defined based on the multi-index. 
\begin{definition}
For a fixed $n\in\{1,2,\ldots,N\}$, the $n$th canonical matricization of a tensor $\ten{X}\in\BB{R}^{I_1\times \cdots \times I_N}$ is defined by the matrix \cite{Holtz2011}
	\begin{equation}
	\BF{X}_{\{n\}}
	\in\BB{R}^{I_{1}I_{2}\cdots I_{n} \times
			   I_{n+1} \cdots I_{N}},
	\end{equation}
with entries 
	$$
	(\BF{X}_{\{n\}})_{\overline{i_1i_2\cdots i_n},\overline{i_{n+1}\cdots i_N}} = x_{i_1,i_2,\ldots,i_N}. 
	$$
\end{definition}
\begin{definition}
For a fixed $n\in\{1,2,\ldots,N\}$, the mode-$n$ matricization of a tensor $\ten{X}\in\BB{R}^{I_1\times \cdots \times I_N}$ is defined by the matrix \cite{KolBa2009}
	\begin{equation} \label{expr_mode_matricization}
	\BF{X}_{(n)}
	\in\BB{R}^{I_n\times I_1I_2\cdots I_{n-1}I_{n+1}\cdots I_N}, 
	\end{equation}
with entries  
	$$
	(\BF{X}_{(n)})_{i_n, \overline{i_1\cdots i_{n-1}i_{n+1}\cdots i_N}}
	= x_{i_1,i_2,\cdots,i_N}. 
	$$
\end{definition}	
Note that the columns of the mode-$n$ matricization $\BF{X}_{(n)}$ are the mode-$n$ fibers of $\ten{X}$. 
Moreover, the vectorization of a tensor $\ten{X}\in\BB{R}^{I_1\times \cdots \times I_N}$ is defined as the $N$th canonical matricization
	\begin{equation}
	\text{vec}(\ten{X})
	\equiv
	\BF{x}_{\{N\}} 
	\in\BB{R}^{I_1I_2\cdots I_N}, 
	\end{equation}	
whose entries are $(\text{vec}(\ten{X}))_{\overline{i_1i_2\cdots i_N}}
= x_{i_1,i_2,\ldots,i_N}$. 


\subsection{The HT Format}
The hierarchical Tucker (HT) and tensor train (TT) formats are low-parametric representations for vectors, matrices, and higher-order tensors. 
The HT format, introduced by Hackbusch and his coworkers \cite{Hac2012,Hac2009}, can be considered as a more general model than the TT format, but most techniques for TT format can be extended to HT format and the generalization is often straightfoward, e.g., see \cite{KresTob2011}. 

Let ${V}={V}^{(1)}\otimes \cdots\otimes {V}^{(N)}$
be a tensor product space, where ${V}^{(n)}$ $(1\leq n\leq N)$
are Hilbert spaces. For example, we can consider that
${V}^{(n)}=\BB{R}^{I_n}$ with the innerproduct defined by
$\langle\BF{v},\BF{w}\rangle = \BF{v}^\RM{T}\BF{w}$, 
or ${V}^{(n)}=\BB{R}^{I_n\times J_n}$ with
$\langle\BF{V},\BF{W}\rangle = \text{trace}(\BF{V}^\RM{T}\BF{W})$. 
For $N>1$, a binary tree $T_N$ is called a dimension tree if 
	\begin{enumerate}
	\item[(a)] all nodes $t\in T_N$ are non-empty subsets of $\{1,\ldots,N\}$, 
	\item[(b)] the set $t_{root}=\{1,\ldots,N\}$ is the root node of $T_N$, and 
	\item[(c)] each node $t\in T_N$ with $|t|\geq 2$ has two children $t_1,t_2\in T_D$ such that $t$ is a disjoint union $t=t_1\cup t_2$. 
	\end{enumerate}
A dimension tree determines recursive partitioning of the modes $\{1,\ldots,N\}$, e.g., see Figures~\ref{fig_dimension_tree}(a) and (b). We denote the set of all leaf nodes by $L\subset T_N$ and the set of children of a non-leaf node $t$ by $S^{(t)} \subset T_N$. 

Let $(R_t)_{t\in T_N}$ be positive integers with $R_{t_{root}} = 1$. An element $\BF{x}\in {V}$ is an HT tensor with $T_N$-rank bounded by $(R_t)_{t\in T_N}$ if 
	\begin{enumerate}
	\item[(a)] there exists a finite dimensional subspace $U^{(t)} = \textrm{span}\{\BF{u}^{(t)}_{r_t}: 1\leq r_t\leq R_t \} \subset {V}^{(t)}$ for each leaf node $t\in L$, and 
	\item[(b)] there exists a coefficient tensor $\ten{G}^{(t)} \in\BB{R}^{R_{t_1}\times R_{t_2}\times R_t}$ for each non-leaf node $t\in T_N\backslash L$, such that 
	\item[(c)] intermediate vectors $\BF{u}^{(t)}_{r_t} \in \bigotimes_{n\in t} {V}^{(n)}$ $(1\leq r_t\leq R_t)$ for each non-leaf node $t\in T_N\backslash L$ are defined recursively by
	\begin{equation} \label{expr_abstract_HT}
	\BF{u}^{(t)}_{r_t}  = 
	\sum_{r_{t_1}=1}^{R_{t_1}}
	\sum_{r_{t_2}=1}^{R_{t_2}}
	g^{(t)}_{r_{t_1},r_{t_2},r_t}
	\BF{u}^{(t_1)}_{r_{t_1}}
	\otimes
	\BF{u}^{(t_2)}_{r_{t_2}}, 
	\qquad 
	\{t_1,t_2\} = S^{(t)}, 
	\end{equation}
and $\BF{x}$ is represented by $\BF{x} = \BF{u}^{(t_{root})}_1$. 
	\end{enumerate}
For example, if the tree $T_N$ is the binary tree as illustrated in Figure~\ref{fig_dimension_tree}(a), then a corresponding HT tensor $\BF{x}\in V$ can be written as 
	$$
	\BF{x} = \BF{u}^{(\{1,2,3,4\})}_1 = \sum\cdots\sum 
	g^{(\{1,2,3,4\})}_{r_{12},r_{34},r_{1234}} g^{(\{1,2\})}_{r_{1},r_{2},r_{12}} 
	g^{(\{3,4\})}_{r_{3},r_{4},r_{34}} 
	\BF{u}^{(1)}_{r_1}\otimes \BF{u}^{(2)}_{r_2}\otimes \BF{u}^{(3)}_{r_3}\otimes \BF{u}^{(4)}_{r_4}
	.
	$$ 
This rather abstract representation in \eqref{expr_abstract_HT} can be immediately applied to the cases when ${V}^{(n)}=\BB{R}^{I_n}$ or ${V}^{(n)}=\BB{R}^{I_n\times J_n}$. 

\begin{figure}
\centering
\begin{tabular}{cc}
\includegraphics[height=2.3cm]{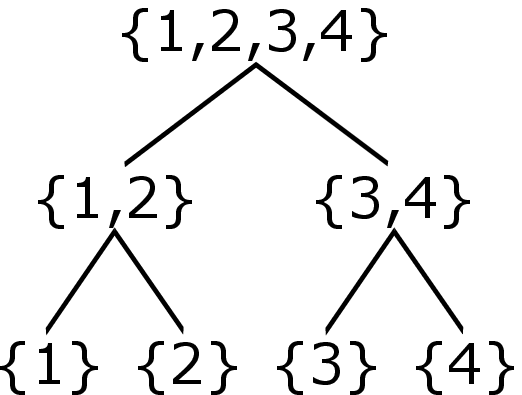}&
\includegraphics[height=2.3cm]{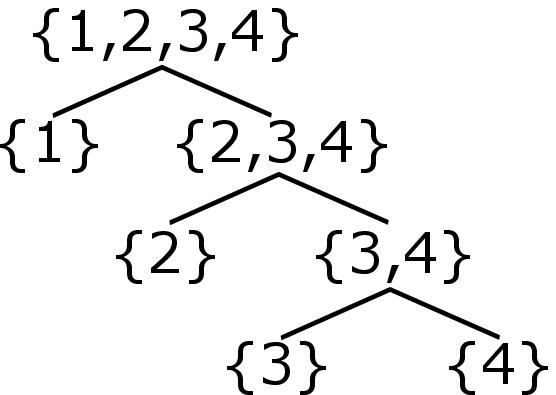}\\
(a) & (b) 
\end{tabular}
\caption{\label{fig_dimension_tree}
	Illustration of two typical examples of dimension trees
	for HT tensors of order $N=4$. 
	(a) A balanced tree structure which can generate an HT format
	and (b) an unbalanced tree structure which can generate 
	a TT format. 
}
\end{figure}

\subsection{The TT Format}

The TT format, introduced in scientific computing by Oseledets and his coworkers \cite{Ose2011,Ose2014,OseDol2012,OseTyr2009,OseTyrZam2011}, can be considered as a special case of the HT format when the dimension tree has a linear structure, such as the example in Figure~\ref{fig_dimension_tree}(b). The TT representations for large scale vectors, matrices, and higher-order tensors can also be derived from the recursive definition of the HT format in 
\eqref{expr_abstract_HT}. 

If we suppose that ${V}^{(n)} = \BB{R}^{I_n}$, then a large scale vector $\BF{x} \in \BB{R}^{I_1\cdots I_N}\equiv {V}$ can be represented in TT format by sums of Kronecker products
	\begin{equation} \label{TTvectorKron}
	\BF{x} = \sum_{r_1=1}^{R_1}
	\sum_{r_2=1}^{R_2}\cdots \sum_{r_{N-1}=1}^{R_{N-1}}
	\BF{x}^{(1)}_{1,:,r_1} \otimes
	\BF{x}^{(2)}_{r_1,:,r_2}\otimes\cdots \otimes
	\BF{x}^{(N)}_{r_{N-1},:,1},
	\end{equation}
where $\BF{x}^{(n)}_{r_{n-1},:,r_n} \in\BB{R}^{I_n}$ are mode-2 fibers of 3rd-order tensors  $\ten{X}^{(n)}\in\BB{R}^{R_{n-1}\times I_n\times R_n}$. The 3rd-order tensors $\ten{X}^{(n)}$ are called TT-cores, $R_1,R_2,\ldots,R_{N-1}$ are called TT-ranks, and we define $R_0=R_N=1$. We assume for convenience that the first and the last TT-cores (matrices) are also 3rd-order tensors of size $1\times I_1\times R_1$ and $R_{N-1}\times I_N\times 1$, respectively. We call the TT representations in \eqref{TTvectorKron} (for large-scale vectors) as the vector TT format, which is equivalent to the MPS with open boundary conditions (OBC) (see Figure~\ref{Diagram_TensorTrains}(a)(top).)

A higher-order tensor $\ten{X}\in\BB{R}^{I_1\times \cdots\times I_N}$ is said to be represented in TT format if its vectorization $\text{vec}(\ten{X}) \in\BB{R}^{I_1\cdots I_N}$ is represented in TT format \eqref{TTvectorKron}, i.e., $\BF{x} = \text{vec}(\ten{X})$. In this case, the tensor $\ten{X}$ can be expressed by the mode-1 contracted products of 3rd-order TT-cores (see Figure~\ref{Diagram_TensorTrains}(a)(top))
	\begin{equation} \label{TTformat}
	\ten{X} = \ten{X}^{(1)}\bullet\ten{X}^{(2)}\bullet\cdots
	\bullet\ten{X}^{(N)}, 
	\end{equation}
where $\ten{X}^{(n)}\in\BB{R}^{R_{n-1}\times I_n\times R_n}$ are the 3rd-order TT-cores. 

Note that TT decompositions \eqref{TTformat} of a tensor $\ten{X}$ may not be unique, and the TT-ranks $R_1,\ldots,R_{N-1}$ depend on a specific decomposition. From \eqref{TTformat}, we have $rank(\BF{X}_{\{n\}}) \leq R_n$, $n=1,\ldots,N-1$. However, if a TT decomposition of a tensor $\ten{X}$ is \textit{minimal}, i.e., all TT-cores have full left and right rank as $rank(\BF{X}^{(n)}_{(3)})=R_{n}$ and $rank(\BF{X}^{(n)}_{(1)})=R_{n-1}$, then the TT-ranks of minimal TT decompositions are unique and satisfy $rank(\BF{X}_{\{n\}}) = R_n$, $n=1,\ldots,N-1$. The TT-ranks of a tensor $\ten{X}$ is defined as the TT-ranks of a minimal TT decomposition. See \cite{Hac2012,Holtz2011} for more details. 

The storage cost can be significantly reduced if large-scale vectors and matrices can be approximately represented in TT formats with relatively small TT-ranks. For example, the storage complexity for a vector $\BF{x}\in\BB{R}^{I_1 \cdots I_N}$ represented in TT format is $\CL{O}(NQR^2)$ with $Q=\max_n(I_n)$ and $R=\max_n(R_n)$ \cite{Ose2011}.

On the other hand, if we suppose that 
${V}^{(n)} = \BB{R}^{I_n\times J_n}$, 
then a large scale matrix $\BF{A} \in \BB{R}^{I_1\cdots I_N 
\times J_1\cdots J_N} \equiv {V}$ can be represented 
in TT format as sums of Kronecker products 
	\begin{equation} \label{TTmatrixKron}
	\BF{A} = \sum_{r^A_1=1}^{R^A_1}
	\sum_{r^A_2=1}^{R^A_2}\cdots \sum_{r^A_{N-1}=1}^{R^A_{N-1}}
	\BF{A}^{(1)}_{1,:,:,r^A_1} \otimes
	\BF{A}^{(2)}_{r^A_1,:,:,r^A_2}\otimes\cdots \otimes
	\BF{A}^{(N)}_{r^A_{N-1},:,:,1},
	\end{equation}
where $\BF{A}^{(n)}_{r^A_{n-1},:,:,r^A_n} \in\BB{R}^{I_n\times J_n}$ are the slice matrices of 4th-order core tensors $\ten{A}^{(n)}\in $ \\ $ \BB{R}^{R^A_{n-1}\times I_n\times J_n\times R^A_n}$, and $R^A_1,R^A_2,\ldots,R^A_{N-1}$ are TT-ranks with $R^A_0=R^A_N=1$.
The TT format for large-scale matrices in \eqref{TTmatrixKron} is equivalent to the matrix product operators
(MPO) in quantum physics \cite{USch2011}, and we refer to it as the matrix TT format (see Figure~\ref{Diagram_TensorTrains}(b)(top)). 

\begin{figure}
\centering
\begin{tabular}{cc}
\includegraphics[width=6cm]{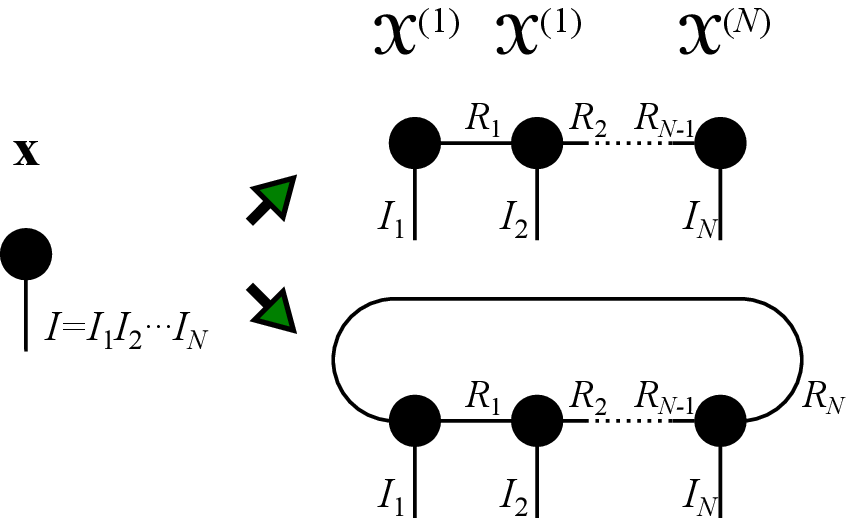} &
\includegraphics[width=6cm]{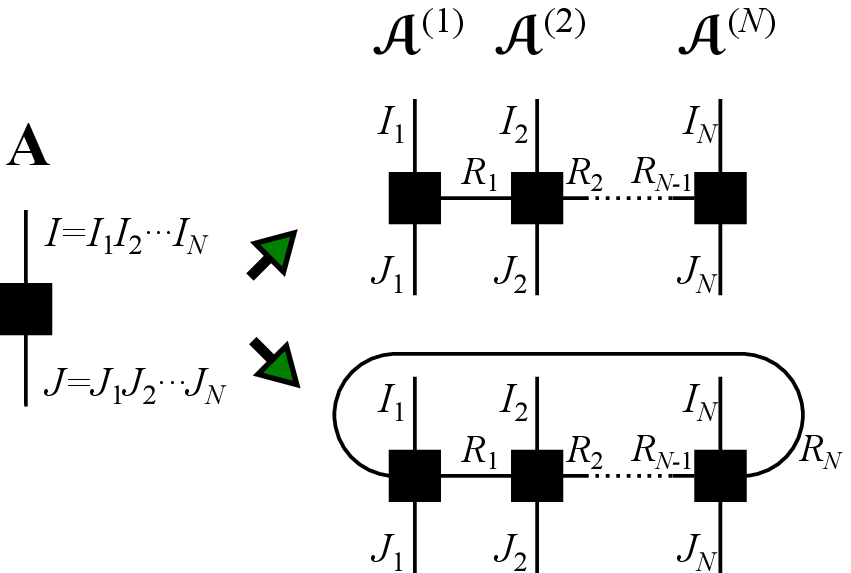}\\
(a) & (b)
\end{tabular}
\caption{\label{Diagram_TensorTrains} Tensor network diagrams for
TT decomposition of large-scale vectors and matrices.
(a) Large-scale vector in vector TT format, which is equivalent to
either the MPS with open boundary conditions (OBC) (top)
or the MPS with periodic boundary conditions (PBC) (bottom),
and (b) large-scale matrix in matrix TT format,
which is equivalent to  either the MPO with OBC (top), or
MPO with PBC (bottom).
}
\end{figure}

\subsection{Extraction of Core Tensor for Matrix TT Format}

For practical and theoretical purposes, we will introduce alternative representations for large-scale matrices based on the matrix TT format or MPO, which will be used for describing the proposed algorithm.

The vectorization, which is an operation on matrices, can be extended to matrices represented in matrix TT format as follows. Let $I=I_1I_2\cdots I_N$ and $J=J_1J_2\cdots J_N$. For a matrix $\BF{A} \in\BB{R}^{I\times J}$ in matrix TT format \eqref{TTmatrixKron} with TT-cores $\ten{A}^{(n)}\in \BB{R}^{R^A_{n-1}\times I_n\times J_n \times R^A_n}$, the extended vectorization of $\BF{A}$ can be defined by the vector in TT format, with slight abuse of our notation, 
	\begin{equation} \label{TT_vectorization}
	{\text{vec}}(\BF{A}) 
	= \sum_{r^A_1=1}^{R^A_1}\cdots
	\sum_{r^A_{N-1}=1}^{R^A_{N-1}}
	\text{vec}(\BF{A}^{(1)}_{1,:,:,r^A_1}) \otimes
	\cdots \otimes
	\text{vec}(\BF{A}^{(N)}_{r^A_{N-1},:,:,1})
	\in\BB{R}^{IJ}, 
	\end{equation}
i.e., each of the 4th-order TT-cores $\ten{A}^{(n)}$ is reshaped into the 3rd-order TT-core
of size ${R_{n-1}\times I_nJ_n \times R_n}$. 

Let 
	\begin{equation} \label{vecp_from_matP}
	{\BF{p}} = {\text{vec}}(\BF{P})\in\BB{R}^{IJ}
	\end{equation}
denote the extended vectorization \eqref{TT_vectorization} of a large-scale matrix $\BF{P}\in\BB{R}^{I\times J}$ in matrix TT format whose 4th-order TT-cores can be reshaped into 3rd-order cores 
	\begin{equation}
	{\CL{P}}^{(n)}\in\BB{R}^{R_{n-1}\times K_n\times R_n}, 
	\quad 	K_n\equiv I_nJ_n, 
	\quad 	\forall n,
	\end{equation}
for the corresponding vector $\BF{p}$. 
From the expression \eqref{TTformat}, the $N$th-order tensor 
${\ten{P}}\in\BB{R}^{K_1\times \cdots\times K_N}$
determined by ${\BF{p}}=\text{vec}({\ten{P}})$ 
can be written as the mode-1 contracted products (see Figure~\ref{Fig:TTmatrix_leftright})
	\begin{equation*} 
	{\ten{P}} = {\ten{P}}^{(1)}\bullet {\ten{P}}^{(2)}\bullet\cdots
		\bullet {\ten{P}}^{(N)}
	\in\BB{R}^{K_1\times K_2\times
	\cdots \times K_N}.
	\end{equation*}
For $n =1,2,\ldots,N$, the mode-1 contracted products of left
TT-cores and right TT-cores are respectively denoted by
	\begin{equation*}
	\begin{split}
	{\ten{P}}^{<n}
	&= {\ten{P}}^{(1)}\bullet {\ten{P}}^{(2)} \bullet
	 \cdots\bullet {\ten{P}}^{(n-1)}
	\in\BB{R}^{K_1\times\cdots\times K_{n-1}
		\times R_{n-1}},
	\\
	{\ten{P}}^{>n}
	& = {\ten{P}}^{(n+1)}\bullet {\ten{P}}^{(n+2)}
	\bullet\cdots\bullet {\ten{P}}^{(N)}
	\in\BB{R}^{R_n\times K_{n+1}\times\cdots
	\times K_N},
	\end{split}
	\end{equation*}
and we define ${\ten{P}}^{<1}={\ten{P}}^{>N}=1$.
The $N$th-order tensor 
${\ten{P}} \in\BB{R}^{K_1\times\cdots\times K_N}$ 
can be rewritten by
	\begin{equation} \label{matrixTTforTensorW}
	{\ten{P}} = 
	{\ten{P}}^{<n}\bullet {\ten{P}}^{(n,n+1)}\bullet 
	{\ten{P}}^{>n+1},
	\quad n=1,2,\ldots,N-1,
	\end{equation}
where
	$$
	{\ten{P}}^{(n,n+1)} =
	{\ten{P}}^{(n)} \bullet {\ten{P}}^{(n+1)}
	\in\BB{R}^{R_{n-1}\times K_n\times K_{n+1}
		\times R_{n+1}}
	$$
is the mode-1 contracted product of the two neighboring TT-cores. The so-called frame matrix \cite{Holtz2012,KresSteinUsh2014} for our model is defined by 
	\begin{equation}\label{def:frameMatrix}
	\BF{P}^{\neq } =
	\left( {\BF{P}}^{>n+1}_{(1)} \right)^\RM{T}
	\otimes \BF{I}_{K_{n+1}}
	\otimes \BF{I}_{K_n}
	\otimes \left({\BF{P}}^{<n}_{(n)} \right)^\RM{T}
	\in\BB{R}^{IJ \times R_{n-1}K_n K_{n+1} R_{n+1}},
	\end{equation}
where  ${\BF{P}}^{<n}_{(n)} \in \BB{R}^{R_{n-1}\times K_1 
\cdots K_{n-1}}$ and 
${\BF{P}}^{>n+1}_{(1)} \in\BB{R}^{R_{n+1}\times 
K_{n+2} \cdots K_N}$ 
are the mode-$n$ and mode-1 matricizations of the tensors
${\ten{P}}^{<n}$ and ${\ten{P}}^{>n+1}$,
respectively.

\begin{figure}
\centering
\includegraphics[height=2.5cm]{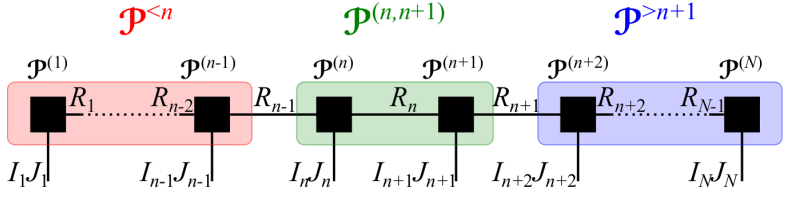}
\caption{\label{Fig:TTmatrix_leftright}The TT format for
the $N$th-order tensor ${\ten{P}}$ in \eqref{matrixTTforTensorW}.
The TT-cores are grouped into three sets and their contractions
are denoted by ${\ten{P}}^{<n}$, ${\ten{P}}^{(n,n+1)}$, and
${\ten{P}}^{>n+1}$. }
\end{figure}

From the expressions \eqref{matrixTTforTensorW} and \eqref{def:frameMatrix},  we can derive an expression for the vector ${\BF{p}} = \text{vec}({\ten{P}})$ \eqref{vecp_from_matP}
as a product of the frame matrix and a local vector $\BF{p}_n$:
	\begin{equation} \label{matrixTTforMatricization}
	{\BF{p}}
	 = \BF{P}^{\neq} \BF{p}_n
	\in\BB{R}^{IJ},
	\quad n=1,2,\ldots,N-1,
	\end{equation}
where 
	$$
	\BF{p}_n = \text{vec}({\ten{P}}^{(n,n+1)}) 
	\in \BB{R}^{R_{n-1} K_n K_{n+1} R_{n+1} }
	$$
is the vectorization of the merged TT-core ${\ten{P}}^{(n,n+1)}$.

In order to simplify computation and improve efficiency of our algorithm, we need to orthogonalize core tensors. For this purpose, left- or right-orthogonalization of the 3rd-order TT-cores  ${\ten{P}}^{(n)} \in\BB{R}^{R_{n-1}\times K_n \times R_n}$ 
is defined as follows \cite{Holtz2011}.
\begin{definition}[Left- or right-orthogonalization \cite{Holtz2011}] \label{label:definition:orthogonalize}
A 3th-order tensor ${\ten{U}} \in\BB{R}^{I \times J \times K}$ 
is called left-orthogonalized if
its mode-3 matricization 
${\BF{U}}_{(3)} \in\BB{R}^{K \times IJ}$ 
has orthonormal rows as
	\begin{equation*}
	{\BF{U}}_{(3)} {\BF{U}}_{(3)}^\RM{T}
	= \BF{I}_{K},
	\end{equation*}
and right-orthogonalized if
its mode-1 matricization ${\BF{U}}_{(1)} \in \BB{R}^{I\times JK}$ 
has orthonormal rows as
	\begin{equation*}
	{\BF{U}}_{(1)} {\BF{U}}_{(1)}^\RM{T}
	= \BF{I}_{I}.
	\end{equation*}
\end{definition}
We note that the left- or right-orthogonalization of 4th-order TT-cores $\overline{\ten{P}}^{(n)} \in\BB{R}^{R_{n-1}\times I_n\times J_n\times R_n}$ can be defined by the left- or right-orthogonalization of the reshaped 3rd-order TT-cores ${\ten{P}}^{(n)} \in\BB{R}^{R_{n-1}\times K_n\times R_n}$, where $K_n=I_nJ_n$. 
The left- or right-orthogonalized 4th-order tensors $\overline{\ten{P}}^{(n)}
\in\BB{R}^{R_{n-1}\times I_n\times J_n\times R_n}$ are denoted by the tensor network diagram shown in Figure~\ref{Diagram_tensor}(c). 
We can show that the mode-$n$ matricization ${\BF{P}}^{<n}_{(n)}$ and the mode-1 matricization ${\BF{P}}^{>n+1}_{(1)}$ have orthonormal rows if  the TT-cores $\ten{P}^{(1)},\ldots,\ten{P}^{(n-1)}$ are
left-orthogonalized and $\ten{P}^{(n+2)}, \ldots, \ten{P}^{(N)}$ are right-orthogonalized \cite{Lee2014}. Consequently, the frame matrix $\BF{P}^{\neq }$ in \eqref{def:frameMatrix} will have orthonormal columns if the TT-cores are properly  left- or right-orthogonalized.

\section{Computation of Approximate Pseudoinverse
	Using TT Decomposition}
\label{section3:Computation}

Without loss of generality, we assume that $I\geq J$ for a large-scale
matrix $\BF{A}\in\BB{R}^{I\times J}$ with
$I=I_1I_2\cdots I_N$ and $J = J_1J_2\cdots J_N$.
We formulate the following optimization problem: for $\lambda\geq 0$,
	\begin{equation} \label{optim:global:TT}
	\begin{split}
	\underset{\BF{P}}{\min}	& \quad
	\left\| \BF{I}_J - \BF{P}^\RM{T}\BF{A} \right\|_\RM{F}^2
	+ \lambda \|\BF{P}\|_\RM{F}^2\\
	\text{s.t.}	& \quad \BF{P}\in\CL{T}_{\leq\overline{R}} \subset\BB{R}^{I\times J},
	\end{split}
	\end{equation}
where $\CL{T}_{\leq\overline{R}}$ denotes the set of TT tensors of TT-ranks bounded by rank $\overline{R} = (\overline{R}_1, \ldots,$ $\overline{R}_{N-1})$. We denote the cost function by $F_\lambda (\BF{P}) \equiv \left\| \BF{I}_J - \BF{P}^\RM{T}\BF{A} \right\|_\RM{F}^2 + \lambda \|\BF{P}\|_\RM{F}^2$. We assume that the matrix $\BF{A}\in\BB{R}^{I\times J}$ is given in matrix TT format \eqref{TTmatrixKron}.


\subsection{Modified Alternating Least Squares (MALS) Algorithm}

In the MALS scheme, for each $n\in\{1,2,\ldots,N-1\}$ at each iteration, only the $n$ and $(n+1)$th TT-cores are optimized while the other TT-cores are fixed. The large-scale optimization problem \eqref{optim:global:TT} can then be reduced to a set of  much smaller scale optimization problems as we explain below.

Note that the cost function in \eqref{optim:global:TT} can be expressed, in matrix form, as 
	\begin{equation*} 
	F_\lambda(\BF{P})
	= J + \text{trace}\left( \BF{P}^\RM{T}\BF{AA}^\RM{T}\BF{P}
	-2\BF{P}^\RM{T}\BF{A}  + \lambda \BF{P}^\RM{T}\BF{P}\right).
	\end{equation*}
Let 
	$
	{\BF{p}} = {\text{vec}}(\BF{P}) \in\BB{R}^{IJ}
	$
denote the extended vectorization \eqref{TT_vectorization} of the matrix $\BF{P} \in\BB{R}^{I\times J}$. 
From the matrix TT representation \eqref{TTmatrixKron}, we can derive that 
	\begin{equation} \label{apply_extended_operation}
	{\text{vec}}(\BF{A}^\RM{T}\BF{P})
	= (\BF{I}_J {\otimes} \BF{A}^\RM{T}) {\BF{p}}
	= (\BF{I}_J {\otimes} \BF{A})^\RM{T} {\BF{p}}, 
	\end{equation}
where 
	$
	\BF{I}_J {\otimes} \BF{A}
	\in \BB{R}^{IJ\times J^2}
	$
denote the (extended) Kronecker product defined by Kronecker products between core tensors as 
	$$
	\BF{I}_J {\otimes} \BF{A}
	= \sum_{r^A_1=1}^{R^A_1} \cdots \sum_{r^A_{N-1}=1}^{R^A_{N-1}}
	(\BF{I}_{J_1} \otimes \BF{A}^{(1)}_{1,:,:,r^A_1}) \otimes \cdots \otimes
	(\BF{I}_{J_N} \otimes \BF{A}^{(N)}_{r^A_{N-1},:,:,1}). 
	$$
Using the expressions \eqref{apply_extended_operation}, we can express $F_\lambda(\BF{P})$ in the vectorized form as 
	$$
	F_\lambda(\BF{p}) 
	= J + {\BF{p}}^\RM{T} 
	(\BF{I}_J {\otimes} \BF{A}\BF{A}^\RM{T}) {\BF{p}}  -2 {\BF{p}}^\RM{T} \BF{a}
	+ \lambda {\BF{p}}^\RM{T} {\BF{p}}, 
	$$
where $\BF{a} = \text{vec}(\BF{A}) \in\BB{R}^{IJ}$.    
From the expression \eqref{matrixTTforMatricization} and the orthogonality 
of the columns of the frame matrix $\BF{P}^{\neq}$, 
$F_\lambda(\BF{P})$ can be further simplified as 
	\begin{equation} \label{reduced_normsquare}
	F_\lambda(\BF{p}_n) 
	= J + \BF{p}_n^\RM{T} \overline{\BF{A}}_n \BF{p}_n
	-2 \BF{p}_n^\RM{T} \overline{\BF{b}}_n + \lambda 
	\BF{p}_n^\RM{T} \BF{p}_n, 
	\end{equation}
where $\overline{\BF{A}}_n$ and $\overline{\BF{b}}_n$ 
are the relatively small-scale matrices and vectors defined by
	\begin{equation}\label{localAn}
	\overline{\BF{A}}_n =
	(\BF{P}^{\neq})^\RM{T}
	(\BF{I}_J {\otimes} \BF{A}\BF{A}^\RM{T})
	\BF{P}^{\neq}
	\in \BB{R}^{R_{n-1} K_n K_{n+1} R_{n+1} 
	\times R_{n-1} K_n K_{n+1} R_{n+1}}
	\end{equation}
and
	\begin{equation} \label{localBn}
	\overline{\BF{b}}_n =
	(\BF{P}^{\neq})^\RM{T}
	{\text{vec}}(\BF{A})
	\in \BB{R}^{R_{n-1} K_n K_{n+1} R_{n+1} }, 
	\quad 
	n=1,\ldots,N-1.
	\end{equation}
Finally, we can obtain a set of reduced linked optimization problems: 
	\begin{equation} \label{optim:local:TTMALS}
	\begin{split}
	\underset{ \BF{p}_n }{\min}	
	& \quad
	\BF{p}_n^\RM{T} \overline{\BF{A}}_n \BF{p}_n
	-2 \BF{p}_n^\RM{T} \overline{\BF{b}}_n + \lambda 
	\BF{p}_n^\RM{T} \BF{p}_n
	\\
	\text{s.t.}	& \quad 
	\BF{p}_n = \text{vec}({\ten{P}}^{(n,n+1)}) \in 
	\BB{R}^{R_{n-1} K_n K_{n+1} R_{n+1}}, 
	\end{split}
	\end{equation}
for $n=1,2,\ldots,N-1$.  It should be noted that the size of the matrix $\overline{\BF{A}}_n$ can be much smaller than $\BF{AA}^\RM{T}$ under the condition that $R_{n-1}$ and $R_{n+1}$ are relatively low and bounded.

Figure~\ref{fig:GAInetworks} illustrates a tensor network diagram
representing the cost function trace($\BF{P}^\RM{T}\BF{AA}^\RM{T}\BF{P} -2\BF{P}^\RM{T}\BF{A}  + \lambda \BF{P}^\RM{T}\BF{P})$ 
both in matrix and vectorized forms.
The trace is indicated by the blue line connecting the start block with the end block
(e.g., see Figure~\ref{Diagram_tensor}(d)).
In order to solve the optimization problem \eqref{optim:global:TT},
each matrix is represented approximately in matrix TT format.
The large-scale optimization problem can then be reduced to
a smaller-scale optimization problem, where the reduced cost functions are described as 
$ \BF{p}_n^\RM{T} \overline{\BF{A}}_n \BF{p}_n
	-2 \BF{p}_n^\RM{T} \overline{\BF{b}}_n + \lambda 
	\BF{p}_n^\RM{T} \BF{p}_n$ $(n=1,2,\ldots,N-1)$. 

\begin{figure}
\centering
\includegraphics[width=13cm]{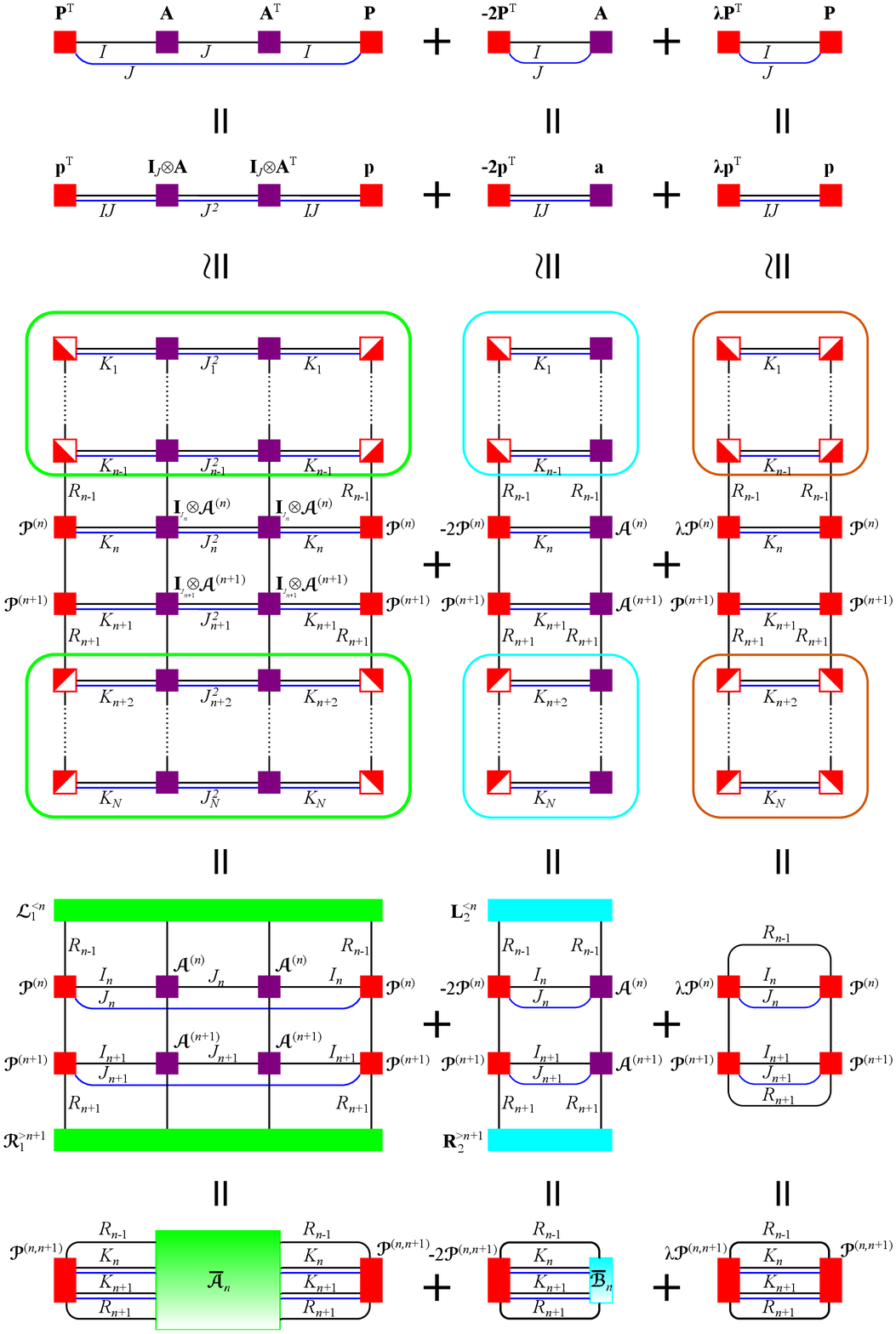}
\caption{\label{fig:GAInetworks}Conceptual
	tensor network diagrams for
	the trace$( \BF{P}^\RM{T}\BF{AA}^\RM{T}\BF{P} -2\BF{P}^\RM{T}\BF{A}
	+ \lambda \BF{P}^\RM{T}\BF{P})$ for the optimization of the
	$n$ and $(n+1)$th TT-cores, i.e., 
	$ \ten{P}^{(n,n+1)}=
	\ten{P}^{(n)}\bullet\ten{P}^{(n+1)}, 
	$ 
	in the MALS algorithm.
	The large-scale optimization problem is reduced to set of
	equivalent and much smaller optimization problems,
	which are expressed by minimization of cost functions: 
	$
	\BF{p}_n^\RM{T} \overline{\BF{A}}_n \BF{p}_n
	-2 \BF{p}_n^\RM{T} \overline{\BF{b}}_n + \lambda 
	\BF{p}_n^\RM{T} \BF{p}_n
	$
	with 
	$
	\BF{p}_n = \text{vec}({\ten{P}}^{(n,n+1)} )
	\in\BB{R}^{R_{n-1} K_n K_{n+1} R_{n+1}} 
	$ $(1\leq n \leq N-1)$.
	The left- and right-orthogonalized TT-cores of $\BF{P}$
	are indicated by the half-filled squares.
	}
\end{figure}

The relatively small-scale matrix $\overline{\BF{A}}_n$ and vector $\overline{\BF{b}}_n$ in the reduced local problem can be computed efficiently by recursive contractions of the core tensors  in the tensor network diagram in Figure~\ref{fig:GAInetworks}.  The recursive computation procedure for the tensor network contractions can be described as follows. Let
	$
	\BF{Z}_1^{(m)} \in\BB{R}^{(R_{m-1}R^A_{m-1})^2 \times (R_m R^A_m)^2}
	$
and
	$
	\BF{Z}_2^{(m)} \in\BB{R}^{R_{m-1}R^A_{m-1} \times R_mR^A_m}
	$
be matrices defined by
	\begin{equation*}
	\begin{split}
	\BF{Z}_1^{(m)} 
	& = \sum_{i_m,j_m,i_m',j_m'}
	\overline{\BF{P}}^{(m)}_{:,i_m,j_m,:} \otimes \BF{A}^{(m)}_{:,i_m,j_m',:}
	\otimes \BF{A}^{(m)}_{:,i_m',j_m',:} \otimes \overline{\BF{P}}^{(m)}_{:,i_m',j_m,:}, 
	\\
	\BF{Z}_2^{(m)} 
	& = \sum_{i_m,j_m}
	\overline{\BF{P}}^{(m)}_{:,i_m,j_m,:} \otimes \BF{A}^{(m)}_{:,i_m,j_m,:}, 
	\quad 
	m=1,\ldots,N,
	\end{split}
	\end{equation*}
where 
$\overline{\ten{P}}^{(m)}\in\BB{R}^{R_{m-1}\times I_m\times J_m\times R_m}$	
and 
$\ten{A}^{(m)}\in\BB{R}^{R^A_{m-1}\times I_m\times J_m\times R^A_m}$	
are the 4th-order TT-cores for the matrices $\BF{P},\BF{A}\in\BB{R}^{I\times J}$
in matrix TT format \eqref{TTmatrixKron}. Note that the trace terms in the 
cost function $F_\lambda(\BF{P})$ can be expressed by
	$$
	\text{trace}(\BF{P}^\RM{T}\BF{AA}^\RM{T}\BF{P}) = 
	\BF{Z}_1^{(1)} \BF{Z}_1^{(2)}  \cdots \BF{Z}_1^{(N)} , 
	\qquad
	\text{trace}(\BF{P}^\RM{T}\BF{A}) = 
	\BF{Z}_2^{(1)} \BF{Z}_2^{(2)}  \cdots \BF{Z}_2^{(N)}. 
	$$
Two 4th-order tensors $\ten{L}_1^{<m}$ and $\ten{R}_1^{>m}$
and two matrices $\BF{L}_2^{<m}$ and $\BF{R}_2^{>m}$
with sizes 
	\begin{equation*}
	\begin{split}
	\ten{L}_1^{<m}
	\in\BB{R}^{R_{m-1}\times R^A_{m-1}\times R^A_{m-1}\times R_{m-1}}, 
	&\quad
\ten{R}_1^{>m}
	\in\BB{R}^{R_{m}\times R^A_{m}\times R^A_{m}\times R_{m}},
	\\
	\BF{L}_2^{<m} \in\BB{R}^{R_{m-1}\times R^A_{m-1}},
	&\quad
	\BF{R}_2^{>m} \in\BB{R}^{R_{m}\times R^A_{m}},
	\end{split}
	\end{equation*}
are defined recursively by, for $p=1,2,$
	\begin{equation} \label{recur_L1}
	\text{vec}\left(\ten{L}_p^{<1}\right) = 1,
	\quad
	\text{vec}\left(\ten{L}_p^{<m}\right)^\RM{T} =
	\text{vec}\left(\ten{L}_p^{<m-1}\right)^\RM{T}
	\BF{Z}_p^{(m-1)},
	\
	m=2,3,\ldots,n,
	\end{equation}
	\begin{equation} \label{recur_R1}
	\text{vec}(\ten{R}_p^{>N}) =1, 
	\quad 
	\text{vec}\left(\ten{R}_p^{>m}\right) =
	\BF{Z}_p^{(m+1)}\text{vec}\left(\ten{R}_p^{>m+1}\right),
	\ 
	m = N-1, N-2, \ldots, n+1. 
	\end{equation}
The tensor $\ten{L}_1^{<m}$ defined in \eqref{recur_L1} can be 
efficiently computed by contractions of the tensors 
$\{\ten{L}_1^{<m-1}$, $\ten{P}^{(m-1)}$,
$\ten{A}^{(m-1)}$, $\ten{A}^{(m-1)}$, $\ten{P}^{(m-1)}\}$.
The tensor $\ten{R}_1^{>m}$ and matrices $\BF{L}_2^{<m}$ and
$\BF{R}_2^{>m}$ can be computed similarly. 

Finally, the matrix $\overline{\BF{A}}_n$ can be computed 
by contractions of the tensors 
$\{\ten{L}_1^{<n},$ $\ten{A}^{(n)},$ $\ten{A}^{(n)},$
$\ten{A}^{(n+1)},$ $\ten{A}^{(n+1)},$ $\ten{R}_1^{>n+1}\}$, 
as illustrated in Figure~\ref{fig:GAInetworks}. 
The vector $\overline{\BF{b}}_n$ can be computed by contractions 
of the tensors $\{\BF{L}_2^{<n}$, $\ten{A}^{(n)}$, $\ten{A}^{(n+1)}$, 
$\BF{R}_2^{>n+1}\}$.
In practice, however, the local matrix $\overline{\BF{A}}_n$ is not computed explicitly, but the matrix-by-vector multiplication $\overline{\BF{A}}_n \BF{x}$ for some $\BF{x}\in\BB{R}^{R_{n-1} K_n K_{n+1} R_{n+1}}$ is used by standard iterative methods. 
See Section~\ref{sec_comput_compl}
for more detail. 

After the solution, the resulting merged tensor 
$\ten{P}^{(n,n+1)}\equiv \ten{P}^{(n)}\bullet 
\ten{P}^{(n+1)}$
is decomposed into two separate TT-cores via 
the $\delta$-truncated SVD \cite{Ose2011}:
for the matrix $\BF{P}^{(n,n+1)}_{\{2\}}\in\BB{R}^{R_{n-1}
K_n\times K_{n+1}R_{n+1}}$, 
	\begin{equation}
	[\BF{U}_1,\BF{S}_1,\BF{V}_1] =
	SVD_\delta \left(\BF{P}^{(n,n+1)}_{\{2\}} \right),
	\end{equation}
with the subsequent updates $R_n = \min(rank(\BF{U}_1),\overline{R}_n)$,
$\BF{P}^{(n)\RM{T}}_{(3)} = \BF{U}_1,$ and
$\BF{P}^{(n+1)}_{(1)} = \BF{S}_1\BF{V}_1^\RM{T}$.
In this way, the TT-ranks can be adaptively determined 
during the iteration process, and the TT-cores can be 
left- or right-orthogonalized.

The proposed MALS algorithm\footnote{We have also developed an alternating least squares (ALS) algorithm (without merging two TT-cores), but its performance was slightly lower than the MALS algorithm presented here. }
for the computation of an approximate Moore-Penrose pseudoinverse of large-scale matrices is described in Algorithm~\ref{Alg:MALSAGINV}.

\begin{algorithm}
\caption{\label{Alg:MALSAGINV}
MALS algorithm for computing approximate pseudoinverse}
  \SetKwInOut{Input}{input}
  \SetKwInOut{Output}{output}

\Input{$\BF{A}\in\BB{R}^{I_1I_2\cdots I_N\times J_1J_2\cdots J_N}$
	in matrix TT format,
	$\epsilon>0$ (tolerance parameter),
	$\delta>0$ (truncation parameter).  }
\Output{$\BF{P}\in\BB{R}^{I_1I_2\cdots I_N\times J_1J_2\cdots J_N}$ in matrix
	TT format with TT-cores $\overline{\ten{P}}^{(n)}$ $(1\leq n\leq N)$ and TT-ranks 
	$R_1,R_2,\ldots,R_{N-1}$.} 
\BlankLine
Initialize $\BF{P}$ randomly with right-orthogonalized TT-cores
	$\ten{P}^{(3)},\ten{P}^{(4)},\ldots,\ten{P}^{(N)}$.
\\
	Set $\ten{L}_p^{<1}= 1,p=1,2.$
	Compute $\ten{R}_p^{>n},n=2,3,\ldots,N,p=1,2$ by \eqref{recur_R1}.
\\
\Repeat{a stopping criterion is met (See \eqref{stopping})}{
  \For{$n=1,2,\ldots,N-2$}{
	\tcp{Optimization}
	Optimize $\ten{P}^{(n,n+1)}$ by solving the local optimization problem \eqref{optim:local:TTMALS}.
  \\
	\tcp{Matrix Factorization by SVD}
	Compute $\delta$-truncated SVD:
	$[\BF{U}_1,\BF{S}_1,\BF{V}_1] = SVD_\delta \left( \BF{P}^{(n,n+1)}_{\{2\}} \right)$.
  \\
	Update $R_n = \min( rank(\BF{U}_1), \overline{R}_n )$.
  \\
  	Update $\overline{\ten{P}}^{(n)} = reshape(\BF{U}_1, R_{n-1}\times I_n \times J_n \times R_n)$.
  \\
	Update $\overline{\ten{P}}^{(n+1)} = reshape(\BF{S}_1\BF{V}_1^\RM{T} ,
		R_n \times I_{n+1} \times J_{n+1} \times R_{n+1})$.
  \\
	Compute $\ten{L}_p^{<n+1},p=1,2$ by \eqref{recur_L1}. 
  }
  \For{$n=N-1,N-2,\ldots,2$}{
	Perform optimization and matrix factorization similarly
  }
}
\end{algorithm}

\subsection{Properties and Practical Considerations}

\subsubsection{Existence and Uniqueness of Solution}

The proposed algorithm can be applied to general singular or non-singular structured matrices $\BF{A}\in\BB{R}^{I\times J}$ provided in TT format.
Moreover, it can be in quite straightforward way extended to the minimization of a more general objective function\footnote{Representing regularized 
	LS solution for $\widetilde{\BF{A}}\BF{X} = \BF{B}$,
	where $\widetilde{\BF{A}} = \BF{A}^\RM{T}$. }
	\begin{equation} \label{optim:general:global:TT}
	F_\lambda (\BF{P}; \BF{B}) = \left\| \BF{B}^\RM{T} - \BF{P}^\RM{T}\BF{A} \right\|_\RM{F}^2
	+ \lambda \|\BF{P}\|_\RM{F}^2,
	\end{equation}
where $\BF{B}\in\BB{R}^{J\times L}$ is a given matrix in TT format
and $\BF{P}\in\BB{R}^{I\times L}$.
For the above cost function, we can construct a similar tensor network
as shown in Figure~\ref{fig:GAInetworks}.
A minimizer of the objective function
$F_\lambda (\BF{P}; \BF{B})$, without constraints, is known as a
least squares (LS) solution. In the case that $\lambda>0$, the solution is
unique and it can be expressed by
	\begin{equation} \label{expr:sol_global_regul}
	{\BF{P}}^*_\lambda =
	\left( \BF{AA}^\RM{T} + \lambda \BF{I}_I \right)^{-1}\BF{AB},
	\quad \lambda>0.
	\end{equation}
On the other hand, if $\lambda=0$, then the solution may not be unique
and it can be expressed by
	\begin{equation} \label{expr:sol_global_zeroregul}
	{\BF{P}}^*_0 =
	\left( \BF{A}^\dagger \right)^\RM{T} \BF{B}
	+\BF{Z},
	\end{equation}
where $\BF{Z}\in\BB{R}^{I\times L}$ is any matrix satisfying
$\BF{Z}^\RM{T}\BF{A}=\BF{0}$. If $\BF{Z}=\BF{0}$, then we call it
a minimum-norm LS solution, and it is unique.
If $\BF{B}=\BF{I}_J$, then the unique solution is
equal to the (transposed) Moore-Penrose pseudoinverse. Furthermore,
it has been shown that (see \cite{Barata12})
	$$
	{\BF{P}}^*_\lambda \rightarrow
	\left( \BF{A}^\dagger \right)^\RM{T} \BF{B}
	\quad \text{as }
	\lambda\rightarrow 0.
	$$

The solution to the local optimization problem \eqref{optim:local:TTMALS},
which is also a standard LS problem, can be written
in the same way by
	\begin{equation}\label{local_solution_explicit}
	{\BF{p}}^*_n = 
	\left( \overline{\BF{A}}_n + \lambda \BF{I}_{R_{n-1} K_n K_{n+1} R_{n+1}} \right)^\dagger
	\overline{\BF{b}}_n. 
	\end{equation}
Note that the local optimal solution \eqref{local_solution_explicit} is the minimum-norm LS solution for the local optimization problem. Note that, from the expression \eqref{matrixTTforMatricization}, taking into account that the frame matrix $\BF{P}^{\neq}$ has orthonormal columns, 
	$$
	\| \BF{p}_n \|^2_\RM{F}
		= \| \BF{P} \|^2_\RM{F}.
	$$


\subsubsection{Stability and Stopping Criterion}
\label{sec:relative_error}

From \eqref{reduced_normsquare}, we note that the value of the objective function  $F_\lambda(\BF{P})$ for the global optimization problem \eqref{optim:global:TT} is exactly 
the same as the value of  the objective function  for the local optimization problem \eqref{optim:local:TTMALS} neglecting irrelevant constant $J$.
Therefore, we can conclude that the value of the 
original objective function
will monotonically decrease during the iteration process.

\begin{prop}[Monotonicity]
Let $\BF{P}_{\lambda,t}\in\BB{R}^{I\times J}$ denote the 
estimated solution at the iteration $t=0,1,2,\ldots,$ and 
$\BF{P}^*_{\lambda,t+1}\in\BB{R}^{I\times J}$
the estimated solution obtained by replacing 
the $n$ and $(n+1)$th TT-cores of $\BF{P}_{\lambda,t}$ 
with the solution $\BF{p}^*_n = \text{vec}({\ten{P}}^{(n,n+1)*})$
to the reduced optimization problem \eqref{optim:local:TTMALS}. 
Then, 
	\begin{equation}
	F_{\lambda}(\BF{P}_{\lambda,t})
	\geq F_{\lambda}( {\BF{P}}^*_{\lambda,t+1}). 
	\end{equation}
\end{prop}

Moreover, we can easily compute the value of the global objective 
function from the value of the local objective function 
at each iteration. We define 
	\begin{equation} \label{rel_residual}
	r_\lambda(\BF{P}) 
	= \frac{ \left( F_\lambda(\BF{P}) \right)^{1/2} }{ J^{1/2}}, 
	\end{equation}
and call it as the relative residual. 
Note that the minimum value of $F_\lambda(\BF{P})$ can be 
expressed in terms of the 
singular values of the matrix $\BF{A}$ as follows:
	\begin{equation} \label{minGlobalObj}
	F_{\lambda}^\text{min}
	\equiv
	\underset{\BF{P}\in\BB{R}^{I\times J}} {\min} \  F_\lambda(\BF{P})
	=
	F_\lambda( {\BF{P}}^*_\lambda)
	=
	J - \sum_{r=1}^{rank(\BF{A})}
	\frac{\sigma_r^2}{\sigma_r^2 + \lambda},
	\end{equation}
where $\sigma_r$ are the nonzero singular values of $\BF{A}$. 
So, the minimum value of the relative residual is bounded as
	\begin{equation} \label{minGlobalRelres}
	1 - \frac{rank(\BF{A})}{J} 
	\leq
	\underset{\BF{P}\in\BB{R}^{I\times J}} {\min} 
	\  r_\lambda^2(\BF{P}) 
	= \frac{ F_\lambda^\text{min} }{ J }
	\leq 1, 
	\end{equation}
where the lower bound can be attained when $\lambda=0$, 
and the upper bound when $\lambda=\infty$.

Given a tolerance parameter $\epsilon>0$, the stopping
criterion of the MALS algorithm can be executed when  
a rate of decrease of the relative residual is smaller than $\epsilon$ as
	\begin{equation}\label{stopping}
	r_\lambda^2(\BF{P}_{\lambda,t-N+2})
	- r_\lambda^2(\BF{P}_{\lambda,t})
	< \epsilon^2 \cdot r_\lambda^2(\BF{P}_{\lambda,t-N+2}).
	\end{equation}
However, due to the machine precision, the computed $r_\lambda$ value
can be not sufficiently precise if its value decreases to a very small value relative
to the norm $\|\BF{I}_J\|_\RM{F}=J^{1/2}$. In this case, it should be computed
directly using the matrices $\BF{A}$ and $\BF{P}$ represented
in TT format, rather than indirectly using
$\overline{\BF{A}}_n$, $\overline{\BF{b}}_n$, 
$\ten{P}^{(n)}$, and $\ten{P}^{(n+1)}$.

\subsubsection{Selection of Truncation Parameter}

The truncation parameter $\delta$ in the $\delta$-truncated SVD step
affects the accuracy and the convergence speed. If $\delta$ is too small,
estimated TT-ranks may increase fast and the computational cost can
become high. If $\delta$ is too large, on the other hand, the algorithm 
may not be able to achieve desired accuracy.
Hence, a $\delta$ value determines the trade-off between 
computational cost and accuracy. 
We note that Oseledets \cite{Ose2011} considered
$\delta_0 = (N-1)^{-1/2}\epsilon$ in the context of low-rank approximation, 
and Lee and Cichocki \cite{Lee2015svd} set $\delta=100\delta_0$ for the first $N-1$ iterations and then set $\delta=\delta_0$ for the rest of the
 iterations. In our numerical simulations in this paper, we used a fixed value 
 $\delta = 10^{-6}(N-1)^{-1/2}$ unless mentioned otherwise.

\subsubsection{Conditioning of Local Optimization Problems}

Note that the contracted matrix $\overline{\BF{A}}_n$ in the expression
\eqref{localAn} is symmetric and positive definite, and the frame
matrix $\BF{P}^{\neq}$ has orthonormal columns if the TT-cores
$\ten{P}^{(m)}$ $(m=1,\ldots,n-1)$ are left-orthogonalized and
$\ten{P}^{(m)}$ $(m=n+2,\ldots,N)$ are right-orthogonalized. 
Assuming the orthonormality of $\BF{P}^{\neq}$, 
we can show that (see \cite{Holtz2012})
	\begin{equation*}
	\lambda_\text{min}(\BF{AA}^\RM{T})
	\leq \lambda_\text{min}(\overline{\BF{A}}_n)
	\leq \lambda_\text{max}(\overline{\BF{A}}_n)
	\leq \lambda_\text{max}(\BF{AA}^\RM{T}),
	\end{equation*}
where $\lambda_\RM{min}(\BF{M})$ and $\lambda_\RM{max}(\BF{M})$
are the minimum and maximum of the eigenvalues of a real symmetric
matrix $\BF{M}$, respectively. Therefore, keeping the TT-cores either left- and
right-orthogonalized is important for running the MALS algorithm
efficiently, especially when we use any standard iterative method for 
solving local optimization problems.  

\subsubsection{Avoiding Breakdowns}

Some of the popular and efficient algorithms for computing preconditioners suffer
from unexpected failures (which is often referred to as breakdowns)
during the preconditioner construction step, e.g., in incomplete
factorization methods \cite{Benzi1999}.
On the other hand, the proposed MALS algorithm is free of such risks,
because we can freely choose any efficient and reliable method for
local optimizations. In the experiments, we used the Matlab
function \texttt{gmres} as a standard iterative method
for the computation of the
solution to the reduced optimization problems 
\eqref{optim:local:TTMALS}. 

\subsubsection{An Effect of Regularization to Convergence}

The regularization term not only alleviates the ill-posedness of the
optimization problem \eqref{optim:global:TT} (or more generally
\eqref{optim:general:global:TT}),
but also improves the convergence property of the proposed algorithm.

Let ${\BF{P}}^*_\lambda$ be the global solution defined by
\eqref{expr:sol_global_regul},
$\BF{P}_{\lambda,t}$ the estimate at the current
iteration ($t=0,1,2,\ldots$), and 
$\BF{R}_{\lambda,t} = \BF{AB} - (\BF{AA}^\RM{T} +
\lambda \BF{I}_I) \BF{P}_{\lambda,t}$ the residual.
Since $\BF{AB} = (\BF{AA}^\RM{T} +
\lambda \BF{I}_I) {\BF{P}}^*_\lambda$ for $\lambda>0$,
we can derive that (see also \cite{Barrett94template})
	\begin{equation} \label{eqn:convergence_error_bound}
	\left \|  {\BF{P}}^*_\lambda -
		\BF{P}_{\lambda,t}
	\right \|^2_\RM{F}
	\leq
	\left \| (\BF{AA}^\RM{T} + \lambda \BF{I}_I)^{-1} \right\|_2^2
	\left\| \BF{R}_{\lambda,t} \right\|^2_\RM{F}
	\leq
	\kappa_2^2 ( \lambda )
	\cdot
	\left\| {\BF{P}}^*_\lambda \right\|^2_\RM{F}
	\cdot
	\frac{ \left\| \BF{R}_{\lambda,t} \right\|^2_\RM{F} }{ \left\| \BF{AB} \right\|^2_\RM{F} },
	\end{equation}
where
	$
	\kappa_2 ( \lambda )
	=
	\| \BF{AA}^\RM{T} + \lambda \BF{I}_I \|_2
	\| (\BF{AA}^\RM{T} + \lambda \BF{I}_I)^{-1} \|_2
	$
is the spectral condition number of the matrix $\BF{AA}^\RM{T} + \lambda \BF{I}_I$.
It is clear that the larger the $\lambda$ value is, the smaller the
values of $\kappa_2 ( \lambda )$
and $\| {\BF{P}}^*_\lambda \|^2_\RM{F}$ become.
So, the regularization with $\lambda >0$ reduces multiplicative
factors on the right-hand side of \eqref{eqn:convergence_error_bound},
which improves the convergence speed of the current estimate
$\BF{P}_{\lambda,t}$ to the global solution ${\BF{P}}^*_\lambda$.

\subsubsection{Preconditioning of Large-Scale Systems of Linear Equations}

The estimated pseudoinverse can be helpful for preconditioning of system
of linear equations (see, \eqref{eqn:precond:system}.) By using the
estimated preconditioner, we can convert overdetermined or
underdetermined systems of linear equations into well-posed determined systems.
In addition, any nonsymmetric data matrix $\BF{A}$ can be
converted to a square symmetric matrix approximately, as stated in
the following proposition.

\begin{prop}[Symmetricity]
\label{label:thm:symmetricity}
Let $\BF{A}\in\BB{R}^{I\times J}$ denote a given matrix with $I\geq J$, 
${\BF{P}}^*_\lambda\in\BB{R}^{I\times J}$ the minimizer
of $F_\lambda(\BF{P})$ defined in \eqref{expr:sol_global_regul} and
 \eqref{expr:sol_global_zeroregul} with $\BF{B}=\BF{I}_J$, $F_\lambda^\text{min}$ 
 the minimum value defined in \eqref{minGlobalObj},
and 
	\begin{equation}
	G_\lambda (\BF{P}) = F_\lambda (\BF{P}) - F_\lambda^\text{min}. 
	\end{equation}
Then, it holds that, for any $\lambda\geq 0$ and $\BF{P}\in\BB{R}^{I\times J}$,
	\begin{equation}
	\left\| \BF{P}^\RM{T}\BF{A}
	- \BF{A}^\RM{T}\BF{P} \right\|_\RM{F}^2
	\leq 2G_\lambda(\BF{P})
	- 2\lambda \left\| \BF{P} - {\BF{P}}^*_\lambda \right\|_\RM{F}^2. 
	\end{equation}
\end{prop}
For the proof of Proposition~\ref{label:thm:symmetricity}, 
we formulate the following lemma, which can be derived
immediately after some algebraic manipulation. 

\begin{lem}
\label{label:lem:Glambda}
For $\lambda \geq 0$, 
	\begin{equation}
	\begin{split}
	G_\lambda (\BF{P}) 
	= F_\lambda (\BF{P}) - F_\lambda ({\BF{P}}^*_\lambda) 
	= \left\| \BF{P}^\RM{T}\BF{A} 
	- {\BF{P}}_\lambda^{*\RM{T}}\BF{A}
	 \right\|_\RM{F}^2 + 
	\lambda \left\| \BF{P} - {\BF{P}}^*_\lambda \right\|_\RM{F}^2. 
	\end{split}
	\end{equation}
\end{lem}
\begin{proof}
\textit{Proof of Proposition~\ref{label:thm:symmetricity}}. 
From Lemma~\ref{label:lem:Glambda}, it follows that 
	\begin{equation*} 
	\begin{split}
	\left\| \BF{P}^\RM{T}\BF{A} - \BF{A}^\RM{T}\BF{P} \right\|_\RM{F}^2
	& = \left\| \BF{P}^\RM{T}\BF{A} 
	- {\BF{P}}_\lambda^{*\RM{T}}\BF{A}
	+ {\BF{P}}_\lambda^{*\RM{T}}\BF{A} 
	- \BF{A}^\RM{T}\BF{P} \right\|_\RM{F}^2 
	\leq 2 \left\| \BF{P}^\RM{T}\BF{A} 
	- {\BF{P}}_\lambda^{*\RM{T}}\BF{A}
	 \right\|_\RM{F}^2	\\
	& = 2G_\lambda(\BF{P})
	- 2\lambda \left\| \BF{P} - {\BF{P}}^*_\lambda \right\|_\RM{F}^2. 
	\end{split}
	\end{equation*}
\end{proof}

The distribution of the eigenvalues and the singular values 
of the preconditioned matrix $\BF{P}^\RM{T}\BF{A}$ affects the 
convergence of an iterative method. The following theorem states that 
the eigenvalues and the singular values of $\BF{P}^\RM{T}\BF{A}$ 
obtained by the proposed method can be made close 
to the eigen/singular values of the matrix 
${\BF{P}}_\lambda^{*\RM{T}}\BF{A}$
by decreasing the $G_\lambda(\BF{P})$ value. 

\begin{thm} 
\label{thm_eigen_singularvalue}
Let $\BF{A}, {\BF{P}}^*_\lambda \in\BB{R}^{I\times J}$
 $(I\geq J)$ and $G_\lambda (\BF{P})$ be defined 
 as in Proposition~\ref{label:thm:symmetricity}. 
Then, for any $\lambda \geq 0$ and $\BF{P}\in\BB{R}^{I\times J}$,
		\begin{align}
		\sum_{r=1}^{J} 
		\left| \lambda_r ([\BF{P}^\RM{T}\BF{A}]_\RM{S})
		- \lambda_r ({\BF{P}}_\lambda^{*\RM{T}}\BF{A})
		\right|^2 
		& \leq G_\lambda (\BF{P})
		- \lambda \left\| \BF{P} - {\BF{P}}^*_\lambda \right\|_\RM{F}^2, 
		\label{thm_eigenvalue}
		\\
		\sum_{r=1}^{J} 
		\left| \sigma_r (\BF{P}^\RM{T}\BF{A})
		- \sigma_r ({\BF{P}}_\lambda^{*\RM{T}}\BF{A})
		\right|^2 
		& \leq G_\lambda (\BF{P})
		- \lambda \left\| \BF{P} - {\BF{P}}^*_\lambda \right\|_\RM{F}^2, 
		\label{thm_singularvalue}
		\end{align}
where $[\BF{N}]_\RM{S} = (\BF{N}+\BF{N}^\RM{T})/2$, 
$\lambda_r(\BF{M})$ are the eigenvalues of a real symmetric 
matrix $\BF{M}$, $\sigma_r(\BF{N})$ are the singular values of 
a matrix $\BF{N}$, and both $\lambda_r(\BF{M})$ and 
$\sigma_r(\BF{N})$ are arranged in decreasing order. 
\end{thm}

\begin{proof} 
Since for any $\BF{M}\in\BB{R}^{J\times J}$, it holds that 
$\| \BF{M} \|_\RM{F}^2 - \| [\BF{M}]_\RM{S} \|_\RM{F}^2
= \| \BF{M} - \BF{M}^\RM{T} \|_\RM{F}^2/4 \geq 0 $, 
we can derive, by the substitution
$\BF{M} = \BF{P}^\RM{T}\BF{A} 
- {\BF{P}}_\lambda^{*\RM{T}}\BF{A}$, that 
	$$
	\left\| \BF{P}^\RM{T}\BF{A} 
	- {\BF{P}}_\lambda^{*\RM{T}}\BF{A} \right\|_\RM{F}^2
	\geq 
	\left\| [\BF{P}^\RM{T}\BF{A}]_\RM{S}
	- {\BF{P}}_\lambda^{*\RM{T}}\BF{A} \right\|_\RM{F}^2. 
	$$
We can expand the right hand side of the above inequality as 
	\begin{equation}
	\left\| [\BF{P}^\RM{T}\BF{A}]_\RM{S}
	- {\BF{P}}_\lambda^{*\RM{T}}\BF{A} \right\|_\RM{F}^2
	= \left\| [\BF{P}^\RM{T}\BF{A}]_\RM{S} \right\|_\RM{F}^2
	+ \left\| {\BF{P}}_\lambda^{*\RM{T}}\BF{A} \right\|_\RM{F}^2
	-2 \cdot \text{trace}\left( [\BF{P}^\RM{T}\BF{A}]_\RM{S}
	\cdot \BF{A}^\RM{T} {\BF{P}}^*_\lambda
	\right). 
	\end{equation}
Note that, for a  matrix $\BF{M}\in\BB{R}^{J\times J}$
with a Shur decomposition $\BF{M}=\BF{QTQ}^\RM{T}$, we have 
	$$
	\| \BF{M} \|_\RM{F}^2 = \| \BF{T} \|_\RM{F}^2
	\geq 	\| \text{diag}(\BF{T}) \|_2^2  
	= 	\sum_{r=1}^J | \lambda_r(\BF{M}) |^2. 
	$$
As it is stated in \cite[Lemma II.1]{Lass1995}, the following relations hold: 
for any $J\times J$ Hermitian matrices $\BF{M}$ and $\BF{N}$, 
	$$
	\sum_{r=1}^J 
	\lambda_r (\BF{M}) \lambda_{J-r+1} (\BF{N})
	\leq 
	\text{trace}( \BF{MN} )  
	\leq 
	\sum_{r=1}^J 
	\lambda_r(\BF{M}) \lambda_r(\BF{N}). 
	$$
By using the above inequalities, we can derive that 
	\begin{equation}
	\begin{split}
	\left\| [\BF{P}^\RM{T}\BF{A}]_\RM{S}
	- {\BF{P}}_\lambda^{*\RM{T}}\BF{A} \right\|_\RM{F}^2
	&  \geq 
	\sum_{r=1}^J \left( 
	| \lambda_r( [\BF{P}^\RM{T}\BF{A}]_\RM{S} ) |^2
	+ | \lambda_r( {\BF{P}}_\lambda^{*\RM{T}}\BF{A} ) |^2
	-2 \lambda_r( [\BF{P}^\RM{T}\BF{A}]_\RM{S} )
	\lambda_r( {\BF{P}}_\lambda^{*\RM{T}}\BF{A} )
	\right)
	\\
	& = 
	\sum_{r=1}^{J} 
		\left| \lambda_r ([\BF{P}^\RM{T}\BF{A}]_\RM{S})
		- \lambda_r ({\BF{P}}_\lambda^{*\RM{T}}\BF{A})
		\right|^2. 
	\end{split}
	\end{equation}
The result in \eqref{thm_eigenvalue} follows using 
Lemma~\ref{label:lem:Glambda}. 

To prove the result in \eqref{thm_singularvalue}, 
we use the expression
	\begin{equation}
	\left\| \BF{P}^\RM{T}\BF{A}
	- {\BF{P}}_\lambda^{*\RM{T}}\BF{A} \right\|_\RM{F}^2
	= \left\| \BF{P}^\RM{T}\BF{A} \right\|_\RM{F}^2
	+ \left\| {\BF{P}}_\lambda^{*\RM{T}}\BF{A} \right\|_\RM{F}^2
	-2 \cdot \text{trace}\left( \BF{P}^\RM{T}\BF{A}
	\cdot \BF{A}^\RM{T} {\BF{P}}^*_\lambda
	\right). 
	\end{equation}
For a matrix $\BF{M}\in\BB{R}^{J\times J}$, 
we have $\|\BF{M}\|_\RM{F}^2 = \sum_{r=1}^J |\sigma_r (\BF{M}) |^2$. 
By applying the von Neumann's trace inequality \cite{Mirsky75}, which states that, 
for $J\times J$ matrices $\BF{M}$ and $\BF{N}$, 
	$$
	| \text{trace} (\BF{MN}) | \leq 
	\sum_{r=1}^J \sigma_r(\BF{M}) \sigma_r(\BF{N}),
	$$
we can finally derive that 
	$$
	\left\| \BF{P}^\RM{T}\BF{A}
	- {\BF{P}}_\lambda^{*\RM{T}}\BF{A} \right\|_\RM{F}^2
	\geq 
	\sum_{r=1}^{J} 
		\left| \sigma_r (\BF{P}^\RM{T}\BF{A})
		- \sigma_r ({\BF{P}}_\lambda^{*\RM{T}}\BF{A})
		\right|^2 . 
	$$
Similarly, we can obtain the result in \eqref{thm_singularvalue} by using
Lemma~\ref{label:lem:Glambda}. 
\end{proof}

From \eqref{thm_eigenvalue} of 
Theorem~\ref{thm_eigen_singularvalue}, 
we can say that the eigenvalues of the matrix $[\BF{P}^\RM{T}\BF{A}]_\RM{S}$ can become 
close to those of the symmetric positive semidefinite matrix 
${\BF{P}}_\lambda^{*\RM{T}}\BF{A}$
when $G_\lambda(\BF{A})$ is small enough. 
If $\BF{A}$ is of full column rank, then the smallest 
eigenvalue of $[\BF{P}^\RM{T}\BF{A}]_\RM{S}$ 
can approach to 
$\lambda_J({\BF{P}}_\lambda^{*\RM{T}}\BF{A})
= \sigma_J^2/(\sigma_J^2+\lambda) >0$, 
where $\sigma_j = \sigma_j(\BF{A})$, 
so the preconditioned matrix $\BF{P}^\RM{T}\BF{A}$
can also become positive definite. 
In addition, from \eqref{thm_singularvalue} of
Theorem~\ref{thm_eigen_singularvalue}, 
the spectral condition number of $\BF{P}^\RM{T}\BF{A}$
also can be made close to that of ${\BF{P}}_\lambda^{*\RM{T}}\BF{A}$
by decreasing $G_\lambda(\BF{P})$ gradually.

\subsubsection{Computational Complexity}
\label{sec_comput_compl}

The most time-consuming step in the MALS algorithm
is the optimization step for solving the reduced optimization 
problems \eqref{optim:local:TTMALS}. 
Let $Q=\max_n(I_n,J_n)$, $R=\max_n(R_n)$, and $R_A=\max_n(R^A_n)$. 
For a fast computation of the solution, we can apply standard 
iterative methods such as the \texttt{gmres} in Matlab
for solving the system of linear equations 
$(\overline{\BF{A}}_n + \lambda \BF{I}_{R_{n-1}K_nK_{n+1}R_{n+1}})\BF{x}$ $= \overline{\BF{b}}_n$. 
In this case,
the matrix $\overline{\BF{A}}_n$ does not need to be computed 
explicitly, instead the matrix-by-vector multiplication 
$\overline{\BF{A}}_n\BF{x}$ can be computed faster 
by the gradual (recursive) contraction of the tensors 
$\{\ten{L}_1^{<n},$ $\ten{X},$ $\ten{A}^{(n)},$ $\ten{A}^{(n)},$
$\ten{A}^{(n+1)},$ $\ten{A}^{(n+1)},$ $\ten{R}_1^{>n+1}\}$, 
as illustrated in Figure~\ref{fig_matvec_Ax}. 
The computational complexity for the multiplication $\overline{\BF{A}}_n \BF{x}$ is $\CL{O}( R^3R_A^2 Q^4 + R^2R_A^3Q^5)$. 
Since the computational cost for each iteration is independent of the order $N$ if $R, R_A,$ and $Q$ are bounded, the total computational cost for optimizing every TT-cores is logarithmic in the matrix size $Q^N\times Q^N$. 

On the other hand, the explicit computation of the  matrix 
$\overline{\BF{A}}_n$ can be performed by the iterative contraction of 
the tensors $\{\ten{L}_1^{<n},$ $\ten{A}^{(n)},$ $\ten{A}^{(n)},$
$\ten{A}^{(n+1)},$ $\ten{A}^{(n+1)},$ $\ten{R}_1^{>n+1}\}$, 
and its computational complexity is 
$\CL{O}(R^4R_A^2Q^4 + R^2R_A^3Q^5)$. Moreover, 
a direct method, such as the LU factorization or the pseudoinverse, for solving the system $\overline{\BF{A}}_n \BF{x} = 
\overline{\BF{b}}_n$ can cost up to $\CL{O}(R^6Q^6)$. 

\begin{figure}
\centering
\includegraphics[height=2.3cm,width=12.7cm]{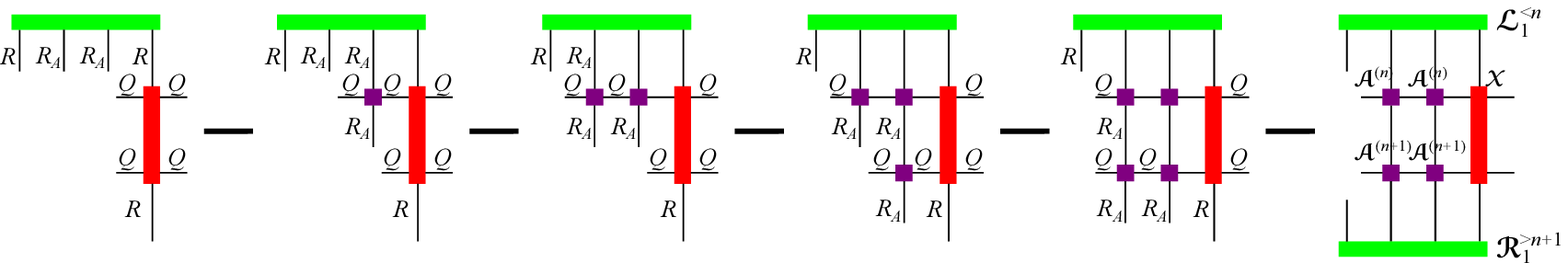}
\caption{\label{fig_matvec_Ax}
A procedure for the sequential computation of the matrix-by-vector multiplication
$\overline{\BF{A}}_n \BF{x}$ with a vector
$\BF{x} \in \BB{R}^{R_{n-1} I_nJ_n I_{n+1}J_{n+1} R_{n+1}}$, 
which can be carried out by a sequential contraction of the tensors 
$\{\ten{L}_1^{<n},$ $\ten{X},$ $\ten{A}^{(n)},$ $\ten{A}^{(n)},$
$\ten{A}^{(n+1)},$ $\ten{A}^{(n+1)},$ $\ten{R}_1^{>n+1}\}$,
where $\ten{X}\in\BB{R}^{R_{n-1} \times I_n \times J_n \times I_{n+1}
 \times J_{n+1} \times R_{n+1}}$ is a 6th-order tensor. 
The mode sizes are denoted for simplicity by 
$Q=\max_n(I_n,J_n)$, $R=\max_n(R_n)$, and $R_A=\max_n(R^A_n)$.
}
\end{figure}

The estimated preconditioner is reusable
and does not need to be updated during the process of solving
the preconditioned system of linear equations. However,
the preconditioner $\BF{P}$ in the TT format needs to take
small TT-ranks because the product $\BF{P}^\RM{T}\BF{A}$
increases the TT-ranks \cite{Ose2011}.
Note that in the case of sparse approximate inverse preconditioning
techniques \cite{Benzi1999}, similarly, the product $\BF{P}^\RM{T}\BF{A}$
would deteriorate sparsity of $\BF{A}$ even if both $\BF{P}$ and
$\BF{A}$ are sparse. In such cases, the preconditioning
can be performed implicitly.

\section{Numerical Simulations}
\label{sec_num_simulation}

In the simulation study, we considered several matrices including rectangular matrices and nonsymmetric matrices in order to check validity and evaluate performance of the proposed MALS algorithm.
We also compared the proposed method with an alternative method 
which applies a ``standard'' MALS method to solve the 
large scale system of linear equations (see also \cite{OseDol2012})
	\begin{equation} \label{std_mals_equation}
	(\BF{I}_J {\otimes} \BF{AA}^\RM{T}  + 
	\lambda \BF{I}_I {\otimes} \BF{I}_J
	) {\text{vec}}(\BF{P})
	= {\text{vec}}(\BF{A}). 
	\end{equation}
In practice, the matrix $\BF{I}_J {\otimes} \BF{AA}^\RM{T}$ is represented in matrix TT format with 4th-order TT-cores of the sizes $(R^A_{n-1})^2 \times I_nJ_n \times I_nJ_n \times (R^A_n)^2$, 
$n=1,2,\ldots,N$. Due to the relatively large sizes of the TT-cores, the computational cost for the matrix-by-vector multiplication in the standard MALS algorithm has complexity of $\CL{O}(R^3R_A^2Q^4 + R^2R_A^4Q^6)$ \cite{OseDol2012}.

The TT-ranks of the estimated pseudoinverse were bounded by $\overline{R}=(50,\ldots,50)$. 
We repeated the simulations 30 times with random initializations and averaged the results.
In the simulation results, we calculated the value of the relative residual $r_\lambda$ \eqref{rel_residual} as described in Section~\ref{sec:relative_error}. 
Our code was implemented in Matlab. We used the Matlab version of TT-Toolbox \cite{Ose2014} for building and manipulating TT formats for large-scale matrices and vectors. Simulations were performed on a desktop computer with an Intel Core i7 4960X CPU at 3.60 GHz with 32 GB of memory running Windows 7 Professional and Matlab R2010b.

\subsection{Example 1: Rectangular Circulant Matrices with Prescribed Singular Values}

Rectangular matrices $\BF{A}\in\BB{R}^{2^{N+1} \times 2^N}$ 
were defined by $\BF{A} = \frac{1}{\sqrt{2}} 
\begin{bmatrix}\BF{C}\\[-0.35pc] \BF{C}\end{bmatrix}$, where 
$\BF{C}=[c_{i-j}]_{ij}\in\BB{R}^{2^N\times 2^N}$ are circulant matrices, 
i.e., $c_{k} = c_{k-2^N}, k=0,1,\ldots,2^N-1$, whose entries 
were determined in the following way.  
The (unordered) eigen/singular values of $\BF{C}$ were 
first prescribed by, with $J=2^N$, 
(see, Figure~\ref{fig-circ-singularvalue}(a))
	$$
	\sigma_j = f(j)
	= \frac{1}{B} \max\left(0, 
	 \left| \frac{j}{J} - 0.5 \right| + B - 0.5
	\right), 
	\quad B>0, 
	\ j=0,\ldots, J-1. 
	$$
Then, the entries $c_k$ of the circulant matrix $\BF{C}$ 
were determined by the discrete Fourier inverse transform (IFFT)
of the eigen/singular values \cite{Davis79}, i.e., (see, Figure~\ref{fig-circ-singularvalue}(b)) 
	$$
	c_k =  \texttt{ifft}(\sigma_j) = \frac{1}{J} \sum_{j=0}^{J-1} \sigma_j 
	\exp \left(\frac{2\pi i j k}{J}\right). 
	$$
The very large scale vector $[c_k]_k$ represented in TT format can be efficiently computed 
by the {QTT-FFT} algorithm \cite{DolKhoSav2012}, 
which is available in TT-Toolbox \cite{Ose2014}. 
By construction, the matrix $\BF{A}$ has the same singular values 
as $\BF{C}$. We note that $\BF{A}$ is ill-conditioned
when $B\leq 0.5$.

\begin{figure}
\centering
\begin{tabular}{cc}
\includegraphics[width=6cm]{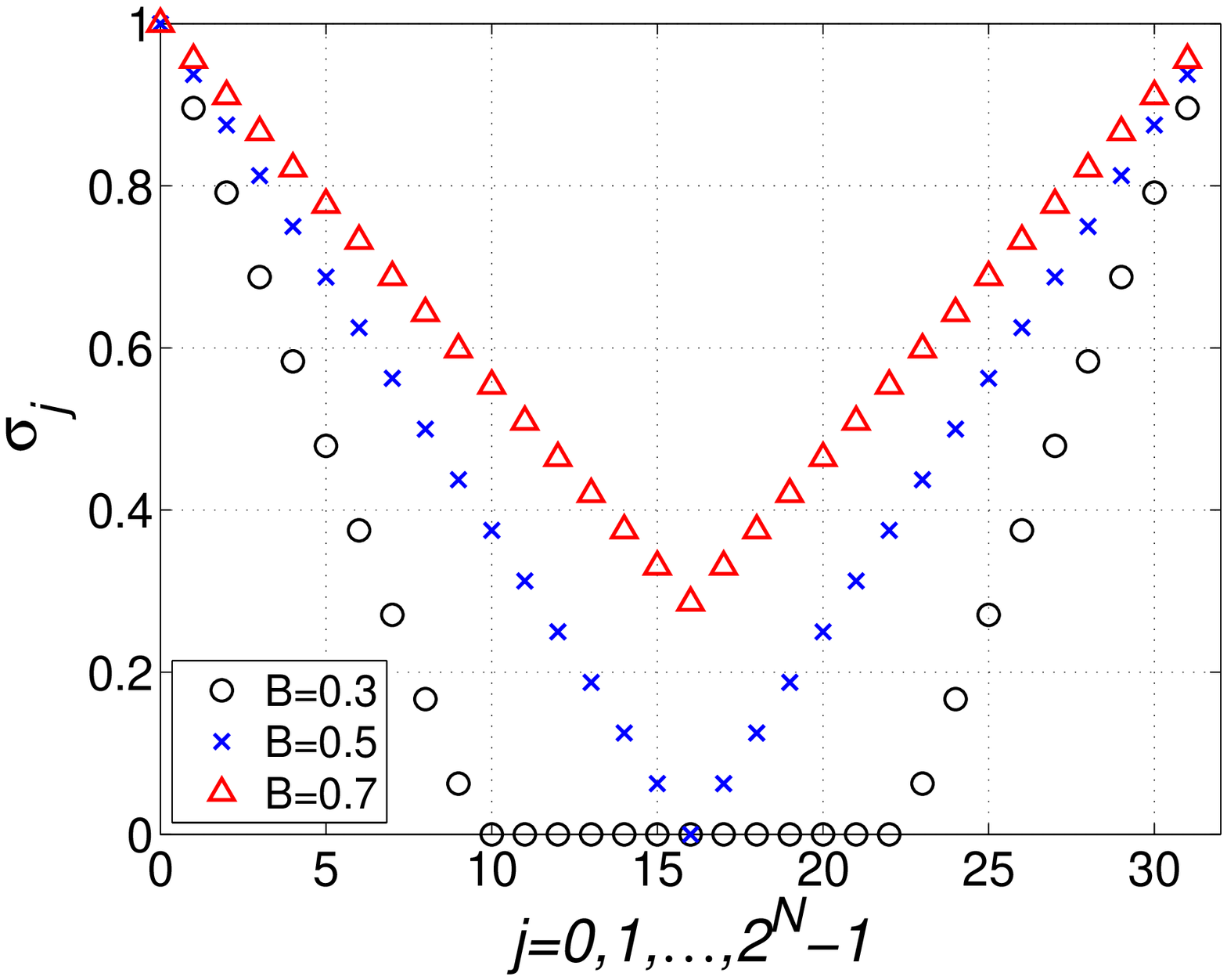}& 
\includegraphics[width=6cm]{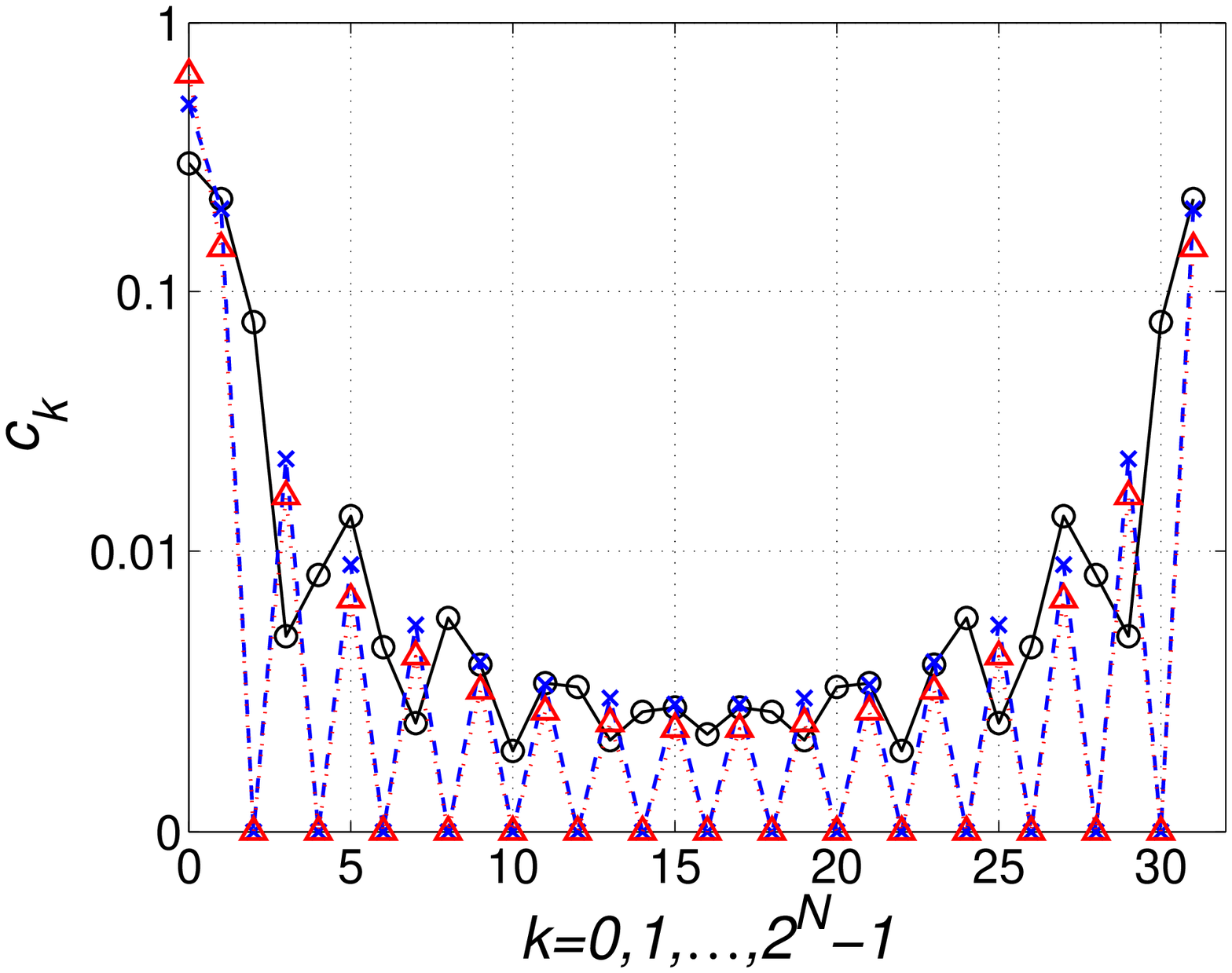}\\
(a) $\sigma_j$ &(b) $c_k$
\end{tabular}
\caption{\label{fig-circ-singularvalue}
(a) The (unordered) prescribed eigen/singular values $\sigma_j$ 
of the $2^{N}\times 2^N$ circulant matrix $\BF{C}$ and
(b) the entries $c_k$ of $\BF{C}$,  with 
$N={5}$ and various values of $B=0.3,0.5,0.7$. 
The rectangular matrix $\BF{A}\in\BB{R}^{2^{N+1}\times 2^N}$
has the same singular values as $\BF{C}$. }
\end{figure}


Figure~\ref{Fig:Circ_Error_vs_iter} illustrates that the 
relative residual decreases monotonically as the iteration proceeds. 
From Figure~\ref{Fig:Circ_Error_vs_iter}(a), we can see that the
proposed algorithm converges relatively fast within one or two 
full sweeps (1 full sweep is equal to $2(N-2)$ iterations)
in the case that $\lambda>0$, and the TT-ranks of the 
estimated pseudoinverse also remain at low values. 
On the other hand, without regularization (i.e.,  $\lambda=0$), 
we can see that the convergence is slow, and the TT-ranks also 
grow very quickly during the iteration process.

\begin{figure}
\centering
\begin{tabular}{cc}
\includegraphics[width=6cm]{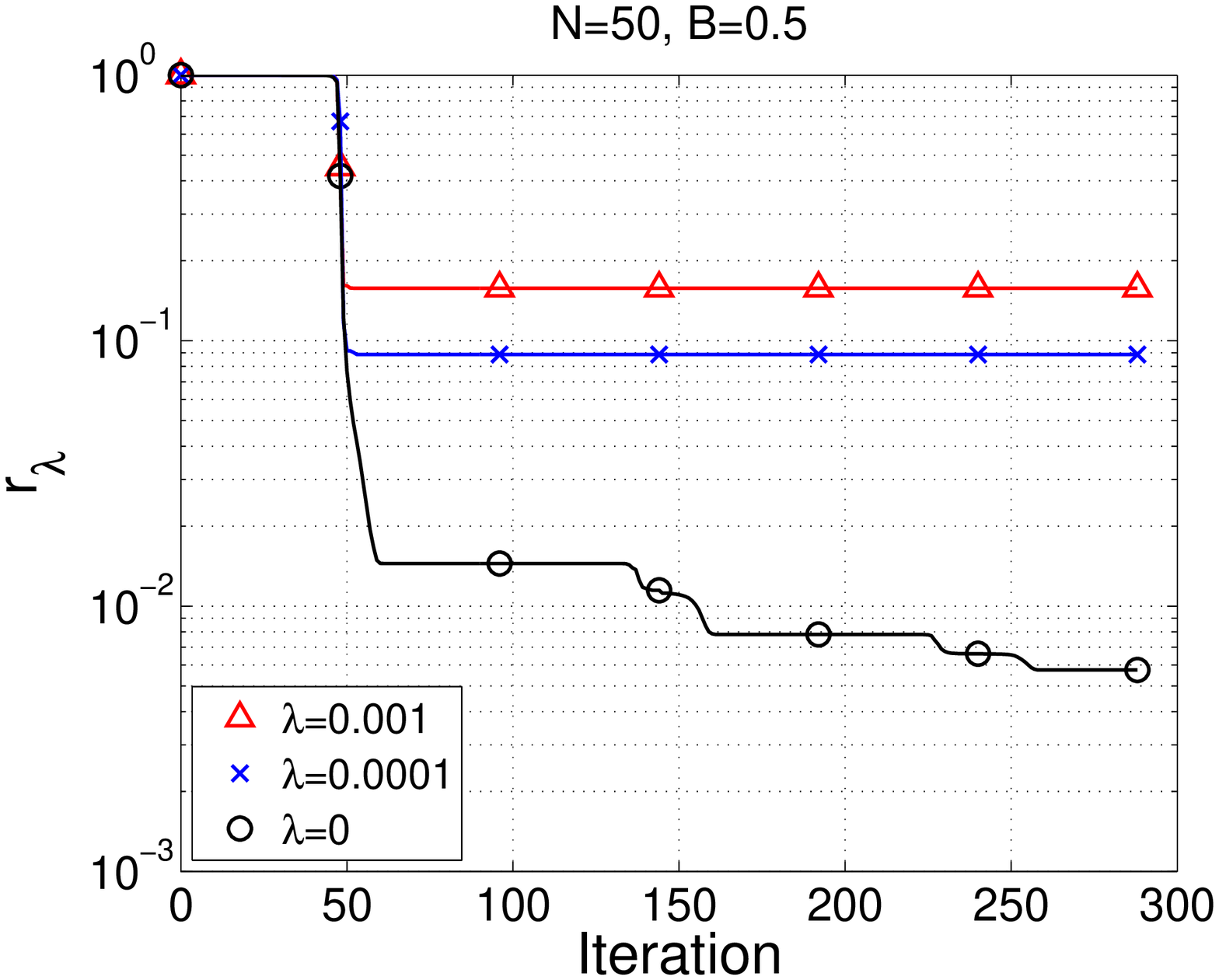}&
\includegraphics[width=6cm]{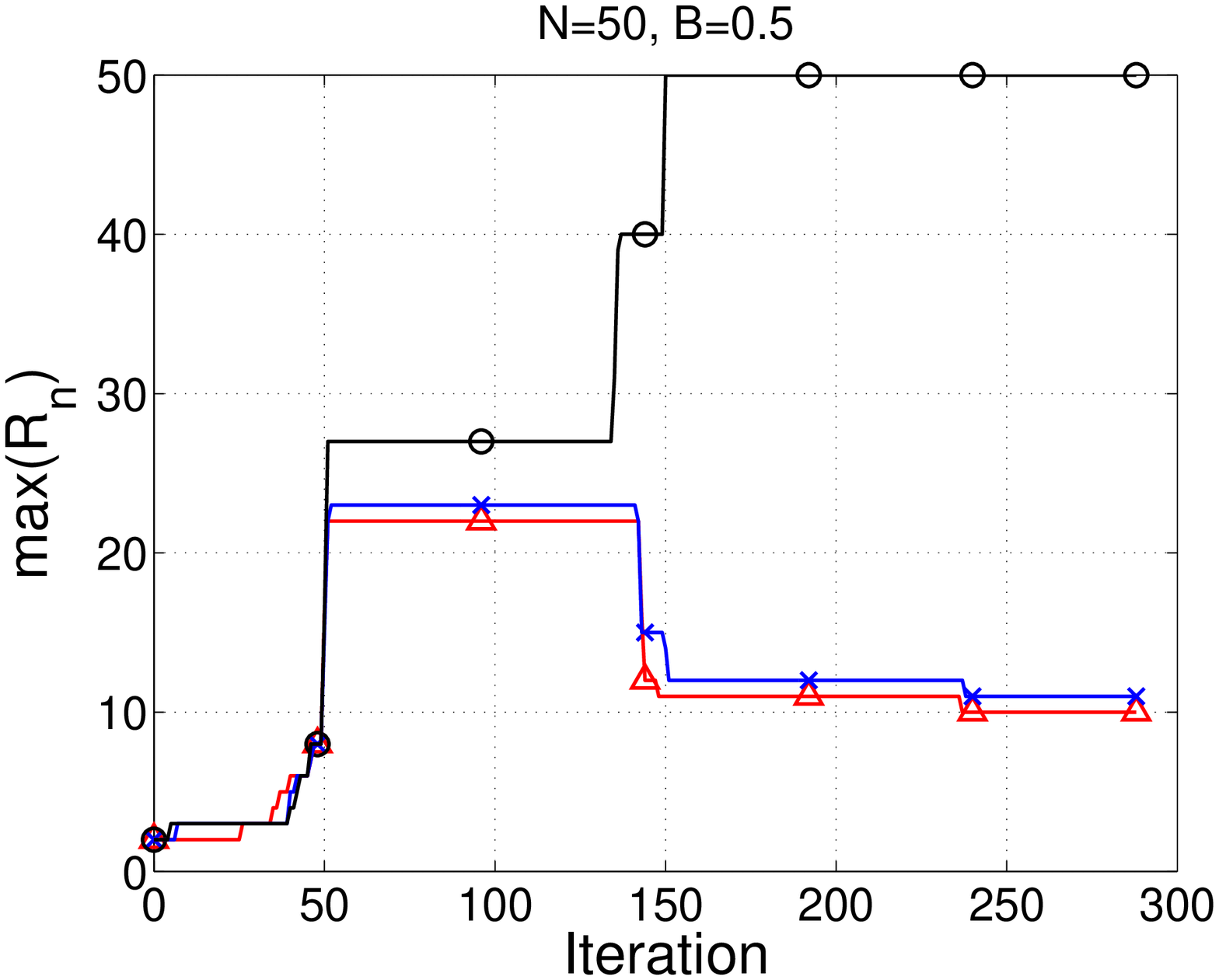}
\\(a) Relative residual  &(b) Maximum of TT-ranks
\end{tabular}
\caption{\label{Fig:Circ_Error_vs_iter}
Convergence of the proposed MALS algorithm for various values of 
the regularization parameter $\lambda$ for the $2^{N+1}\times 2^N$ 
rectangular circulant matrices with prescribed singular values
and $N=50, B=0.5$. 
(a) Relative residual versus iteration, and 
(b) maximum of the TT-ranks of the estimated pseudoinverse 
versus iteration. 
The markers on the lines indicate half-sweeps, i.e., every $N-2$ iterations.
}
\end{figure}

For comparison of the computational costs, we set the tolerance parameter at $\epsilon=0.2$. Figure~\ref{fig-circ_computation_BothMALS} illustrates the computational costs and the estimated TT-ranks by the proposed MALS algorithm (MALS-PINV) and the standard MALS algorithm (Std MALS), for various values of $20\leq N\leq 100$ and $B\in\{0.3,0.5,0.7\}$. Some values of the computational time are not displayed in the figure if it was larger than 360 seconds. 
We can  see that the computational costs increased 
only logarithmically with the matrix size $2^{N+1}\times 2^N$. 
The computation time of the proposed MALS algorithm 
was much smaller than that of the standard MALS algorithm (typically, by one or even two orders) as analyzed theoretically in Section~\ref{sec_comput_compl}
and in the first paragraph of Section~\ref{sec_num_simulation}. 
Moreover, the computation time for the regularized optimization
with $\lambda>0$ was shorter than that for the optimization 
without regularization, i.e., $\lambda=0$. 
In addition, for the case of ill-conditioned matrices with either 
$B=0.3$ or $B=0.5$, no regularization with $\lambda=0$
resulted in large estimated TT-ranks and high computational
costs, whereas the regularization with $\lambda>0$
significantly reduced estimated TT-ranks and 
also computational costs.

\begin{figure}
\centering
\begin{tabular}{ccc}
\includegraphics[width=3.9cm]{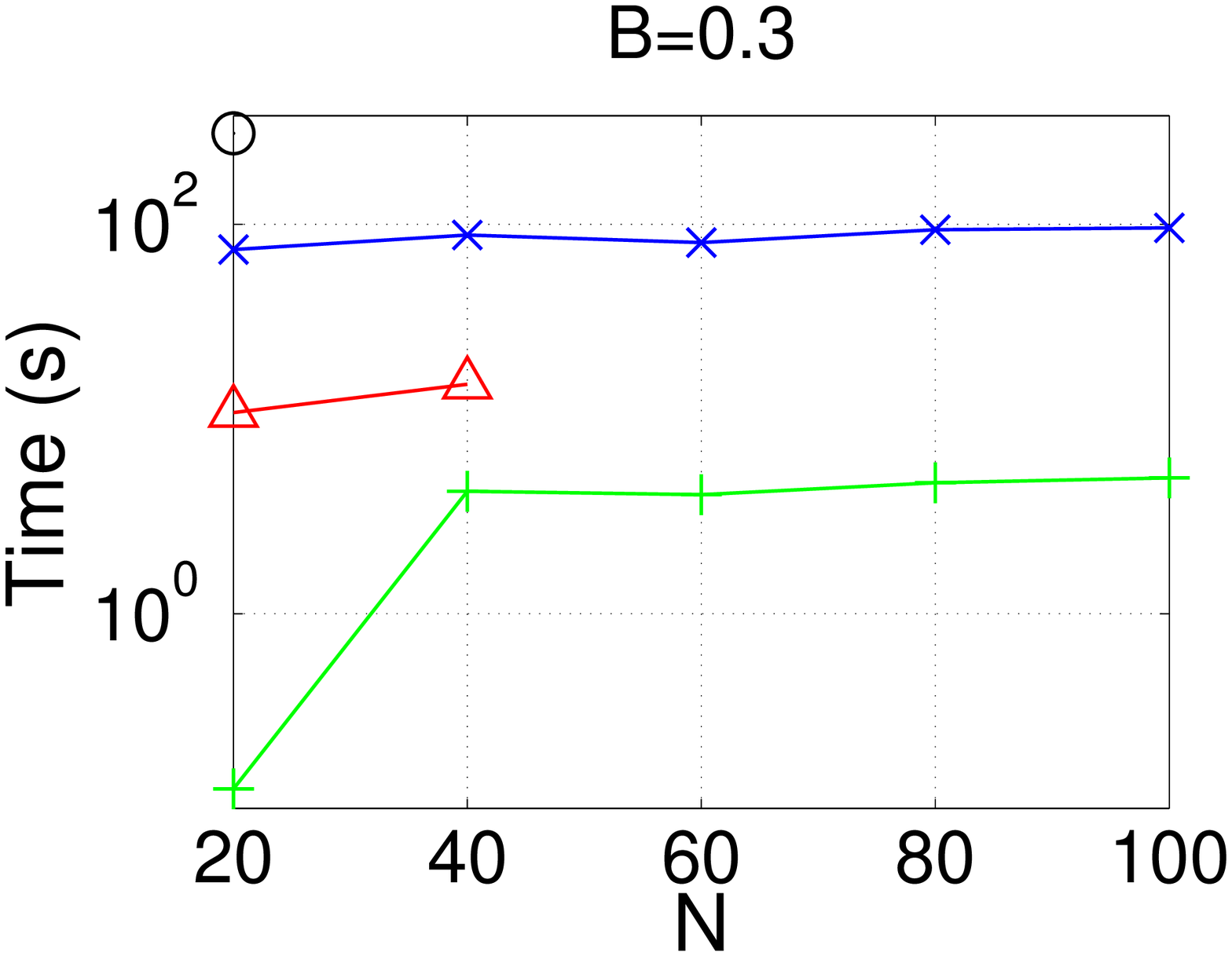}&
\includegraphics[width=3.9cm]{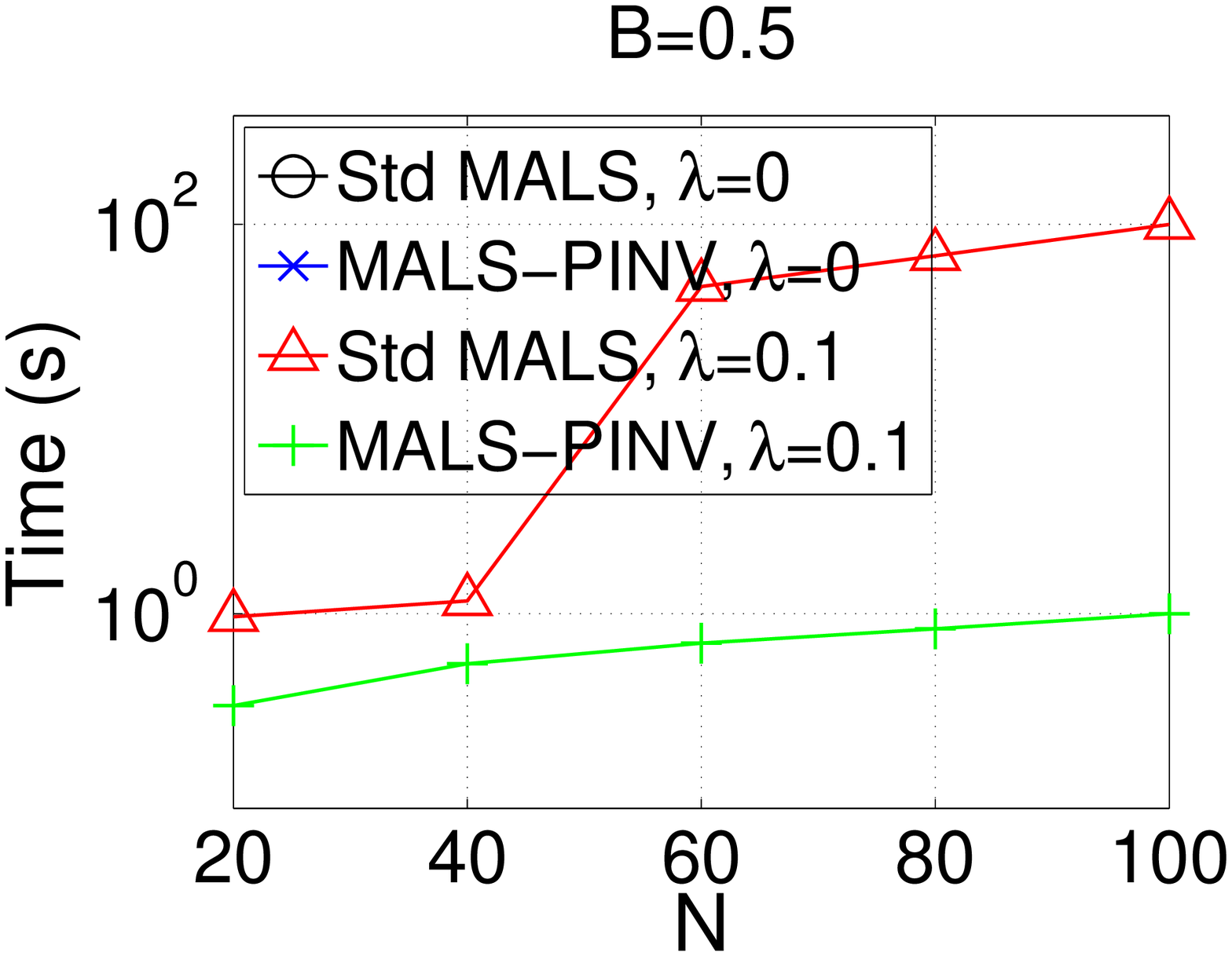}&
\includegraphics[width=3.9cm]{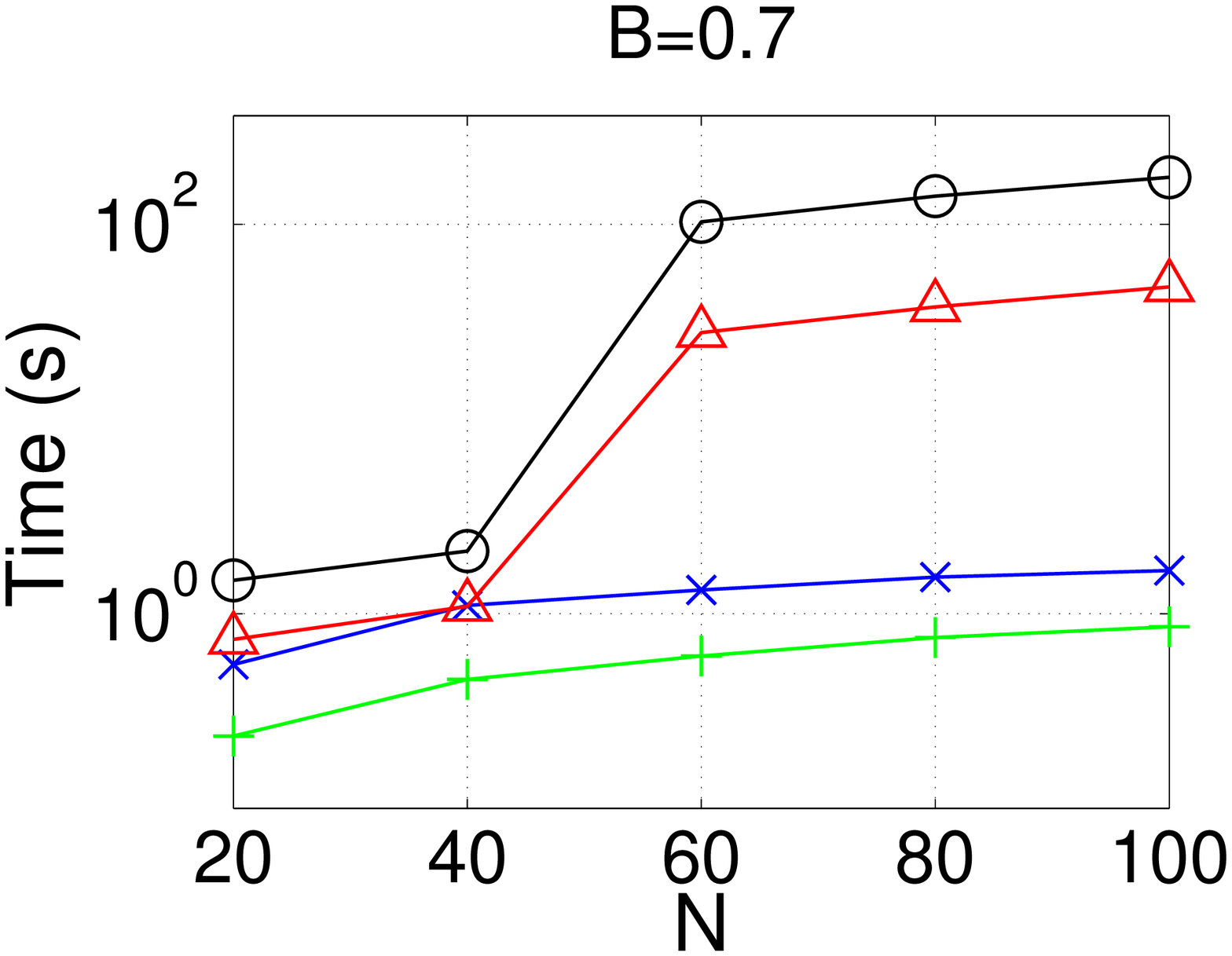}\\
\includegraphics[width=3.9cm]{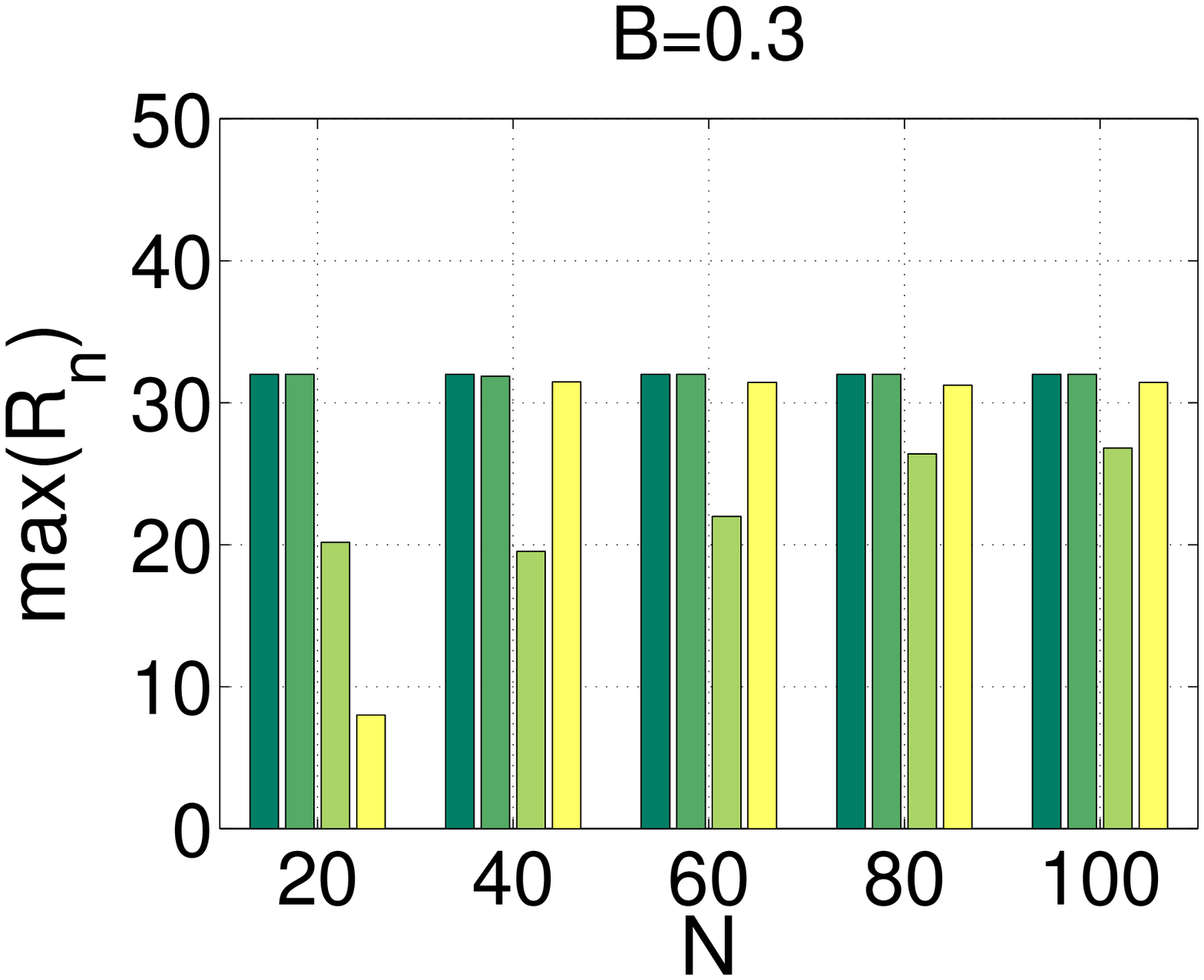}&
\includegraphics[width=3.9cm]{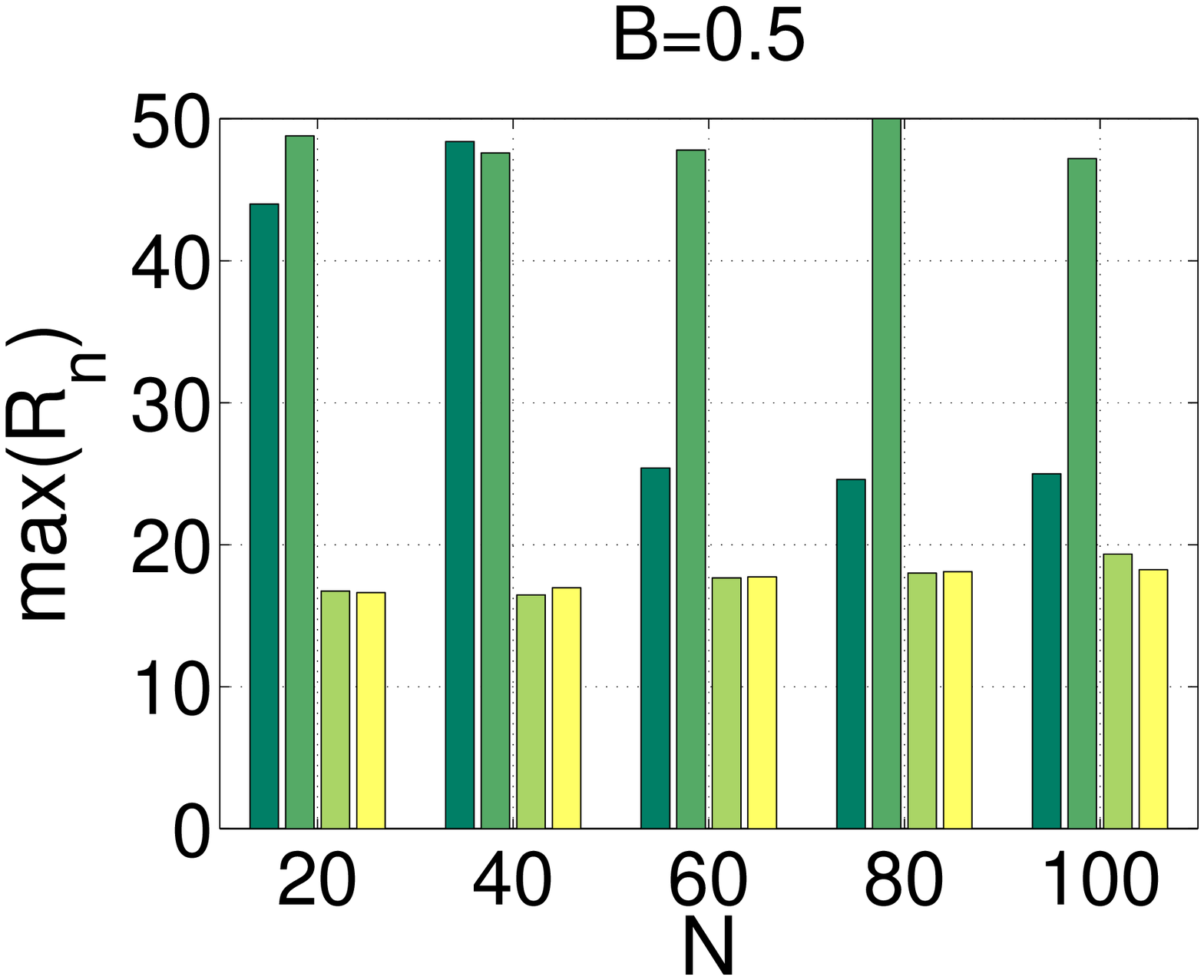}&
\includegraphics[width=3.9cm]{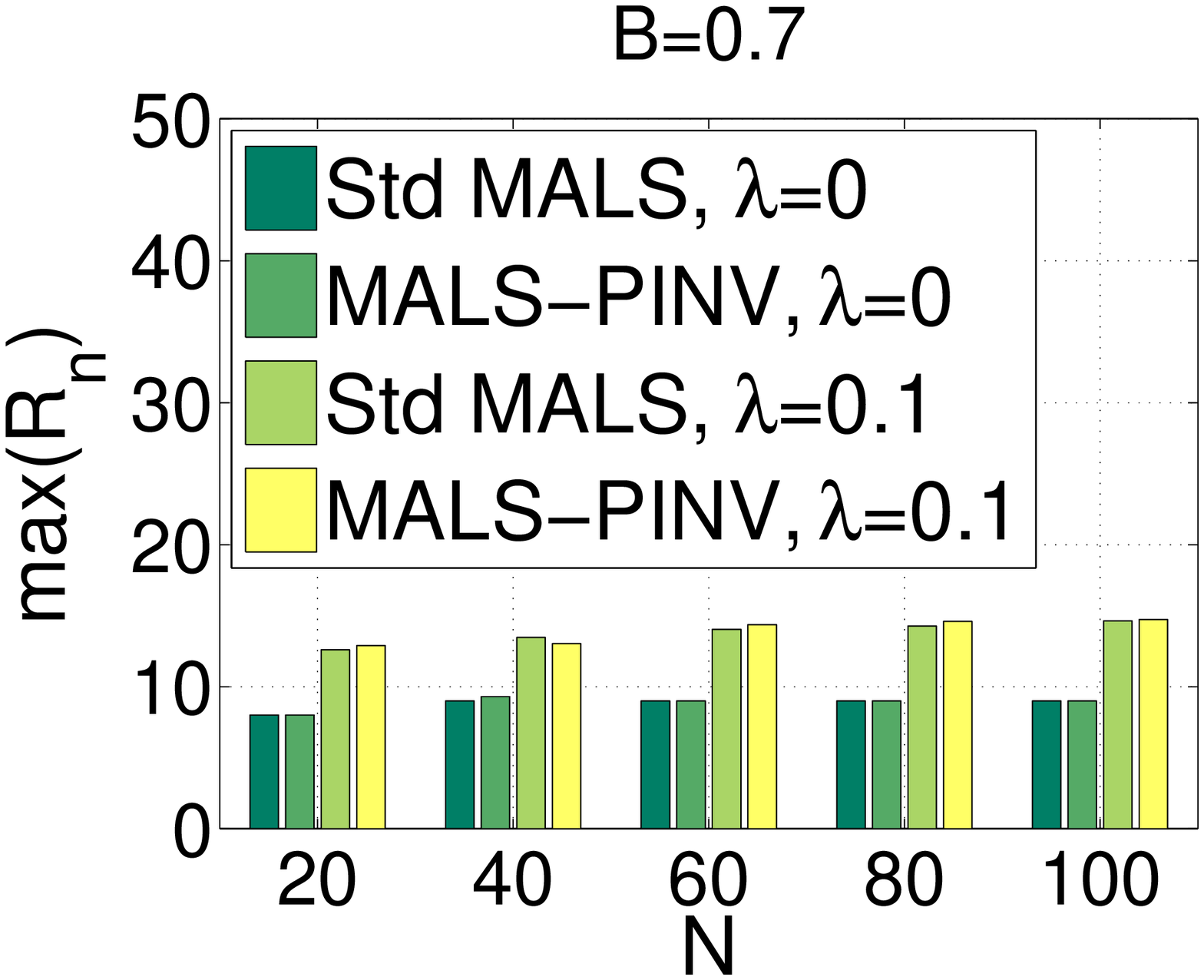}\\
(a) $B=0.3$ & (b) $B=0.5$ & (c) $B=0.7$ \\
\end{tabular}
\caption{\label{fig-circ_computation_BothMALS}
Computational cost versus $N$ (top row)
and maximum of the estimated TT-ranks versus $N$
(bottom row), for the $2^{N+1}\times 2^N$ rectangular circulant matrices with various values of 
$20\leq N\leq 100$, $B\in\{0.3,0.5,0.7\}$, and $\lambda\in\{0,0.1\}$. 
Some values of the computational time are not displayed in the figure if it was larger than 360 seconds. Std MALS means the standard MALS method applied for solving the linear system \eqref{std_mals_equation}, and MALS-PINV means the proposed MALS method.  
}
\end{figure}

\subsection{Example 2: Randomly Generated Matrices with Prescribed Singular Values}

Nonsymmetric matrices $\BF{A}\in\BB{R}^{2^{N}\times 2^N}$
were constructed in $N$-dimensional matrix TT format using
	$$
	\BF{A} = \BF{U\Sigma V}^\RM{T},
	$$
where $\BF{\Sigma}=\text{diag}(\sigma_0,\sigma_1,\ldots,\sigma_{2^N-1})$
is a diagonal matrix of singular values, and $\BF{U}$ and $\BF{V}$ are $2^N\times 2^N$ orthogonal matrices with
TT-ranks equal to 1, i.e.,
	$$
	\BF{U} = \BF{U}^{(1)} \otimes \cdots \otimes \BF{U}^{(N)},
	\quad
	\BF{V} = \BF{V}^{(1)} \otimes \cdots \otimes \BF{V}^{(N)},
	$$
and $\BF{U}^{(n)}\in\BB{R}^{2\times 2}$ and $\BF{V}^{(n)}\in\BB{R}^{2\times 2}$
are randomly generated orthogonal matrices.
The singular values were determined by
	$$
	\sigma_j = 10^{-\frac{j}{JK_0}},
	\quad
	j=0,1,\ldots,2^N-1,
	$$
for fixed $0 < K_0 \leq 1$.
The largest singular value is equal to 1, and the singular values decay to zero
as the index $j$ increases in a rate determined by $K_0$.
The TT-ranks of the matrix $\BF{A}$ and the inverse $\BF{A}^{-1}$ are the same as 
those of $\BF{\Sigma}$ and $\BF{\Sigma}^{-1}$, respectively, which are 
$R_n^A=R_n^{A^{-1}}=1$ $(1\leq n\leq N-1)$.

Figure~\ref{Fig:usv_Error_vs_it_K9} illustrates the convergence of the
proposed MALS algorithm for various values of the regularization parameter
$\lambda$ for the matrices with $N=50$ $(2^{50}\times 2^{50}\sim 10^{15}\times 10^{15})$ and $K_0=0.5$.
The convergence was monotonic, i.e., the relative residual was nonincreasing during the iteration process.
The larger $\lambda$ values resulted in larger relative residual values,
as described in Section~\ref{sec:relative_error}.
Since the matrix $\BF{A}$ is not ill-conditioned when $K_0=0.5$, 
the convergence was relatively fast, and the estimated
TT-ranks were small. In Figure~\ref{Fig:usv_Error_vs_it_K9}(b), 
when $\lambda=0$ the estimated TT-ranks were $R_n=1$,
which are equal to the true TT-ranks of the inverse $\BF{A}^{-1}$.

\begin{figure}
\centering
\begin{tabular}{cc}
\includegraphics[width=6cm]{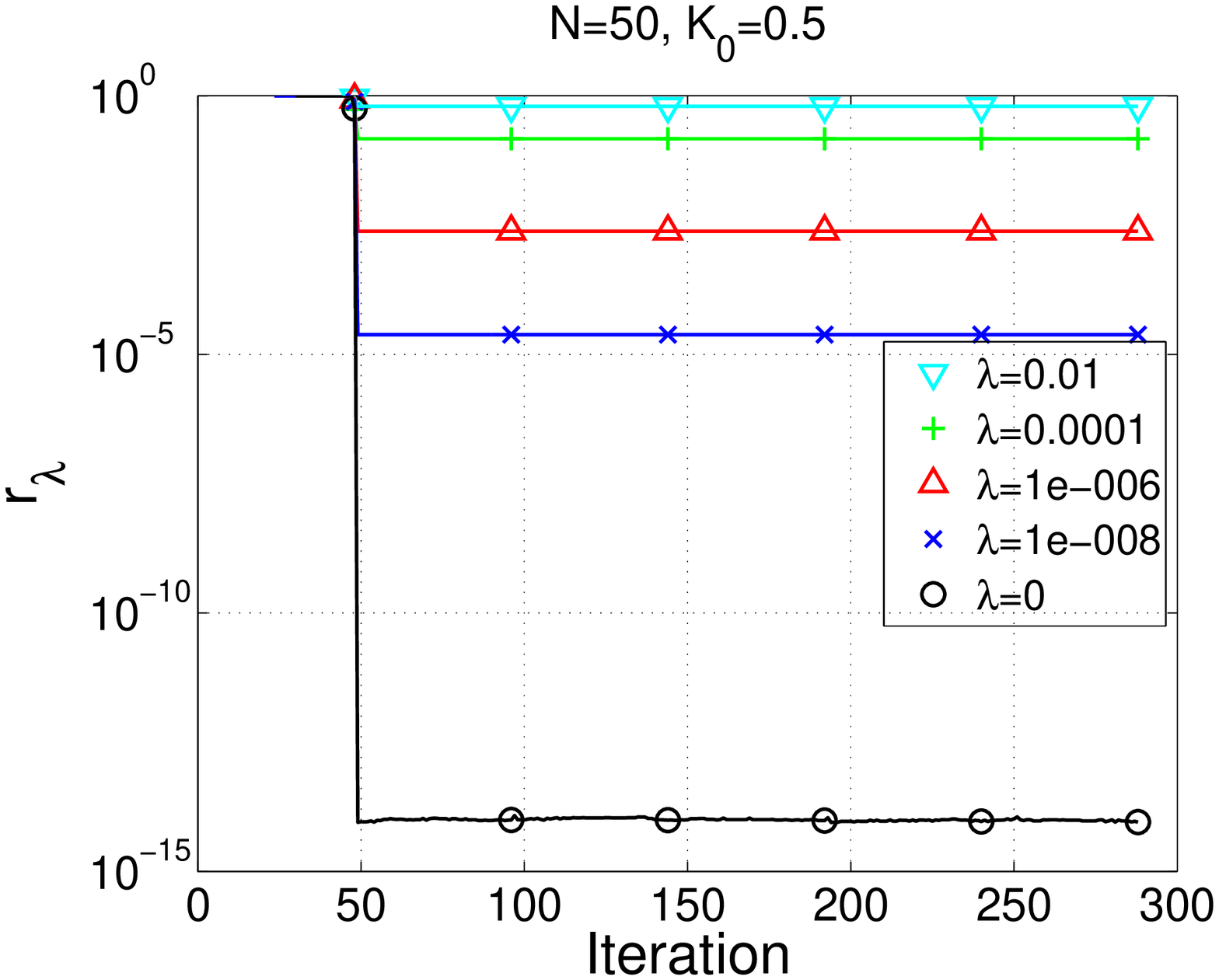}&
\includegraphics[width=6cm]{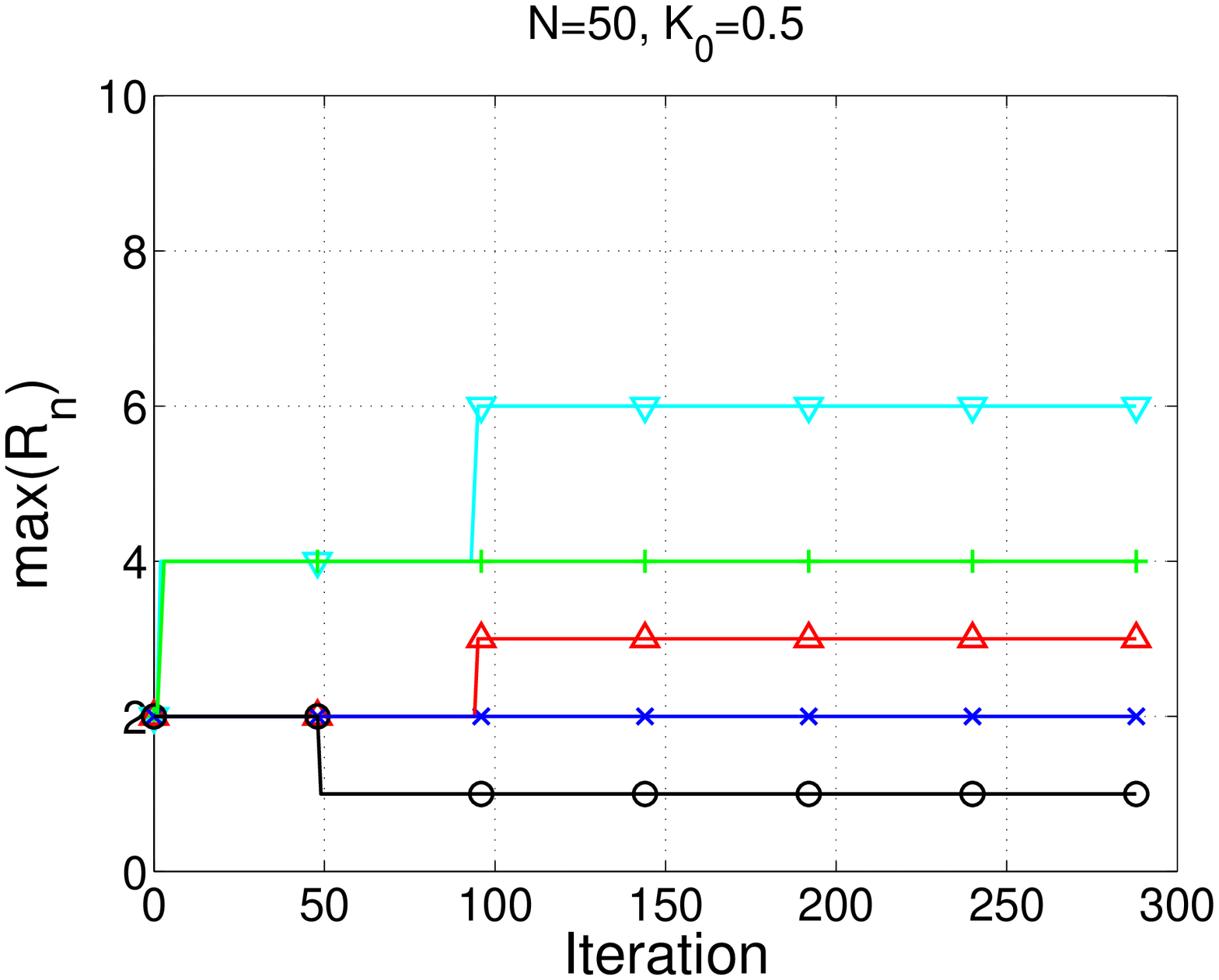}\\
(a) Relative residual & (b) Maximum of TT-ranks 
\end{tabular}
\caption{\label{Fig:usv_Error_vs_it_K9}Convergence of the
MALS algorithm for various values of the regularization parameter
 $\lambda$ for the $2^N\times 2^N$ randomly generated matrices with prescribed 
 singular values and $N=50$ and $K_0=0.5$.
(a) Relative residual versus iteration, and 
(b) maximum of the TT-ranks of the estimated pseudoinverse 
versus iteration. 
The markers on the lines indicate half-sweeps, i.e., every $N-2$ iterations.
}
\end{figure}

Figure~\ref{Fig:usv_Error_vs_K_L1}(a) shows that the obtained minimal
relative residual values were not different for various values
of $N$, but they were different over various values of $K_0$.
In addition, Figure~\ref{Fig:usv_Error_vs_K_L1}(b) shows that the
relative residual values were different for various values of
the regularization parameter $\lambda$. The simulation results
illustrated in Figures~\ref{Fig:usv_Error_vs_K_L1}(a) and (b)
are consistent and in agreement with the analysis in
Section~\ref{sec:relative_error}.

\begin{figure}
\centering
\begin{tabular}{cc}
\includegraphics[width=6cm]{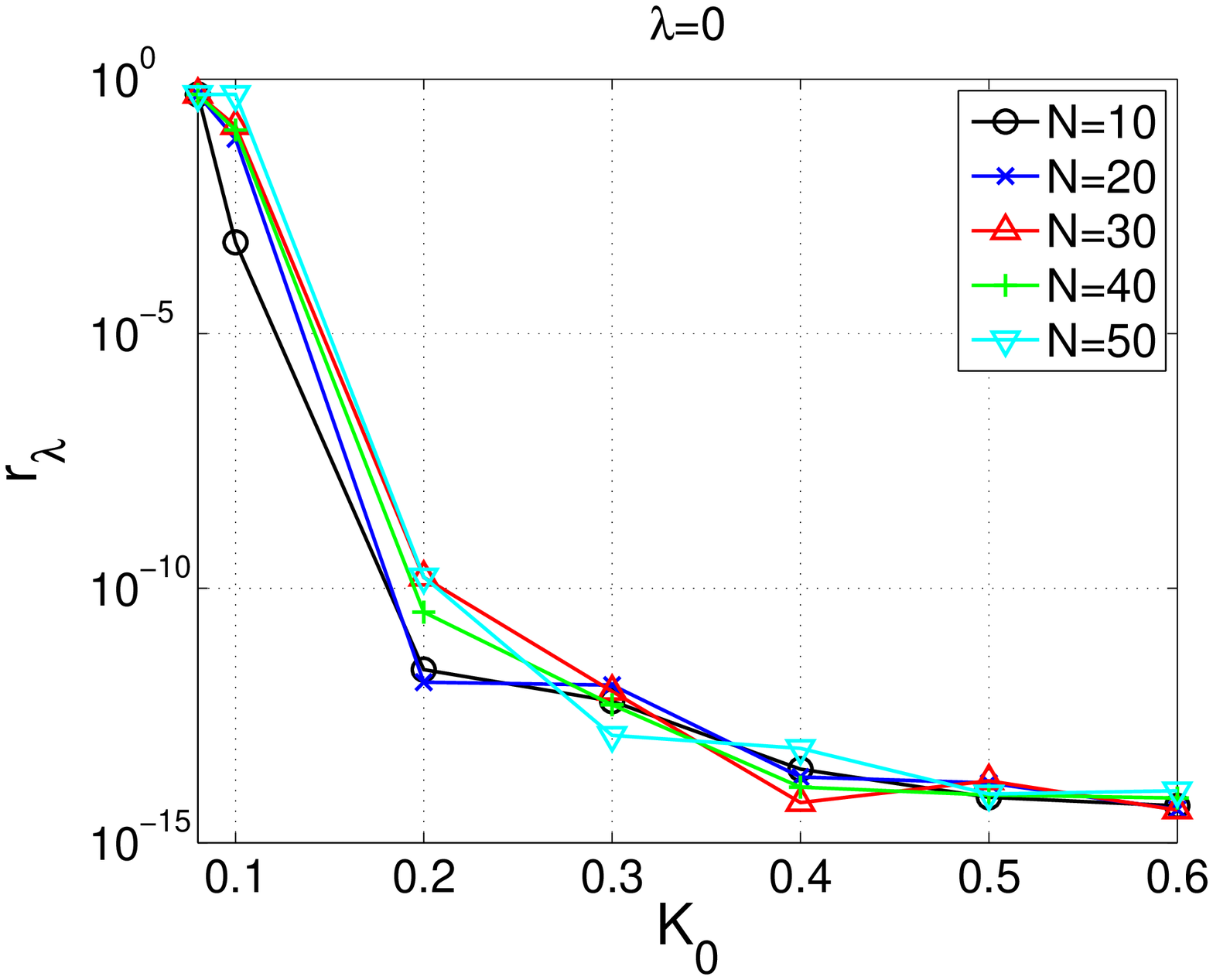} &
\includegraphics[width=6cm]{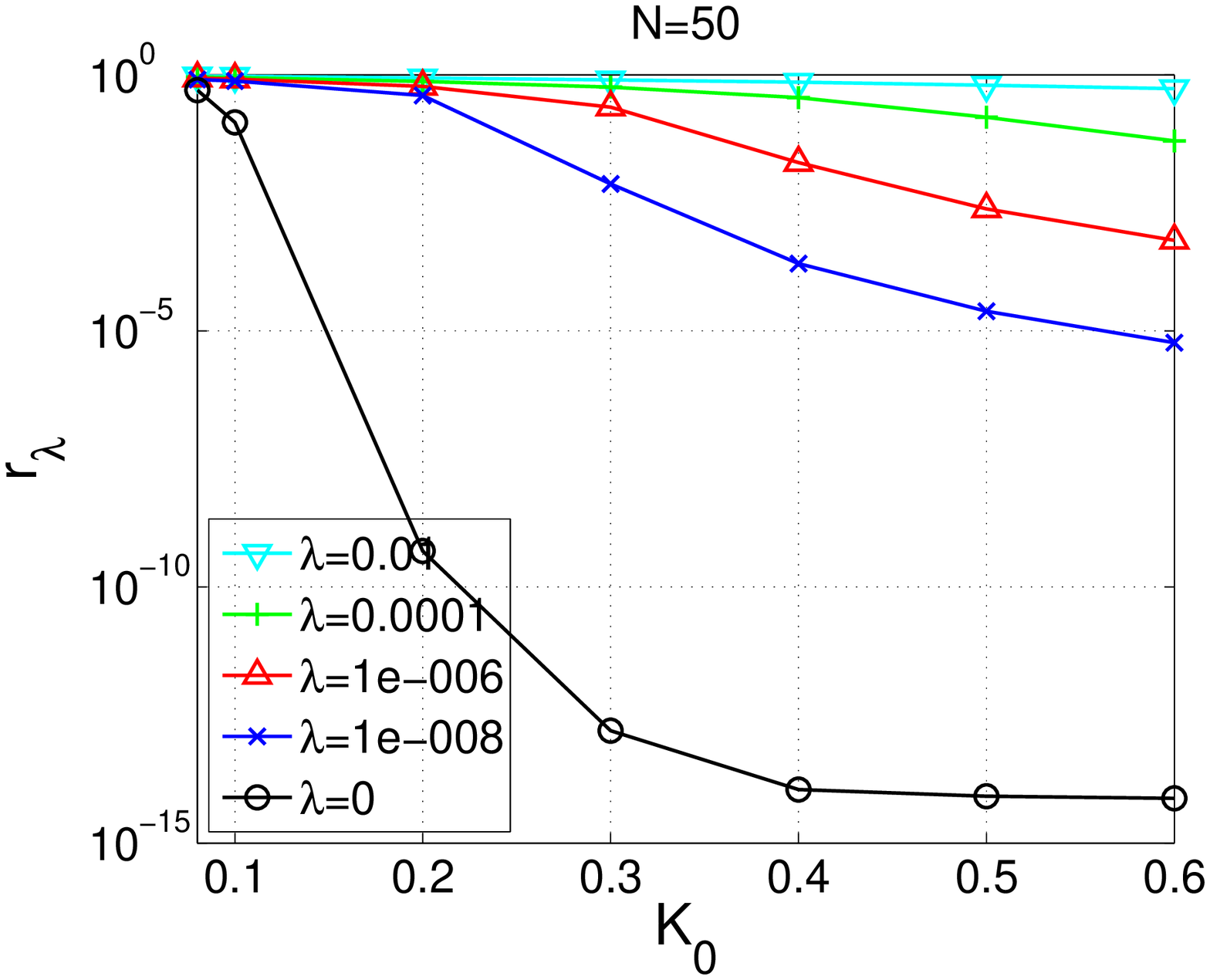} \\
(a) $\lambda=0$, $N\in\{10,\ldots,50\}$  & (b) $\lambda=\{0,\ldots,0.01\}$, $N=50$
\end{tabular}
\caption{\label{Fig:usv_Error_vs_K_L1} (a) Relative residual ($r_\lambda$) versus $K_0$
for various values of $N$, and (b) relative residual ($r_\lambda$) versus $K_0$
for various values of the regularization parameter $\lambda$, for the $2^N\times 2^N$ randomly generated matrices with prescribed singular values. 
}
\end{figure}

\subsection{Example 3: Laplace Operator}

The 1-D discrete Laplace operator of size $2^N\times 2^N$
with Dirichlet-Dirichlet boundary condition is considered \cite{Kaz2012}:
	$$
	\BF{A} = \text{tridiag}(-1,2,-1) \in\BB{R}^{2^N\times 2^N}.
	$$
The matrix $\BF{A}$ is square and symmetric, and its explicit TT representation and TT-ranks were already investigated and presented in \cite{Kaz2012}. The TT-ranks of the inverse $\BF{A}^{-1}$ are also known to be $(R_1,R_2,\ldots,R_{N-1})=(4,5,5,\ldots,5,4)$ \cite{Kaz2012}.

Figure~\ref{Fig:Laplace_Error_vs_it_L} illustrates the convergence
of the MALS algorithm for various values of $\lambda\in\{0, 10^{-6}, 10^{-4},
10^{-2}\}$ and a fixed $N=60$. The relative residual decreased monotonically, but the convergence was slow when the $\lambda$ value was small in the range $\lambda\in\{0,10^{-6}\}$. In Figure~\ref{Fig:Laplace_Error_vs_it_L}(b), for the small $\lambda$ values, 
the maximum of the estimated TT-ranks increased largely during the iteration process, whereas for the larger $\lambda$ values ($\lambda=10^{-4},10^{-2}$), the TT-ranks were relatively small. 
Note that the values $\lambda=10^{-4}, 10^{-2}$ are still relatively small compared to the diagonal entries of $\BF{AA}^\RM{T}$ in \eqref{expr:sol_global_regul}, which are 5 or 6. We can conclude that the regularization with a $\lambda>0$ value is necessary to obtain low-rank approximate pseudoinverses of ill-conditioned large-scale matrices.

\begin{figure}
\centering
\begin{tabular}{cc}
\includegraphics[width=6cm]{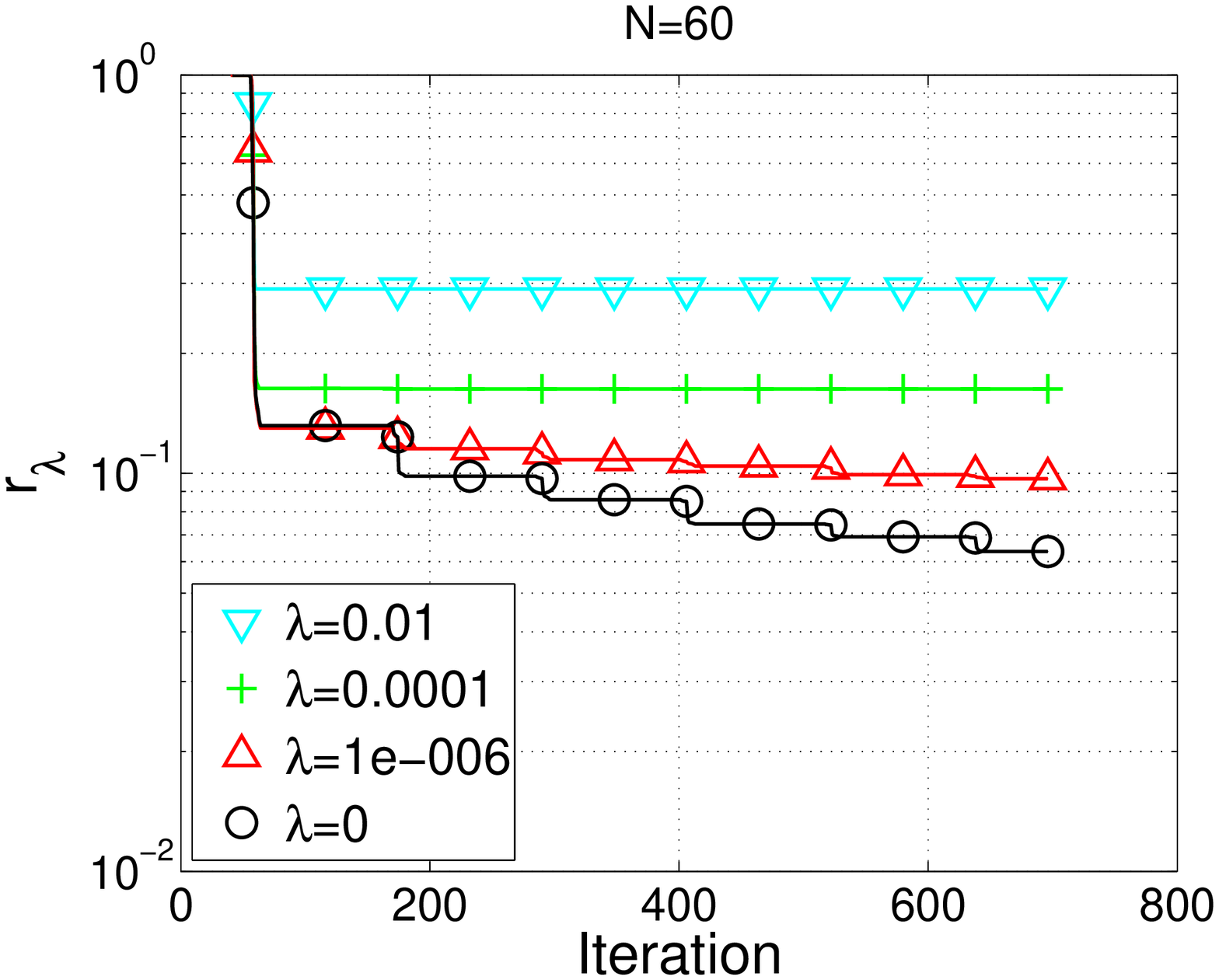} &
\includegraphics[width=6cm]{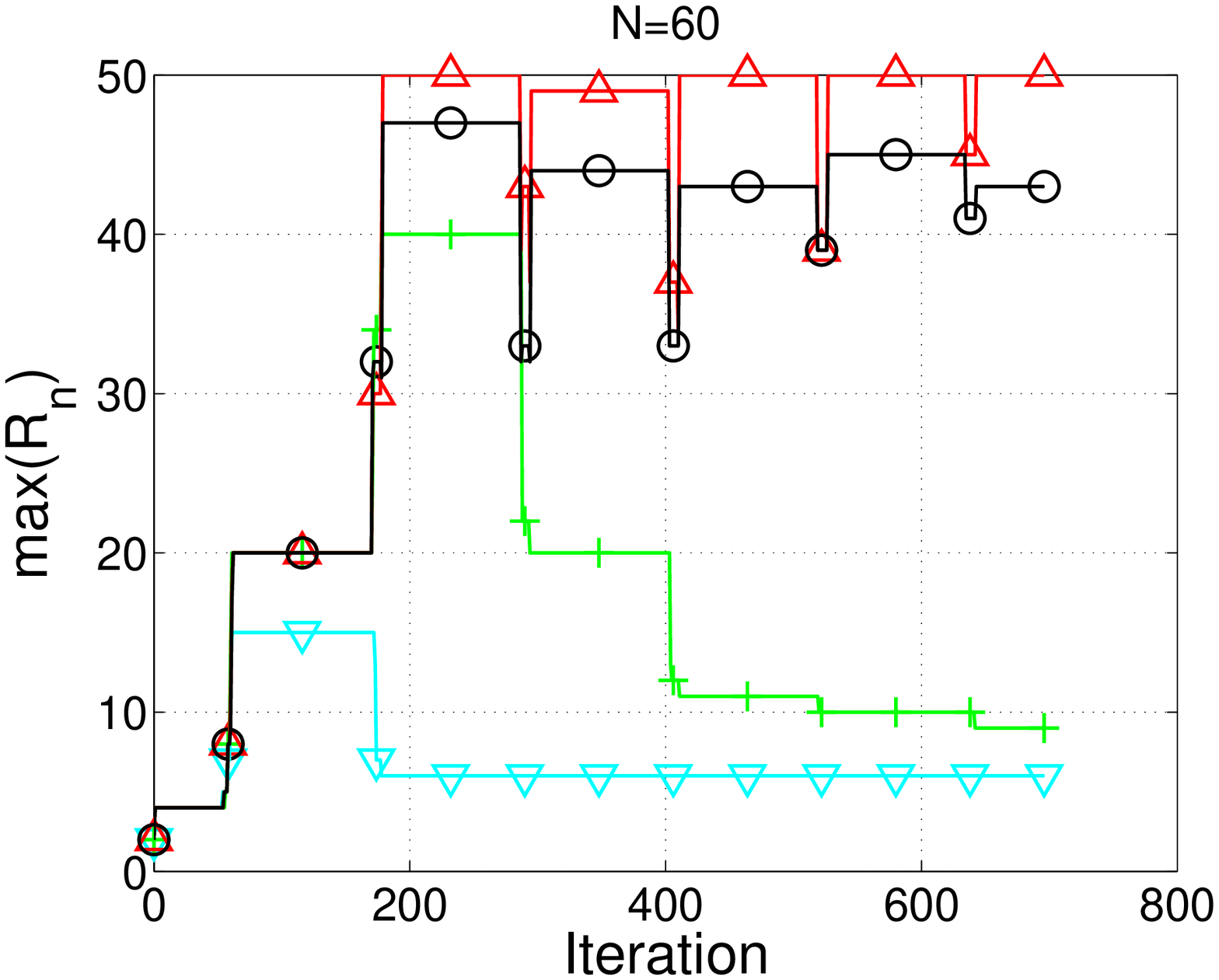} \\
(a) Relative residual & (b) Maximum of TT-ranks
\end{tabular}
\caption{\label{Fig:Laplace_Error_vs_it_L}
Convergence of the proposed MALS algorithm
illustrated by (a) relative residual ($r_\lambda$) and (b) maximum of the TT-ranks of the estimated pseudoinverses, for various values of the regularization parameter $\lambda$ for the $2^{60}\times 2^{60}$ discrete Laplace operator ($N=60$). The markers on the lines indicate half-sweeps, i.e., every $N-2$ iterations.
}
\end{figure}



The computational costs for the estimation of the pseudoinverses are illustrated in Figure~\ref{Fig:Laplace_Time_N} together with the estimated maximum TT-ranks. 
It is illustrated that the computational time increased logarithmically with the matrix size $2^N\times 2^N$, and the estimated TT-ranks were bounded over all $N$ values. Moreover, the proposed MALS algorithm needs shorter computational time than the standard MALS algorithm.

\begin{figure}
\centering
\begin{tabular}{cc}
\includegraphics[width=6cm]{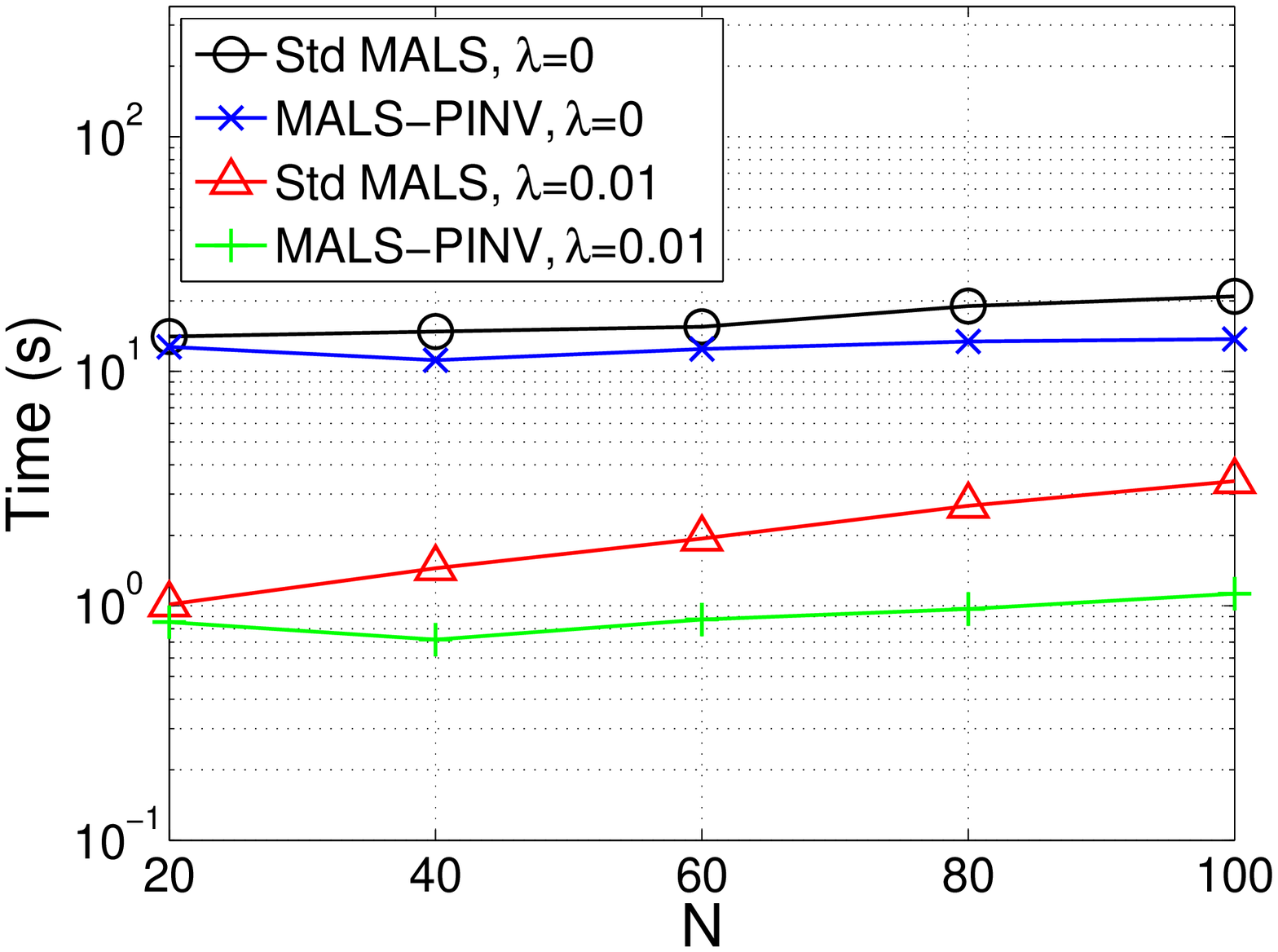} &
\includegraphics[width=6cm]{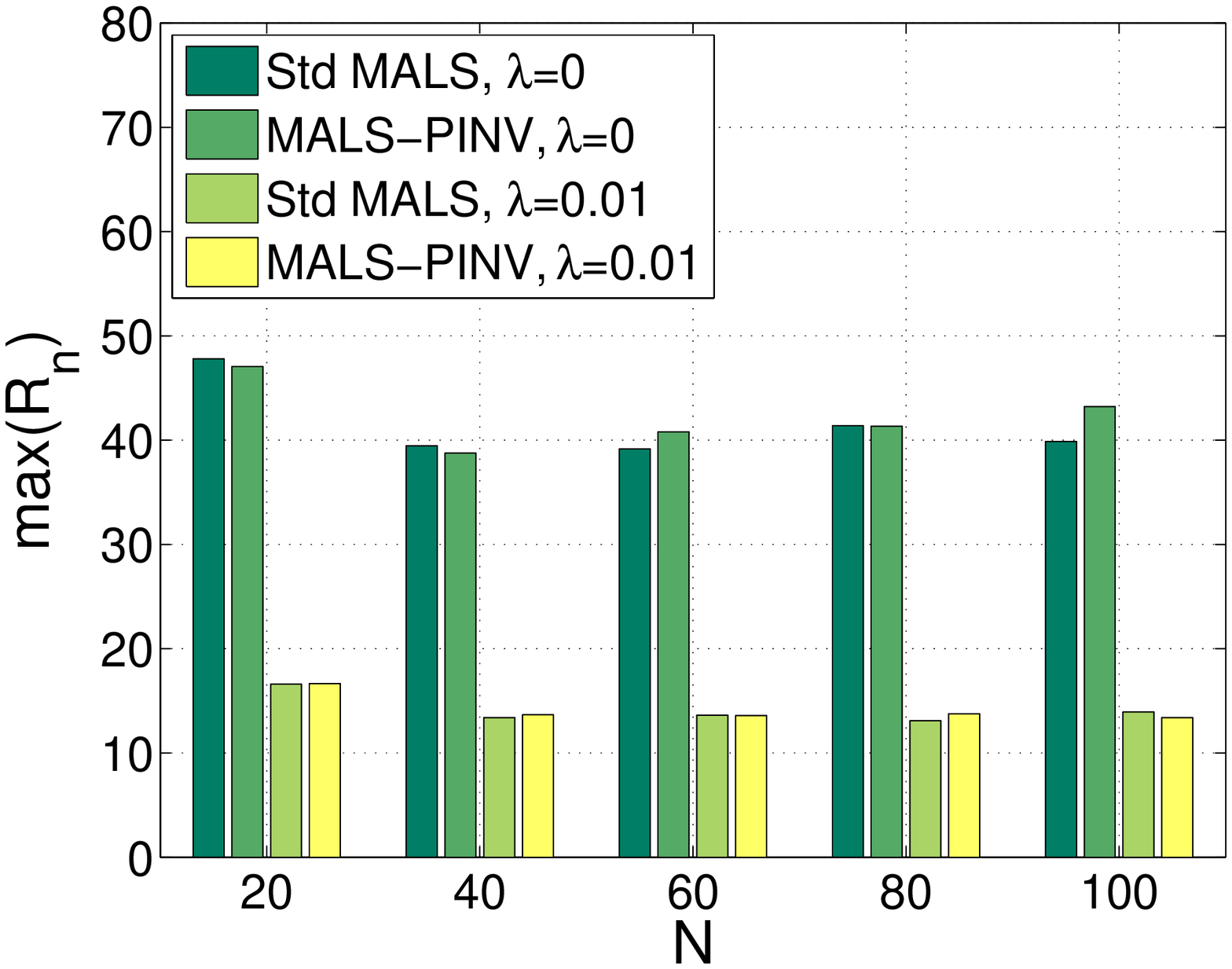} \\
(a) Computation time & (b) Maximum of TT-ranks
\end{tabular}
\caption{\label{Fig:Laplace_Time_N}(a) Computation time and
(b) maximum of the TT-ranks of the estimated pseudoinverse, for the $2^{N}\times 2^{N}$ discrete Laplace operators
for various values of the regularization parameter $\lambda\in\{0,10^{-2}\}$.
Std MALS means the standard MALS method applied for solving the linear system \eqref{std_mals_equation}, and MALS-PINV means the proposed MALS method.  
}
\end{figure}

\subsection{Example 4: Convection-Diffusion Equation}
\label{sec-simul-convection}

In order to demonstrate the effectiveness of the proposed algorithm
in preconditioning nonsymmetric systems of linear equations,
we consider the 3-D convection-diffusion equation on the unit cube $[0,1]^3$ described in \cite{Sonne2008}:
	\begin{equation} \label{con_diff_eqn}
	u_{xx} + u_{yy} + u_{zz} + cu_x = f,
	\end{equation}
where the function $f$ is defined by the solution
	\begin{equation} 
	u(x,y,z) = \exp(xyz)\sin(\pi x)\sin(\pi y)\sin(\pi z). 
	\end{equation}
By finite difference discretization, each axis is discretized by
$2^M+2$ grid points including the boundary points. The number of
equations is $2^M \cdot 2^M \cdot 2^M = 2^{3M} = 2^N$ with $N=3M$.
We set $c=2^{N-10}$. It is noted in \cite{Sonne2008} that the matrix $\BF{A}$ is a strongly nonsymmetric matrix, and it has eigenvalues with large imaginary parts, which slow the convergence of conjugate gradient-type algorithms such as Bi-CGSTAB.

First, the regularized inverses of the coefficient matrix were estimated by the proposed algorithm. Figure~\ref{Fig:Convection_convergence_D} illustrates the convergence of the proposed algorithm for three truncation parameter values: $\delta=2\cdot 10^{-3},2\cdot 10^{-5},$ and $2\cdot 10^{-7}$. 
We can see that smaller $\delta$ values result in a faster convergence per each iteration, but the TT-ranks also increase faster, which may cause high computational costs. On the other hand, a too large value of $\delta$ may result in large relative residuals. 
Since the TT-ranks of the estimated pseudoinverse influence the computational cost in the next step of solving a preconditioned system of linear equations, it is important to balance the approximation accuracy and
TT-ranks of the estimated pseudoinverse.


\begin{figure}
\centering
\begin{tabular}{cc}
\includegraphics[width=6cm]{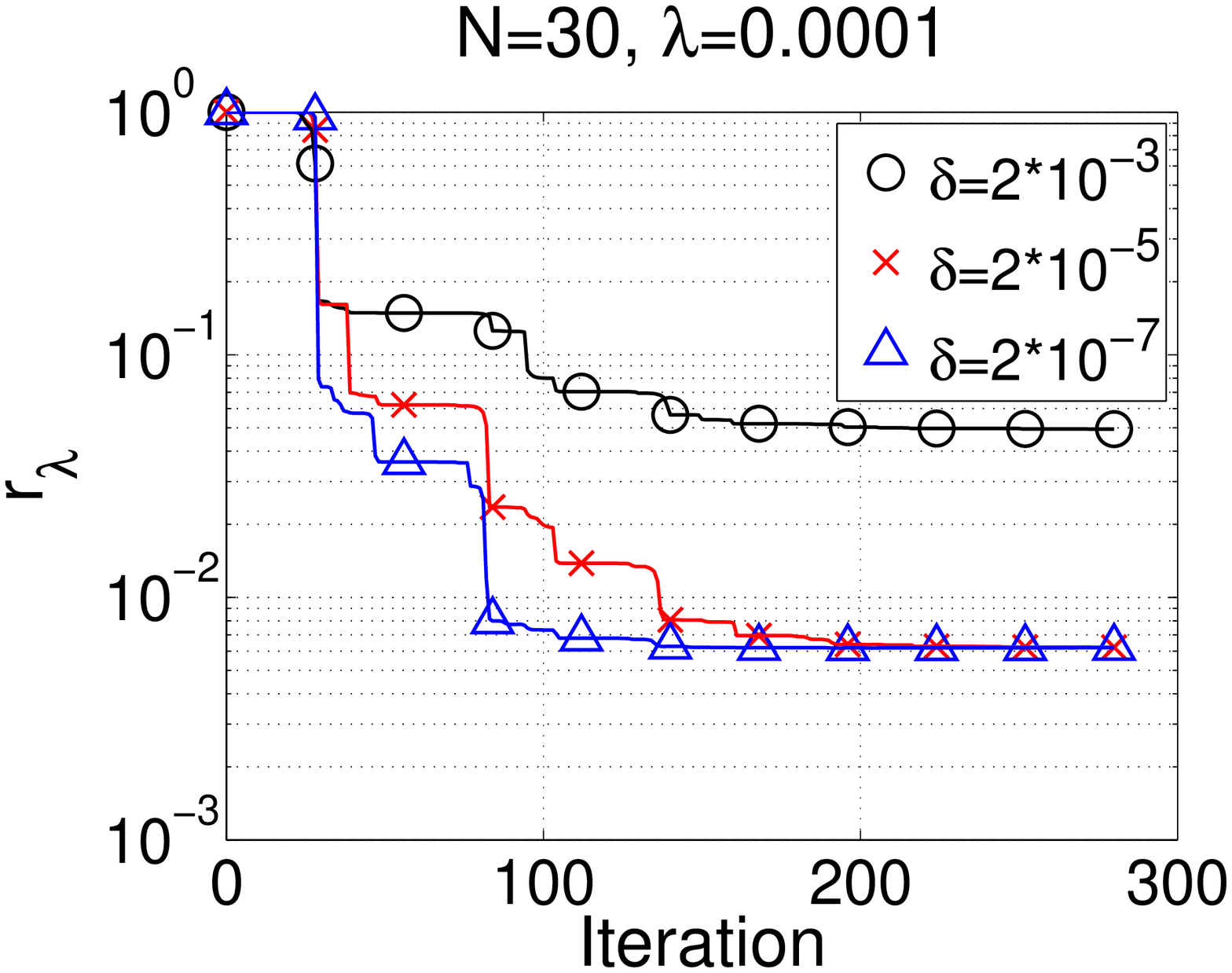} &
\includegraphics[width=6cm]{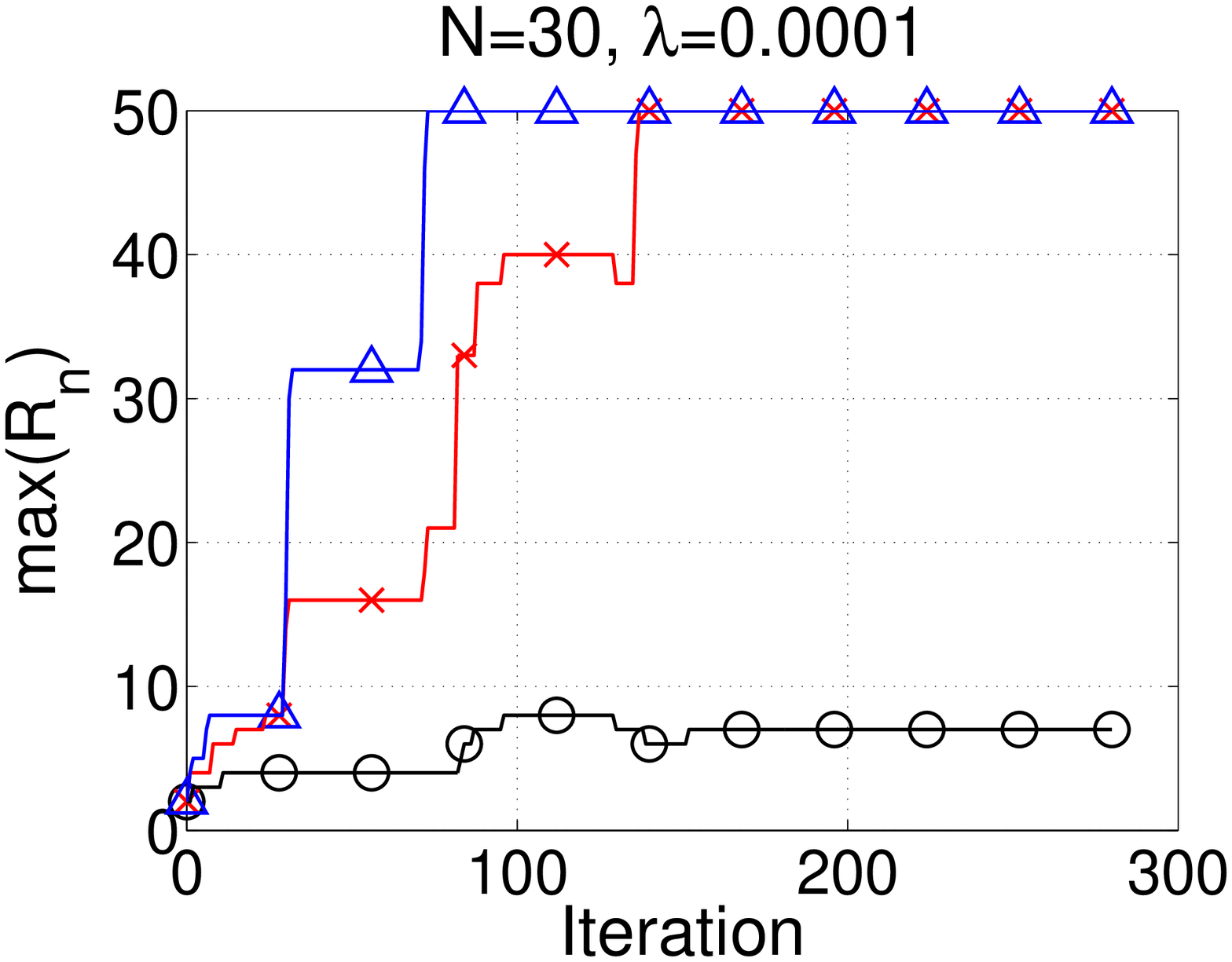}\\
(a) Relative residual & (b) Maximum of TT-ranks
\end{tabular}
\caption{\label{Fig:Convection_convergence_D}
Convergence of the proposed MALS algorithm for various values of the truncation parameter $\delta = 2\cdot 10^{-3}, 2\cdot 10^{-5}, 2\cdot 10^{-7}$
for the discretized convection-diffusion equation with
$N=30$ and $\lambda=10^{-4}$.
(a) Relative residual and (b) maximum of the TT-rank of the
estimated pseudoinverse. 
The markers on the lines indicate half-sweeps, i.e., every $N-2$ iterations.
}
\end{figure}

Figure~\ref{Fig_Convection_Preconditioner} illustrates the computation time for the estimation of the regularized pseudoinverses, for various values of $N$ and regularization parameter $\lambda$, when the truncation parameter was set at $\delta=10^{-4}(N-1)^{-1/2}$. Large $\lambda$ value, e.g.,  $\lambda=0.01$, resulted in relatively small TT-ranks and short computation time for $N=15,90$. 

\begin{figure}
\centering
\begin{tabular}{ccc}
\includegraphics[width=3.82cm]{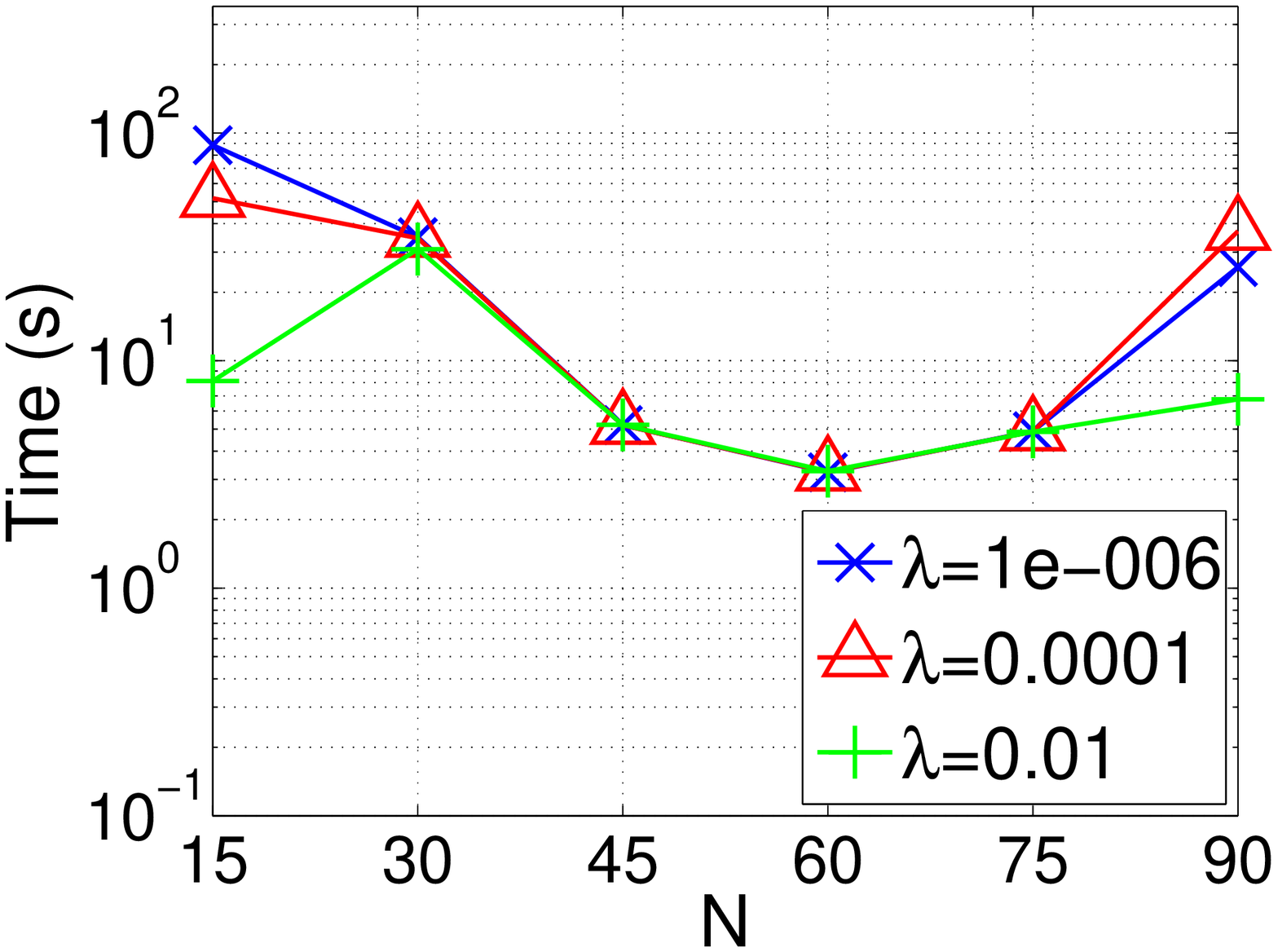} &
\includegraphics[width=3.82cm]{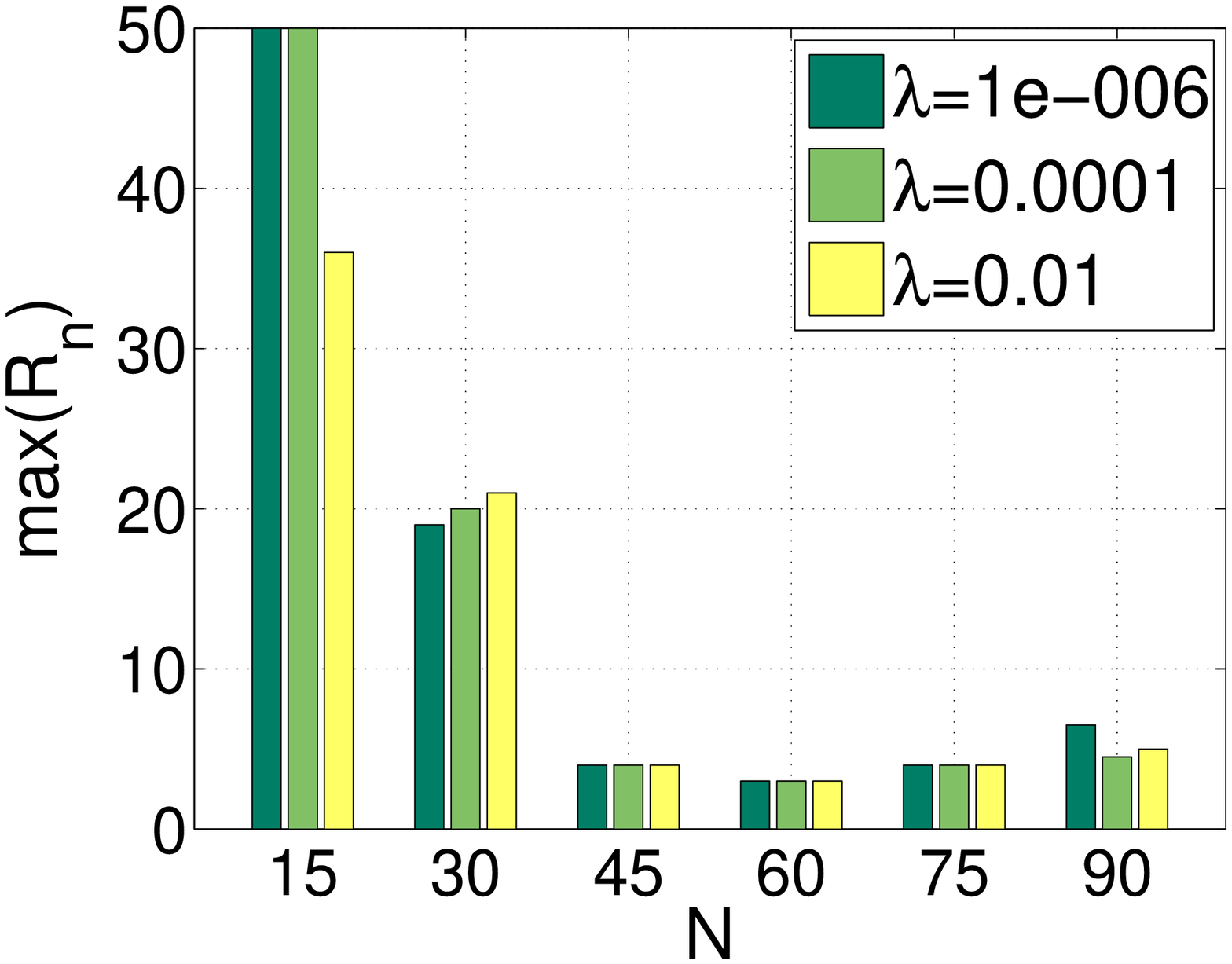} & 
\includegraphics[width=3.82cm]{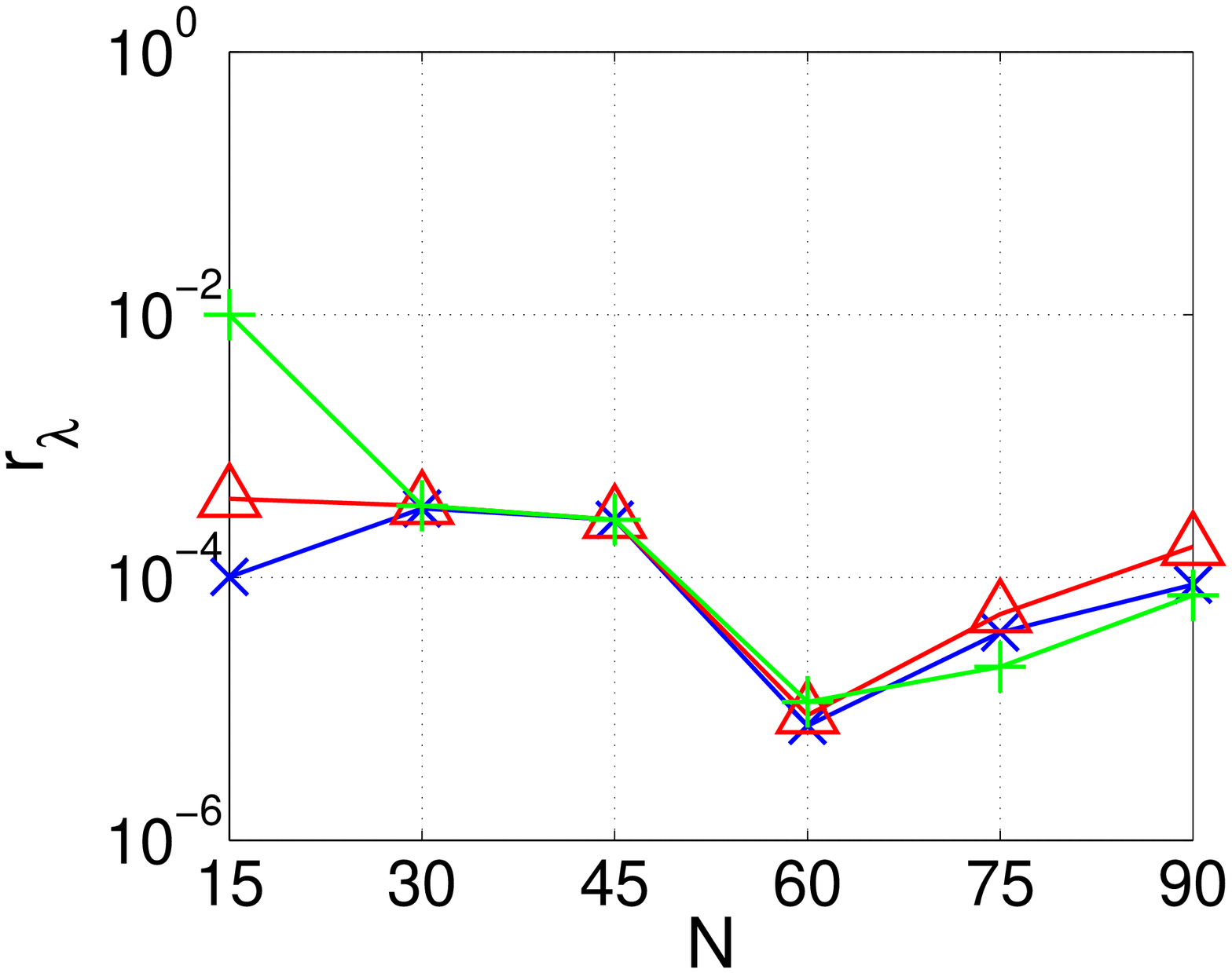} \\
(a) Computation time & (b) Maximum of TT-ranks & (c) Relative residual
\end{tabular}
\caption{\label{Fig_Convection_Preconditioner}(a) Computation time, (b) maximum of the estimated TT-ranks, and (c) relative residuals obtained by the proposed MALS algorithm for the $2^N\times 2^N$ coefficient matrix of the convection-diffusion equation for various values of the regularization parameter
$\lambda=10^{-6},10^{-4},10^{-2}$.
}
\end{figure}

Next, we computed solutions to the convection-diffusion equation
\eqref{con_diff_eqn} numerically by using the function
\texttt{dmrg\_solve2} \cite{OseDol2012} in TT-Toolbox \cite{Ose2014},
where the linear equation was either preconditioned or not
by the estimated regularized pseudoinverse. 
For solving local optimization problems in the \texttt{dmrg\_solve2}, we used one of the three different Matlab functions, \texttt{gmres}, \texttt{bicgstab}, and \texttt{pcg}, in order to compare
them for solving nonsymmetric systems of linear equations. However, we found almost no differences in performances between them in this simulation, so, we only presented the results of the \texttt{bicgstab}. 
The \texttt{dmrg\_solve2} algorithm converged to the relative residual
tolerance of $10^{-4}$ within 20 full-sweeps mostly (one full-sweep is
equivalent to solving the local problems $2(N-2)$ times) for $15\leq N\leq 90$.  
The computational costs for solving the linear systems are illustrated in Figures~\ref{Fig:Convection_Linear}(a) and (b), where the preconditioned systems were solved much faster than without preconditioning. 
The estimated TT-ranks of the solution $\BF{x}$ ranged between 10 and 15 and were almost constant as $N$ increased (although not presented here), implying that the computational costs were affected mostly by the convergence of the local algorithm.
In addition, it should be noted that the computational costs for solving the square linear systems as illustrated in Figure~\ref{Fig:Convection_Linear} were lower than the costs for the preconditioner computation in Figure~\ref{Fig_Convection_Preconditioner}(a). 

\begin{figure}
\centering
\begin{tabular}{cc}
\includegraphics[width=6cm]{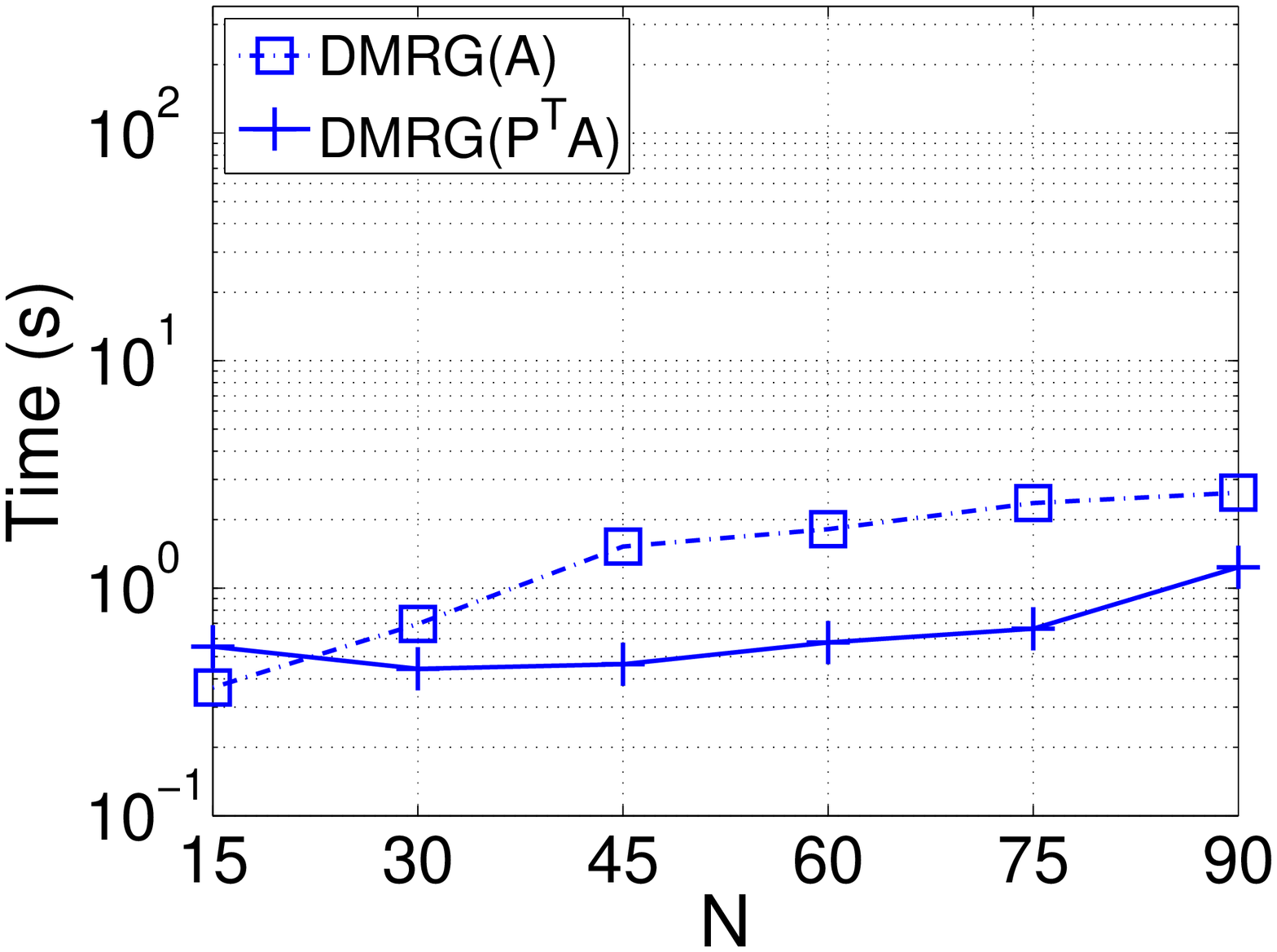}& 
\includegraphics[width=6cm]{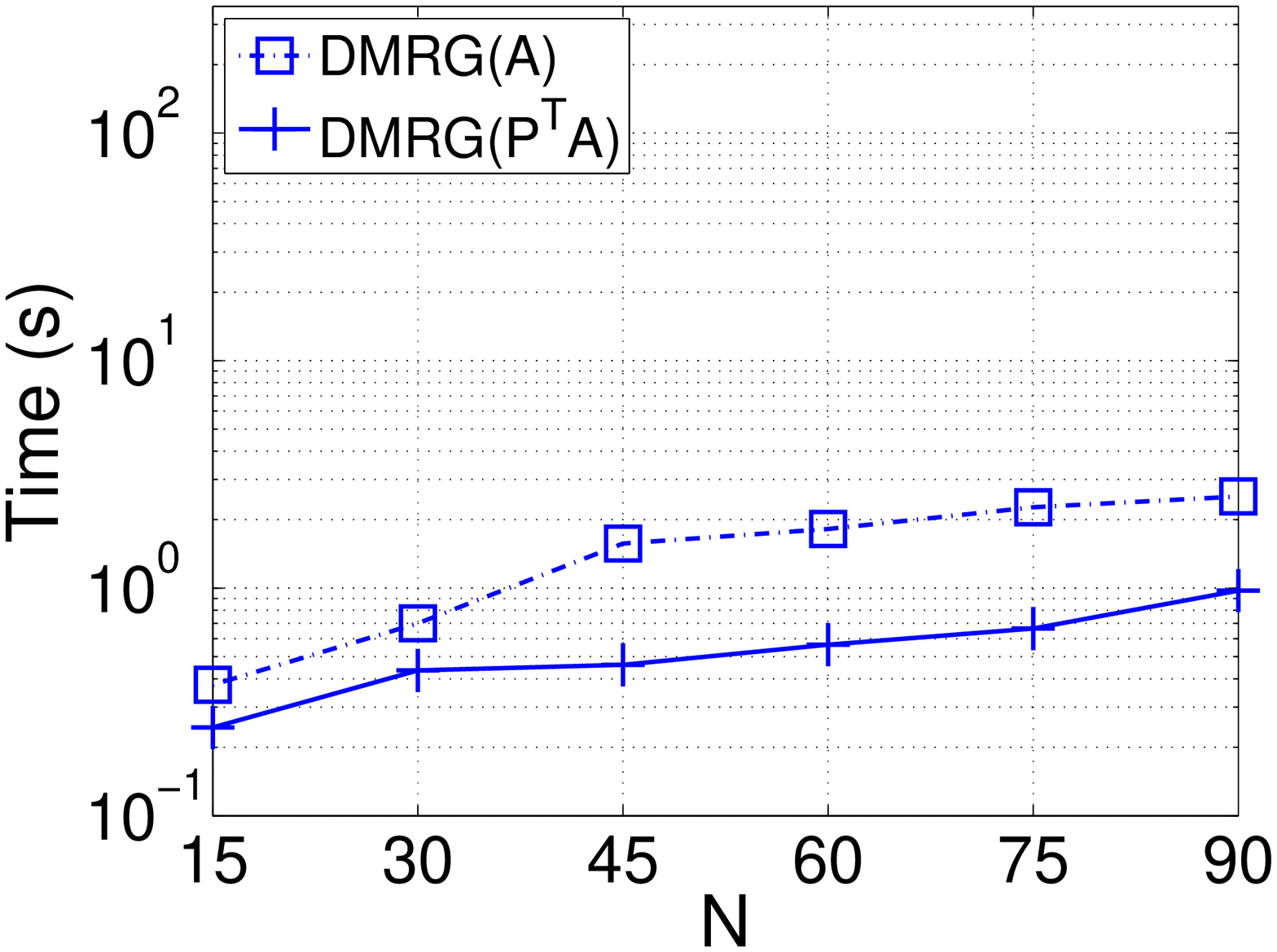}
\\
(a) $\lambda=10^{0}$  & (b) $\lambda=10^{-6}$
\end{tabular}
\caption{\label{Fig:Convection_Linear}
Comparison of the performances of the numerical solution algorithm
(\texttt{dmrg\_solve2} \cite{Ose2014,OseDol2012}) for solving
of the linear systems without preconditioning (DMRG(A)) and
preconditioned systems (DMRG($P^TA$)) described in
Section~\ref{sec-simul-convection}. 
For local optimizations for \texttt{dmrg\_solve2}, the Matlab function \texttt{bicgstab} was applied for the cases that (a) $\lambda=10^{0}$ and (b) $\lambda=10^{-6}$. 
}
\end{figure}

\section{Conclusion and Discussion}

We presented a new MALS algorithm for the computation
of approximate pseudoinverses of extremely large-scale structured matrices
using low-rank TT decomposition.
The proposed method can estimate the Moore-Penrose pseudoinverses
of any nonsymmetric or nonsquare structured matrices in low-rank
matrix TT format approximately, so it can be useful for preconditioning
overdetermined or underdetermined large-scale
systems of linear equations.

The proposed method provides stability and the fast convergence speed even for
very ill-conditioned large-scale matrices by regularization. The regularized
solutions were shown to have
relatively small TT-ranks in the numerical simulations, so the
computational costs for the construction of preconditioners and
solution to huge systems  of linear equations were
significantly smaller than without regularization.
The regularization technique is especially important when
the size of a data matrix is huge and ill-conditioned.

The proposed algorithm converts the large-scale minimization
problem into sequential smaller-scale optimization problems
to which any standard optimization methods can be applied. The
convergence to the desired solution is stable and relatively fast
because the TT-ranks of the approximate inverses can be adaptively
determined during the iteration process, and the decrease in the objective function value is monotonic. The computational cost
for running the proposed MALS algorithm is logarithmic in the
matrix size under the assumption of boundedness of TT-ranks.

The estimated pseudoinverses were applied to preconditioning of the
strongly nonsymmetric matrix occurring in systems of linear equations in the
convection-diffusion equation problem.
Several standard iterative algorithms such as GMRES, Bi-CGSTAB, 
and PCG showed the improved convergence in the numerical 
simulations, which demonstrate the effectiveness of the proposed
 algorithm by preconditioning and symmetrizing the coefficient matrix.

The main advantage of the proposed method lies in its applicability to
any rectangular huge structured matrices. Moreover, the regularization technique
employed in the proposed method helps to compute approximate
pseudoinverses reliably for any ill-conditioned structured matrices
which admit low-rank TT approximations.

The developed algorithm can further be directly applied to the
following areas. First, the computation of regularized Moore-Penrose
pseudoinverses is closely related to the regularized (filtered) solution of
systems of linear equations $\BF{Ax}=\BF{b}$, by
$\hat{\BF{x}} = {\BF{P}}_\lambda^{*\RM{T}} \BF{b}$, where
${\BF{P}}_\lambda^{*\RM{T}} \approx \BF{A}^\dagger$
\cite{Chung2014inv,Chung2015laa}.
Second, the large scale generalized eigenvalue decomposition (GEVD) problem
described by $\BF{Ax} = \lambda \BF{Bx}$ for a square matrix
$\BF{A}$ and a nonsingular square matrix $\BF{B}$
\cite{Golub1996matrix} can be transformed to a standard eigenvalue
decomposition problem $\BF{B}^{-1}\BF{Ax} = \lambda \BF{x}$
if the large-scale inverse matrix $\BF{B}^{-1}$ can be approximately
computed in TT format efficiently. Once the large-scale matrices
$\BF{A}$ and $\BF{B}^{-1}$ are represented in TT format, the
multiplication $\BF{B}^{-1}\BF{A}$ can be relatively easily performed
\cite{Ose2011}. Third, a special case of the optimization problem
\eqref{optim:general:global:TT} arises in important subspace clustering
problems \cite{Yu2011UAI_rank}, which can also be efficiently solved
using the proposed algorithm based on TT decompositions.


\begin{thebibliography}{00}



\bibitem{Andreev2014multilevel}  
{\sc R. Andreev and C. Tobler},
{\em Multilevel preconditioning and low-rank tensor iteration for space-time
simultaneous discretizations of parabolic PDEs},
Numer. Linear Algebra Appl., 22 (2015), pp.~317--337.

\bibitem{Barata12}
{\sc J.~C.~A. Barata and M.~S. Hussein},
{\em The Moore-Penrose pseudoinverse: A tutorial review of the theory},
Braz. J. Phys., 42 (2012), pp.~146--165.

\bibitem{Barrett94template}
{\sc R. Barrett, M. Berry, T.~F. Chan, J. Demmel, J.~M. Donato,
J. Dongarra, V. Eijkhout, R. Pozo, C. Romine, and H. Van der Vorst},
Templates for the Solution of Linear Systems: Building Blocks for
Iterative Methods, 2nd Edition,
SIAM, Philadelphia, 1994.

\bibitem{Braess2005}
{\sc D. Braess and W. Hackbusch},  
{\em Approximation of $1/x$ by exponential sums in $[1,\infty)$},  
IMA J. Numer. Anal., 25 (2005), pp.~685--697. 


\bibitem{Benzi1999}
{\sc M. Benzi and M. T\r{u}ma},
{\em A comparative study of sparse approximate inverse
preconditioners},
Appl. Numer. Math., 30 (1999), pp.~305--340.

\bibitem{Bey2002}
{\sc G. Beylkin and M.~J. Mohlenkamp},
{\em Numerical operator calculus in higher dimensions},
Proc. Natl. Acad. Sci. USA, 99 (2002), pp.~10246--10251.


\bibitem{Chung2014inv}
{\sc J. Chung and M. Chung},
{\em An efficient approach for computing optimal
low-rank regularized inverse matrices},
Inverse Problems, 30 (2014), 114009.

\bibitem{Chung2015laa}
{\sc J. Chung, M. Chung, and D.~P. O'Leary},
{\em Optimal regularized low rank inverse approximation},
Linear Algebra Appl., 468 (2015), pp.~260--269.

\bibitem{Cic2014a}
{\sc A. Cichocki},
{\em Era of big data processing: A new approach via tensor networks and
tensor decompositions}, arXiv:1403.2048, 2014.

\bibitem{Cic2014TN-b} 
{\sc A. Cichocki},
{\em Tensor networks for big data analytics and large-scale optimization problems}, 
arXiv:1407.3124, 2014.

\bibitem{Cic2009}
{\sc A. Cichocki, R. Zdunek, A. H. Phan, and S. Amari},
Nonnegative Matrix and Tensor Factorizations: Applications to
Exploratory Multi-way Data Analysis and Blind Source Separation,
Wiley, Chichester, 2009.

\bibitem{Cui2009}
{\sc X. Cui and K. Hayami},
{\em Generalized approximate inverse preconditioners
for least squares problems},
Japan J. Indust. Appl. Math., 26 (2009), pp.~1--14.


\bibitem{Davis79}
{\sc P.~J. Davis}, 
Circulant matrices, 
Wiley, New York, 1979. 


\bibitem{DeLath2000}
{\sc L. De Lathauwer, B. De Moor, and J. Vandewalle},
{\em A multilinear singular value decomposition},
SIAM J. Matrix Anal. Appl., 21 (2000), pp.~1253--1278.



\bibitem{Dol2013gmres}
{\sc S.~V. Dolgov},
{\em TT-GMRES: Solution to a linear system in the
structured tensor format},
Russian J. Numer. Anal. Math. Model., 28 (2013),
pp.~149--172.

\bibitem{DolKhoSav2012}
{\sc S. Dolgov, B. Khoromskij, D. Savostyanov}, 
{\em Superfast Fourier transform using QTT approximation}, 
J. Fourier Anal. Appl., 18 (2012), pp.~915--953. 

\bibitem{DolSav2014}
{\sc S.~V. Dolgov and D.~V. Savostyanov},
{\em Alternating minimal energy methods for linear
systems in higher dimensions},
SIAM J. Sci. Comput., 36 (2014), pp.~A2248--A2271.



%

\bibitem{Espig2011}    
{\sc M. Espig, W. Hackbusch, S. Handschuh, and R. Schneider},
{\em Optimization problems in contracted tensor networks},
Comput. Vis. Sci., 14 (2011), pp.~271--285.

\bibitem{FalHac2012}     
{\sc A. Falc\'o and W. Hackbusch},
{\em On minimal subspaces in tensor representations},
Found. Comput. Math., 12 (2012), pp.~765--803.

\bibitem{Giraldi2014low}    
{\sc L. Giraldi, A. Nouy, and G. Legrain},
{\em Low-rank approximate inverse for preconditioning
tensor-structured linear systems},
SIAM J. Sci. Comput., 36 (2014), pp.~A1850--A1870.

\bibitem{Golub1996matrix}
{\sc G.~H. Golub and C.~F. Van Loan},
Matrix Computations, Third Edition,
Johns Hopkins University Press, Baltimore, 1996.

\bibitem{Gra2010}
{\sc L. Grasedyck},
{\em Hierarchical singular value decomposition of tensors},
SIAM J. Matrix Anal. Appl., 31 (2010),
pp.~2029--2054.

\bibitem{Gra2013}
{\sc L. Grasedyck, D. Kressner, and C. Tobler},
{\em A literature survey of low-rank tensor approximation techniques},
GAMM-Mitt., 36 (2013), pp.~53--78.

\bibitem{Grote1997}
{\sc M.~J. Grote and T. Huckle},
{\em Parallel preconditioning with sparse approximate inverses},
SIAM J. Sci. Comput., 18 (1997), pp.~838--853.

\bibitem{Hac2012} 
{\sc W. Hackbusch},
Tensor Spaces and Numerical Tensor Calculus,
Springer, Berlin, 2012.



\bibitem{Hac2009}
{\sc W. Hackbusch and S. K\"{u}hn},
{\em A new scheme for the tensor representation}.
J. Fourier Anal. Appl., 15 (2009),
pp.~706--722.


\bibitem{Holtz2011}
{\sc S. Holtz, T. Rohwedder, and R. Schneider},
{\em On manifolds of tensors of fixed TT-rank},
Numer. Math., 120 (2012), pp.~701--731.

\bibitem{Holtz2012}
{\sc S. Holtz, T. Rohwedder, and R. Schneider},
{\em The alternating linear scheme for tensor optimization in the
tensor train format},
SIAM J. Sci. Comput., 34 (2012), pp.~A683--A713.


\bibitem{kazeev2014direct}    
{\sc V. Kazeev,  M. Khammash, M. Nip, and C. Schwab},
{\em Direct solution of the chemical master equation using quantized tensor trains},
PLoS Comput. Biol., 10 (2014), e1003359.
doi:10.1371/journal.pcbi.1003359

\bibitem{Kaz2012}
{\sc V.~A. Kazeev and B.~N. Khoromskij},
{\em Low-rank explicit QTT representation of the Laplace
operator and its inverse},
SIAM J. Matrix Anal. Appl., 33 (2012), pp.~742--758.



\bibitem{Kho2009}
{\sc B.~N. Khoromskij},
{\em Tensor-structured preconditioners and
approximate inverse of elliptic operators in $\BB{R}^d$},
Constr. Approx., 30 (2009), pp.~599--620.

\bibitem{Kho2011}
{\sc B.~N. Khoromskij},
{\em $O(d logN)$-quantics approximation of $N$-$d$ tensors
in high-dimensional numerical modeling},
Constr. Approx., 34 (2011), pp.~257--280.

\bibitem{Kho2012}
{\sc B.~N. Khoromskij},
{\em Tensors-structured numerical methods in
scientific computing: Survey on recent advances},
Chemometr. Intell. Lab. Syst., 110 (2012), pp.~1--19.


\bibitem{KolBa2009}
{\sc T.~G. Kolda and B.~W. Bader},
{\em Tensor decompositions and applications},
SIAM Rev., 51 (2009), pp.~455--500.

\bibitem{KresSteinUsh2014}
{\sc D. Kressner, M. Steinlechner, and A. Uschmajew},
{\em Low-rank tensor methods with subspace correction
for symmetric eigenvalue problems},
SIAM J. Sci. Comput., 36 (2014), pp.~A2346--A2368.

\bibitem{KresTob2011}
{\sc D. Kressner and C. Tobler}, 
{\em Preconditioned low-rank methods for high-dimensional elliptic PDE eigenvalue problems}, 
Comput. Methods Appl. Math., 11 (2011), pp.~363--381. 


\bibitem{Lass1995}
{\sc J.~B. Lasserre}, 
{\em A trace inequality for matrix product}, 
IEEE Trans. Autom. Control, 40 (1995), pp.~1500--1501.


\bibitem{Lee2014}
{\sc N. Lee and A. Cichocki},
{\em Fundamental tensor operations for large-scale data analysis in
tensor train formats},
arXiv:1405.7786, 2014.

\bibitem{Lee2015svd}
{\sc N. Lee and A. Cichocki},
{\em Estimating a few extreme singular values and vectors for large-scale
matrices in tensor train format},
SIAM J. Matrix Anal. Appl., 36 (2015), pp.~994--1014.

\bibitem{Mirsky75}
{\sc L. Mirsky}, 
{\em A trace inequality of John von Neumann}, 
Monatsh. Math., 79 (1975), 
pp.~303--306. 

\bibitem{Ose2011}
{\sc I.~V. Oseledets},
{\em Tensor-train decomposition},
SIAM J. Sci. Comput., 33 (2011), pp.~2295--2317.

\bibitem{Ose2014}
{\sc I.~V. Oseledets},
MATLAB TT-Toolbox, Version 2.3, June 2014,
https://github.com/oseledets/TT-Toolbox.

\bibitem{OseDol2012}
{\sc I.~V. Oseledets and S.~V. Dolgov},
{\em Solution of linear systems and matrix
inversion in the TT-format},
SIAM J. Sci. Comput., 34 (2012), pp.~A2718--A2739.

\bibitem{OseTyr2009}
{\sc I.~V. Oseledets and E.~E. Tyrtyshnikov},
{\em Breaking the curse of dimensionality, or how to
use SVD in many dimensions},
SIAM J. Sci. Comput., 31 (2009), pp.~3744--3759.

\bibitem{OseTyrZam2011}
{\sc I.~V. Oseledets, E.~E. Tyrtyshnikov, and N.~L. Zamarashkin},
{\em Tensor-train ranks of matrices and their inverses},
Comput. Meth. Appl. Math., 11 (2011), pp.~394--403.



\bibitem{Roh2013}    
{\sc T. Rohwedder and A. Uschmajew},
{\em On local convergence of alternating schemes for optimization
of convex problems in the tensor train format},
SIAM J. Numer. Anal., 51 (2013), pp.~1134--1162.

\bibitem{savostyanov2014exact}    
{\sc D.~V. Savostyanov, S.~V. Dolgov, J.~M. Werner, and I. Kuprov},
{\em Exact NMR simulation of protein-size spin systems using tensor train formalism},
Phys. Rev. B, 90 (2014), 085139.

\bibitem{signoretto2014learning}    
{\sc M. Signoretto, Q.~T. Dinh, L. De Lathauwer, and J.~A.~K. Suykens},
{\em Learning with tensors: A framework based on convex optimization and
spectral regularization}, Machine Learning, 94 (2014), pp.~303--351.

\bibitem{Sonne2008}
{\sc P. Sonneveld and M.~B. van Gijzen},
{\em IDR(s): A family of simple and fast algorithms for solving large nonsymmetric linear systems},
SIAM J. Sci. Comput., 31 (2008), pp.~1035--1062.

\bibitem{sorber2014exact}   
{\sc L. Sorber, I. Domanov, M. Van Barel, and L. De Lathauwer},
{\em Exact line and plane search for tensor optimization},
Comput. Optim. Appl. (2015), pp.~1--22. 
doi:10.1007/s10589-015-9761-5


\bibitem{USch2011}
{\sc U. Schollw\"{o}ck},
{\em The density-matrix renormalization group in the age of matrix product states},
Ann. Phys., 326 (2011), pp.~96--192.

\bibitem{Verv2014}
{\sc N.~Vervliet, O.~Debals, L.~Sorber, and L.~De Lathauwer},
{\em Breaking the curse of dimensionality using decompositions of
incomplete tensors: Tensor-based scientific computing in big data
analysis}, IEEE Signal Process. Mag., 31 (2014), pp.~71--79.

\bibitem{Whi1993}   
{\sc S.~R. White},
{\em Density-matrix algorithms for quantum renormalization groups},
Phys. Rev. B., 48 (1993), pp.~10345--10356.

\bibitem{Yu2011UAI_rank}
{\sc Y.-L. Yu and D. Schuurmans},
{\em Rank/norm regularization with closed-form solutions:
Application to subspace clustering},
in Proceedings of the Twenty-Seventh Conference on Uncertainty in Artificial Intelligence,
F. Cozman and A. Pfeffer, eds., AUAI Press, Corvallis, Oregon, 2011,
pp.~778--785.


\end{thebibliography}
\end{document}